\documentclass[11pt,english]{amsart}

\usepackage{amsmath,leftidx,amsthm}
\usepackage{xspace}
\usepackage[psamsfonts]{amssymb}
\usepackage[latin1]{inputenc}
\usepackage{graphicx,color}
\usepackage[all]{xy}
\usepackage{thm-restate}

\usepackage{hyperref}

\theoremstyle{remark}
\newtheorem{example}{Example}[section]

\newtheorem{remark}{Remark}[section]

\theoremstyle{plain}
\newtheorem{theorem}{Theorem}[section]
\newtheorem{proposition}[theorem]{Proposition}
\newtheorem{definition}{Definition}[section]
\newtheorem{Theorem}{Theorem}

\newtheorem{lemma}[theorem]{Lemma}
\newtheorem*{keylemma}{Key Lemma}
\newtheorem{corollary}[theorem]{Corollary}
\newtheorem{conjecture}{Conjecture}
\newtheorem*{conjecture*}{ Conjecture}

\def\C{{\mathbb C}}
\def\R{{\mathbb R}}
\def\Z{{\mathbb Z}}

\def\Q{{\mathbb Q}}

\def\D{\mathbb{D}}

\def\p{\mathbb{P}}
\def\N{{\mathbb N}}

\def\cO{\mathcal{O}}

\def\cG{\mathcal{G}}

\def\cE{\mathcal{E}}
\def\cM{\mathcal{M}}
\def\cN{\mathcal{N}}
\def\cT{\mathcal{T}}

\def\cH{\mathcal{H}}

\def\fU{\mathfrak{U}}
\def\fV{\mathfrak{V}}

\def\fX{\mathfrak{X}}
\def\fY{\mathfrak{Y}}
\def\fL{\mathfrak{L}}
\def\fA{\mathfrak{A}}
\def\fB{\mathfrak{B}}

\def\fm{\mathfrak{m}}
\def\ff{\mathfrak{f}}

\def\om{\omega}
\def\la{\lambda}

\def\e{\varepsilon}

\DeclareMathOperator{\alg}{alg}
\DeclareMathOperator{\gal}{Gal}
\DeclareMathOperator{\sep}{sep}

\DeclareMathOperator{\id}{id}
\DeclareMathOperator{\Id}{Id}
\DeclareMathOperator{\Ind}{Ind}

\DeclareMathOperator{\an}{an}
\DeclareMathOperator{\vol}{Vol}

\DeclareMathOperator{\proj}{Proj}
\DeclareMathOperator{\con}{con}
\DeclareMathOperator{\zar}{zar}
\DeclareMathOperator{\Sk}{Sk}

\DeclareMathOperator{\hotimes}{\hat{\otimes}}

\DeclareMathOperator{\htop}{h_{top}}

\DeclareMathOperator{\NS}{NS}

\DeclareMathOperator{\red}{red}
\DeclareMathOperator{\self}{self}

\DeclareMathOperator{\prim}{\bar{P}}
\DeclareMathOperator{\res}{res}
\DeclareMathOperator{\ext}{ext}

\DeclareMathOperator{\spec} {Spec}

\def\and{{\quad\text{and}\quad}}

\begin{document}

\title[Entropy of rational maps]
{Topological entropy of a rational map over a complete metrized field}
\author{Charles Favre}
\address{CMLS, \'Ecole polytechnique, CNRS, Universit\'e Paris-Saclay, 91128 Palaiseau Cedex, France}
\email{charles.favre@polytechnique.edu}
\author{Tuyen Trung Truong}
\address{Department of Mathematics, University of Oslo, P.O. Box 1053, Blindern, 0316 Oslo, Norway}
\email{tuyentt@math.uio.no}

\author{Junyi Xie}
\address{BICMR, Peking University, Haidian District, Beijing 100871, China}
\email{xiejunyi@bicmr.pku.edu.cn}

\thanks{This project started when the first author was supported by the ERC-starting grant project "Nonarcomp" no.307856.
The first author also acknowledge the support of the National Science Foundation under Grant No. 1440140, while he was in residence at the Mathematical Sciences Research Institute in Berkeley, California. Second author is partially supported by Young Research Talents grant \#300814 by Research Council of Norway.
This project started when the third author was partially supported by project ``Fatou'' ANR-17-CE40-0002-01.}

\date{\today}

\begin{abstract}
We prove that the topological entropy of any dominant rational self-map of a projective variety defined over a complete  non-Archimedean field is
bounded from above by the maximum of its dynamical degrees, thereby extending a theorem of Gromov and Dinh-Sibony from the complex to the non-Archimedean setting. 
We proceed by proving that any regular self-map which admits a regular extension to a projective model defined over the valuation ring has necessarily zero entropy.
To this end we introduce the $\epsilon$-reduction of a Berkovich analytic space, a notion of independent interest.
\end{abstract}

\maketitle

\tableofcontents

\section*{Introduction}

%%%%%%%
The most basic dynamical invariant associated to a continuous self-map of a compact space is arguably its topological entropy,  first defined by Adler-Konheim-McAndrew in~\cite{AKMA}. In order to better understand the dynamics of maps of algebraic origin, we propose in this paper to 
estimate the topological entropy of rational maps of projective varieties defined over a complete metrized field. 

Let us first consider a dominant rational self-map $f \colon X \dashrightarrow X$ of a complex projective variety\footnote{that is a projective reduced and irreducible scheme over $\C$.} $X$ of dimension $d\ge1$.
The complex analytic variety $X^{\an}$ is then compact and metrizable, and even though $f$ may have indeterminacy points,
one may define its topological entropy ``\`a la Bowen''~\cite{bowen} by counting the exponential growth of the number of $(n,\e)$-separated well-defined orbits as $n\to \infty$.
Gromov~\cite{gromov} proved that $\htop(f) = \log (\deg(f))$ for any regular endomorphism of the projective space\footnote{in fact Gromov only proved the upper bound; the lower bound was proved earlier by Misiurewicz-Przytycki and a  far reaching generalization was then obtained by Yomdin, see the discussion on the next page.}; and Dinh-Sibony~\cite{DS05,DS08} partially extended this result
to any meromorphic self-map of a compact K\"ahler manifold (so in particular to any rational map of a complex projective manifold). Indeed, they proved:
\begin{equation}\label{eq:DShtop}
\htop(f) \le \max_{0\le i \le d} \{\log  \lambda_i(f)\}
\end{equation} 
where  $\lambda_i(f) \ge 1$ are the dynamical degrees of $f$ (we discuss below their definitions in the context of projective varieties). 

Now let $(k,|\cdot|)$ be a complete metrized non-Archimedean field, and 
let $f \colon X \dashrightarrow X$ be a dominant rational self-map of a projective variety $X$ defined over $k$. 
Defining the topological entropy of $f$ raises some difficulties in this context, since the set $X(k)$ of $k$-points in $X$ can be endowed 
with a metric topology coming from the norm on $k$, but this space 
is not locally compact if $d\ge 1$ except if $k$ is a local field. To get around this problem, one considers the Berkovich analytification $X^{\an}$ of $X$, see~\cite{berkovich}. 
This space has good topological properties and is in particular compact. The map $f$ induces an analytic (hence continuous) map 
on a dense subset of $X^{\an}$. 
Note that most of the time (e.g., when the residue field of $k$ is uncountable) $X^{\an}$ is not metrizable, and 
the topological entropy of a continuous self-map of $X^{\an}$ is defined as the growth of the complexity of the sequence of open covers $\fU \vee \cdots \vee f^{-n} \fU$ where $\fU$ is a fixed open cover of $X$. 
This definition is convenient when $f$ is regular, but it is not easy to adapt it when $f$ is merely rational since it is not immediately related to orbits of points. 
We explain in details in \S\ref{sec:uniform} 
how to interpret Adler-Konheim-McAndrew's original definition in a Bowen perspective using the canonical uniform structure on $X^{\an}$ (see also~\cite{hood}).
We show that the topological entropy is computed by the growth of the cardinality of maximal $(n, \cE)$-separating sets
where $\cE$ is any entourage of the diagonal in $X^{\an} \times X^{\an}$. Using this interpretation, it is then not difficult 
to define the topological entropy of any rational map, even though it only defines a partially continuous map on $X^{\an}$, see \S\ref{sec:definition-topo-rat}.

To state our main theorem, we now need to discuss the notion of dynamical degrees first introduced by Russakovskii-Shiffman~\cite{rushifman}.
For any ample line bundle $L \to X$ and for any integer $j\in \{0, \ldots, d\}$, set
 $\deg_{L,j}(f) = f^* c_1(L)^j \cdot c_1(L)^{d-j}$. When $k=\C$ or more generally if the characteristic of $k$ is zero, then it is a theorem of Dinh and Sibony~\cite{DS05} that the sequence $\{C\deg_{L,j}(f^n)\}_{n\in \N}$ is sub-multiplicative for some $C>0$, so that one can define the $j$-th dynamical degree of $f$ by setting $\lambda_j(f) := \lim_n \deg_{L,j}(f^n)^{1/n}$. 
 Dynamical degrees do not depend on the choice of $L$, and are invariant under birational conjugacies. 
 When the characteristic of $k$ is positive, the sub-multiplicativity of the sequence of degrees (hence the existence of the dynamical degrees) 
was first obtained by the second author~\cite{TTT}, and an alternative approach has been recently developped  by N.-B. Dang~\cite{bac}.

Our first result states that the upper bound~\eqref{eq:DShtop} is true over any field.  
\begin{Theorem}\label{thm:main}
Let $(k,|\cdot|)$ be any complete metrized non-Archimedean field, and 
let $f \colon X \dashrightarrow X$ be any dominant rational self-map of a projective variety $X$ defined over $k$. 

Then the  topological entropy of the map induced by $f$ on the Berkovich analytification of $X$ 
is bounded from above by $\max_{0\le i \le \dim(X)} \{ \log \lambda_i(f) \}$.
\end{Theorem}

This result was obtained by the first author and Rivera-Letelier in the case $X = \mathbb{P}^1_k$, see~\cite[Th\'eor\`eme~C]{FR10}. 

Let us briefly indicate how our proof works in arbitrary dimension. For each $n$, consider 
 the Zariski closure $\Gamma_n \subset X^n$ of the set of orbits of length $n$, that is of the set of points $(x, \cdots, f^{n-1}(x))$ where
$x$ is not indeterminate under $f, \cdots, f^{n-1}$.  Observe that $\Gamma_n$ is an algebraic variety of dimension $d:= \dim(X)$ for all $n$. 
Fix an ample line bundle $L \to X$, and 
let $L_n$ be the restriction to $\Gamma_n$ of the line bundle $\sum_{i=0}^{n-1}\pi^*_i L$ where $\pi_i$ is the projection $X^n \to X$ onto the $i$-th factor. 

In the complex setting, we may choose a smooth positive metrization on $L$ whose curvature defines a K\"ahler form $\om$ on $X$.
By Wirtinger's theorem the volume of $\Gamma_n$ for the K\"ahler form $\om_n:=\sum_{i=0}^{n-1}\pi^*_i \om$ can be computed in cohomological 
terms, and we get $\vol_{\om_n} (\Gamma_n) = \deg_{L_n}(\Gamma_n) := c_1(L_n)^{\wedge d}$. 
By a clever calculation due to Gromov, we obtain
$\limsup_n \frac1n \log \deg_L(\Gamma_n) \le\max_{0\le i \le d} \{ \log \lambda_i(f) \}$. 
On the other hand, the minimality of complex subvarieties (or a theorem by Lelong) implies the volume $\vol_{\om_n} (\Gamma_n)$ to be no less
than  the number of $(n,\e)$-separated orbits up to a uniform constant which implies~\eqref{eq:DShtop}. 

Definitions of a K\"ahler form have been proposed in~\cite[\S 2.3]{NAMA} and~\cite[\S 3]{yuyue} in the non-Archimedean case
using the notion of models over the ring of integers $k^\circ = \{ |z| \le 1\} \subset k$. The precise definition of a K\"ahler form
is not relevant to our discussion, but models play a crucial role in our approach. A model of $X$ over $k^\circ$ will be by convention a flat projective scheme $\fX \to \spec k^\circ$
whose generic fiber is isomorphic to $X$. In our argument, the choice of a relatively ample line bundle $\fL \to \fX$ replaces 
the choice of a K\"ahler form in the complex setting. 

Given any entourage $\cE$ of the diagonal, we are first able to construct a suitable model $\fX$ of $X$
such that for any integer $n\ge0$ the following holds. 
The number of $(n,\cE)$-separated points is bounded from above by the number $Q_n$ of irreducible components
of the model $\mathfrak{G}^{\langle n \rangle}$ of $\Gamma_n$ obtained by taking the closure of $\Gamma_n$ inside $\fX^n$. 
Next we prove that
\[Q_n \le c_1(\fL_n)^{\wedge d} \cdot \mathfrak{G}^{\langle n \rangle}_s\]
where $\fL \to \fX$
is a relatively ample line bundle, 
$\fL_n = \sum_{i=0}^{n-1}\pi^*_i \fL$, and $\pi_i$ is the projection $\fX^n \to \fX$ onto the $i$-th factor. 
Using the invariance of intersection numbers under flat morphism, we get  $Q_n\le c_1(L_n)^{\wedge d}$. 
We then adapt Gromov's calculation using ideas from~\cite{bac} to relate $c_1(L_n)^{\wedge d}$ to the dynamical degrees, 
and our proof is complete.

\smallskip
\begin{center}
$\diamondsuit$
\end{center}
\smallskip

In the complex case, it follows from a deep theorem of Yomdin~\cite{yomdin,gromov-bourbaki} that Gromov's upper bound is attained whenever $f$ is regular.
The case of rational maps has triggered a series of works using pluripotential techniques to build invariant measures of maximal entropy. Using these tools, it is now known for many classes of rational maps that equality holds in~\eqref{eq:DShtop}, see e.g.~\cite{bedford-smillieII,Guedj2005,dujardin,DS10,DDG,dethelin-vigny,vigny,diller-liu-roeder}. In general however the inequality
can be strict (take e.g., $f[x:y:z]= [xz+z^2:y^2:z^2]$ which has topological entropy $0$ in $\mathbb{P}^2$). It is expected that for any dominant rational map 
$f\colon X \dashrightarrow X$, and for any $\epsilon>0$ 
there exists a birational map $h\colon Y \dashrightarrow X$ such that the topological entropy of $g := h^{-1}\circ f \circ h$ is no less than 
$-\epsilon+ \max_{0\le i \le d} \{ \log \lambda_i(f) \}$ (see~\cite{dethelin} for a recent construction that sheds some interesting perspective on this problem). 

In the non-Archimedean setting, the situation is very different and has been explored only in dimension $1$. 
The first author and Rivera-Letelier~\cite{FR22} have studied in depth the case of rational 
maps of the projective line. They have proved that any rational map $f\colon \p^1_k\to \p^1_k$ whose equilibrium measure has positive Lyapunov exponents satisfies $\htop(f)=\log(\deg(f))$, see~\cite{FR22}. 
Benedetto et al~\cite{BCK+17} have also produced examples of rational maps with coefficients in a $p$-adic 
field having positive entropy, for which one has strict inequality in~\eqref{eq:DShtop}. 

From now on, we shall focus our attention to one of the most striking difference between complex and non-Archimedean dynamics, that is the existence
of regular algebraic self-maps having dynamical degrees $>1$ and vanishing topological entropy. We shall also stick our attention to \emph{regular} maps
$f \colon X \to X$, in other words we assume that $f$ is an endomorphism of a projective variety $X$ of dimension $d$. Define
the residue field of $k$ by setting $\tilde k = k^\circ/k^{\circ\circ}$ with $k^{\circ\circ}=\{|z|<1\}$.

When $X= \p^1_k$ and $\deg(f)>1$, then $\htop(f)=0$ if and only if $f$ is M\"obius conjugated to a rational
map that descends to a rational map on $\tilde k$ the same degree as $f$, by~\cite[Th\'eor\`eme~D]{FR10}. 
In more geometric terms, $f$ is conjugated to a rational map $g\colon \p^1_k\to \p^1_k$ that extends to a morphism
of $k^\circ$-schemes $\mathfrak{g} \colon \p^1_{k^\circ}\to \p^1_{k^\circ}$. 
When the latter property is satisfied, then we say that $g$ has good reduction. Similar notions were introduced in higher dimensions, notably in~\cite{hutz,PST,STW}. 
In this paper, we shall say that a regular self-map $f\colon X \to X$ has good reduction if there exists a (flat projective) model $\fX \to \spec k^\circ$
such that the extension of $f$ to $\fX$ remains regular.

We can now state our second result.

\begin{Theorem}\label{thm:entropy good reduction}
Let $f : X \to X$ be any regular surjective map on a projective
variety defined over a complete non-Archimedean  field $k$.
Suppose one can find a model $\mathfrak{X}$ such that $f$ extends to a regular self-map of $\mathfrak{X}$. Then
the topological entropy of the map induced by $f$ on the Berkovich analytification $X^{\an}$ of $X$ is equal to $0$.
\end{Theorem}

It is natural to ask whether conversely any regular self-map of a projective variety having vanishing topological entropy has good reduction. 
In dimension $1$, Favre-Rivera Letelier's result mentioned above settles this question by the affirmative when $\deg(f)>1$. In higher dimensions, we may build
counter-examples as follows. Let $f_1\colon \p^1_k\to \p^1_k$ be an endomorphism of degree $\geq 2$ having good reduction and let $f_2\colon \p^1_k\to \p^1_k$ be an automorphism having a repelling fixed point. Then $f:=f_1\times f_2\colon (\p^1)^2\to (\p^1)^2$ has vanishing topological entropy and does not have good reduction. 

Recall that a morphism $f\colon X \to X$ is said to be polarized if there exists an ample line bundle $L\to X$ and an integer $q\ge 2$ such that 
$f^*L$ is linearly equivalent to $L^{\otimes q}$. 

We propose the following conjecture. 
\begin{conjecture}\label{conj1}
Let $k$ be any complete non-Archimedean field, and $X$ be a projective $k$-variety. 
Suppose $f\colon X \to X$ is a polarized morphism and $\htop(f)=0$. Then $f$ has
good reduction.
\end{conjecture}
Note that in dimension 1, for any selfmap $\deg(f)>1$, one can find some $q\ge2$ and a line bundle of positive degree $L$ such that $f^*L\simeq L^{\otimes q}$, hence the above conjecture is known to be true by \cite[Th\'eor\`eme~C]{FR10}. Also, in the above conjecture one can be even more ambitious (or optimistic), and relax the polarization assumption by only imposing that all dynamical degrees of $f$  are distinct.

\smallskip
\begin{center}
$\diamondsuit$
\end{center}
\smallskip

Let us explain how we prove Theorem~\ref{thm:entropy good reduction}. Recall that any flat projective model
$\pi\colon \fX \to \spec k^\circ$ has the following structure. The preimage of the generic point of $\spec k^\circ$ is isomorphic to $X$; and the preimage
of the special point (i.e., over the maximal ideal $k^{\circ\circ}$) is a projective scheme $\fX_s$ defined over the residue field $\tilde k$. 
In this situation, we have a canonical reduction map $\red\colon X^{\an} \to \fX_s$ which is anti-continuous when $\fX_s$ is endowed with the Zariski topology. 
Since $f$ has good reduction, it induces a regular scheme morphism $\ff_s \colon  \fX_s\to \fX_s$, and the following diagram is commutative:
\begin{equation}\label{eq:reduction}
\xymatrix{
 X^{\an}  \ar[d]_{\red} \ar[r]^f & X^{\an}  \ar[d]^{\red}\\
 \fX_s  \ar[r]^{\ff_s} & \fX_s}
\end{equation}
If we want to exploit this diagram, then it is natural  to first prove that the topological entropy of $\ff_s$ is $0$. 
We actually prove a very general result for any quasi-compact Noetherian schemes, endowed with either
its Zariski topology $\tau^{\zar}$ or with its constructible topology $\tau^{\con}$. We refer to \S\ref{sec:two-topos} for a discussion on the constructible
topology. Suffice it to say that it always Hausdorff, it refines the Zariski topology, and 
a quasi-compact scheme is compact for $\tau^{\con}$. 
\begin{Theorem}\label{thm:entropy zar et con}
Suppose $X$ is a quasi-compact Noetherian scheme, and let $f\colon X\dashrightarrow X$ be any dominant rational self-map.
Then we have
\[ \htop(f,\tau^{\con})=\htop(f,\tau^{\zar})=0~.\]
\end{Theorem}
We give a (short) combinatorial proof of $\htop(f,\tau^{\zar})=0$  in \S\ref{sec:noether-entropy}. In order to prove $\htop(f,\tau^{\con})=0$, we rely
on the variational principle which holds unconditionnally on a compact space. 

Returning to the proof of Theorem~\ref{thm:entropy good reduction}, we are now facing the problem to understand the entropy of $f$ on the fibers of the reduction map. 
When $V$ is any affine subscheme of $\fX_s$, then $\red^{-1}(V)$ is an affinoid domain by~\cite[Theorem~3.1]{Bosch77}, that is, a closed analytic subset of the closed unit ball $\bar{\D}^N(0,1)$. 
We have:
\begin{restatable}{Theorem}{affinoid}\label{thm:entropy-affinoid}
The topological entropy of any analytic self-map of an affinoid domain is equal to $0$. 
\end{restatable}

 Since the proof of this result contains the key ideas of the proof of Theorem~\ref{thm:entropy good reduction}, we shall 
 describe it in some details. 
 
Recall that an affinoid domain is the Berkovich spectrum\footnote{i.e., the set of multiplicative semi-norms on $A$ that are bounded by $|\cdot|$.}of a Banach $k$-algebra $(A, |\cdot|)$ of a special type, called an affinoid algebra. We denote by $\cM(A)$ this spectrum; this is a compact space.
We shall suppose that $A$ is reduced so that its spectral norm $\rho(a) := \lim_n | a^n|^{1/n}$ is a norm 
equivalent to $|\cdot|$. In this situation it is a classical construction to consider the reduction $\tilde A:= A^\circ/A^{\circ\circ}$
where $A^\circ = \{ \rho \le 1\}$ and $A^{\circ\circ}=  \{ \rho < 1\}$. One can prove that $\tilde A$ is a finitely generated $\tilde k$-algebra,
and we get a reduction map $\red \colon X \to \spec \tilde A$ sending a point $x\in X$ to the projection of the prime ideal $\{ a\in A, |a(x)| < 1\}$ in $\tilde A$. 
This reduction map is anti-continuous when $\spec \tilde A$ is endowed with the Zariski topology, and any analytic map $f\colon \cM(A) \to \cM(A)$ corresponds to a bounded endomorphism of $A$, and gives rise to a similar commutative
diagram as~\eqref{eq:reduction}.

 Pick any $0<\e \le1$. We extend the previous construction, and define  $A_\e:= A^\circ/\{ \rho < \e\}$. When $k$ is algebraically closed, one can prove that $A_\e$ is a 
 finitely generated algebra over the ring $k_\e = k^\circ/\{|z| < \e\}$, and as above
it is natural to associated to a point $x\in \cM(A)$ the ideal $\{ |a(x)|< \e\} \subset A_\e$. It is a fact that this ideal is only a primary ideal in general, and this suggests one to replace
$\spec A_\e$ by the space $\Id(A_\e)$ of \emph{all} ideals in $A_\e$. Again we have a reduction map  $\red_\e \colon X \to \Id(A_\e)$ which is anti-continuous when $ \Id(A_\e)$
 is endowed with the Zariski topology, and this construction is functorial.

The first key observation (Lemma~\ref{lem:subcoverAra}) is that for any open cover $\fU$ of $\cM(A)$, there exists $0<\e\le1$ and a finite cover $\fV$ of $\Id(A_\e)$
such that  $\red_\e^{-1}(\fV)$ refines $\fU$. This cover $\fV$ is not open for the Zariski topology, but it is open for the constructible topology $\tau^{\con}$ on $\Id(A_\e)$. (On the other hand, it is worthy to note that $\red_\e$ is still not a continuous map even if $\Id(A_\e)$ is equipped with the constructible topology.) These observations show that it is sufficient to prove that for any analytic map $f\colon \cM(A) \to \cM(A)$, and for any $0<\e \le1$, the induced map
$f_\e \colon (\Id(A_\e),\tau^{\con}) \to (\Id(A_\e),\tau^{\con})$ has  topological entropy $0$. 

Theorem \ref{thm:entropy zar et con} could be applied to yield the above needed property, had one not a difficulty: $k_\e$ is not a Noetherian ring in general (e.g., when $k$ is algebraically closed). We reduce to that case by considering finitely generated rings $R\subset k_\e$
containing all coefficients determining the morphism $f_\e$, and by building on an argument of absolute Noetherian approximation (see, e.g.,~\cite[\S 32.5]{stackproject}). 
To simplify the discussion we shall thus do as if $A_\e$ was Noetherian. 
In that case, $(\Id(A_\e),\tau^{\con})$ is a compact space which has very peculiar structure. 
Using Hochster's characterization of spectral spaces~\cite{hochster}, we prove that $(\Id(A_\e),\tau^{\con})$ is homeomorphic to $(\spec B,\tau^{\con})$ where $B$
is some Noetherian ring depending on $A$ and $\e$\footnote{Such a space is called a Priestley space, see~\cite{dickmann}.}. 
Extending an argument going back to Gignac~\cite{gignac}, see also~\cite{xie}, we prove that any Radon measure on $(\spec\bar{A},\tau^{\con})$ is atomic, 
which allows us to conclude by the variational principle. 

 \medskip
 
For the proof of Theorem~\ref{thm:entropy good reduction},  we were not able to exploit~\eqref{eq:reduction} and to apply directly Theorem~\ref{thm:entropy-affinoid} (in combination with Theorem~\ref{thm:entropy zar et con}).  However, the same line of arguments can be used. Recall that we have a projective model $\fX$ of $X$ such that the extension of $f$ to this model $\ff \colon \fX \to \fX$ is regular. We cover $\fX_s$ by affine subschemes $\{V_i\}_i$, so that  $\{\red^{-1}(V_i)\}_i$
 forms a finite cover by affinoid domains. Denote by $\cO(V_i)$ the affinoid algebra of analytic functions on $V_i$. Then for any $0<\e\le 1$ we may consider the $\e$-reductions 
 $ \Id(\cO(V_j)_\e)$, each equipped with its canonical reduction map
 $\red_\e \colon \red^{-1}(V_i) \to \Id(\cO(V_j)_\e)$. 
 
 The natural idea would be to patch these pieces $\Id(\cO(V_j)_\e)$ 
 together 
 and apply the previous arguments.
 Unfortunately,  we could not find a natural way to patch them directly, see Remark \ref{rem:cover-prim} for a more precise discussion.
 To get around this problem we observe that for any affinoid algebra $A$ and any $x\in \cM(A)$, 
the ideal $\red_\e(x)$ is primary, and this suggests to consider the closure  $\prim(A_\e)$ of the set of all primary ideals
in $A_\e$. We show that there is a natural way to patch $\prim(\cO(V_i)_\e)$ together (along compact subsets)  to get a compact space $\fX[\e]$ and an $\e$-reduction map $\red_\e \colon X^{\an} \to \fX[\e]$ .
We further show that $\ff$  induces a canonical continuous map $f_\e \colon (\fX[\e],\tau^{\con}) \to (\fX[\e],\tau^{\con})$.

 We conclude by proving that $f_\e$ has topological entropy zero (Proposition~\ref{prop:eps-zero}).
Our argument is based  on the variational principle and as before on a reduction to the Noetherian case
by a delicate absolute Noetherian approximation principle that we develop in details. 

\medskip

The construction of the $\e$-reduction which is natural in our dynamical context uses several objects that are exotic in the landscape of commutative algebra. 
The ring $k_\e$ is a local ring for which its maximal ideal is equal to its nilradical, but this ring is far from being Noetherian. The ideal space $\Id(R)$ was considered in some papers including~\cite{cornulier,fontana,dickmann}. The $\e$-reduction can be extended to $k$-analytic spaces equipped with a formal structure in the sense of Bosch~\cite{Bosch77}. 
It would be interesting to develop further applications of our construction in analytic geometry. 

\smallskip
\begin{center}
$\diamondsuit$
\end{center}
\smallskip

The plan of the paper is as follows. We discuss the topological entropy of continuous maps on compact spaces, and then more generally of partially defined maps
on compact spaces in Section\ref{sec:topo-def}. Section\ref{sec:gromov} contains the proof of Theorem~\ref{thm:main}. 
In Section \ref{sec:noetherian-priestley}, we give a proof of Theorem~\ref{thm:entropy zar et con}, and discuss in some details Priestley spaces and their relationships with spectral spaces.  In particular, we introduce the notion of Noetherian Priestley spaces that plays an important role in the sequel. 
Section\ref{sec:entropyaffinoid} concerns affinoid spaces and contains a proof of Theorem~\ref{thm:entropy-affinoid}. 
We develop our theory of $\e$-reduction associated to a projective model over $k^\circ$, and complete the proof of Theorem~\ref{thm:entropy good reduction}
in Section~\ref{sec:endo-goodred}. 
Base change is a powerful tool in non-Archimedean analysis. We discuss the behaviour of the topological entropy under this operation in the Appendix~\ref{sec:appennd}.

\medskip

\noindent
{\bf Acknowledgements}: we would like to thank Antoine Ducros, S\'ebastien Boucksom and Xinyi Yuan for discussions on the reduction map of a Berkovich analytic space
and on the invariance of the Euler characteristic in a non-Noetherian context.

%%%%%%%%%%%%%%%%%%%%%%%%%%%%%%%%%%%%%%%%%%%%%%%%

\section{Topological entropy}\label{sec:topo-def}

In this section, we gather some general facts on the topological entropy of continuous maps, and more generally
for partially defined continuous maps on compact sets. 

\subsection{Complexity of covers}
A finite cover $\fU= \{U_i\}_{i\in I}$ of a non-empty set $X$ is a finite family such that $\cup_i U_i = X$. 
We define its complexity $N(\fU)\in \N^*$ as the minimum number of sets belonging to $\fU$ that 
are needed to cover $X$. 
Given a second cover $\fV=\{V_j\}_{j\in J}$, we define the joint cover by $\fU \vee \fV = \{ U_i\cap V_j\}_{i,j}$. 
Observe that $N(\fU \vee \fV)\le N(\fU)\times N(\fV)$.
A cover $\fV$ is a refinement of $\fU$ if for all $j\in J$ one can find $i = i(j) \in I$ such that $V_j \subset U_{i}$. 
When $\fV$ refines $\fU$, we write $\fV \le \fU$, and it follows that $N(\fV)\ge N(\fU)$.

Let $f\colon X \to X$ be any map. Any cover $\fU$ induces a cover $f^{-1}(\fU)= \{f^{-1}(U_i)\}_{i\in I}$ 
whose complexity satisfies $N(f^{-1}(\fU))\le N(\fU)$ with equality when $f$ is surjective.
For each $n\in \N$, write $\fU_n= \fU \vee \cdots \vee f^{-(n-1)}(\fU)$.
The previous observations yield
\[
N(\fU_{n+m}) \le N(\fU_{n}) \times N(f^{-n}(\fU_{m})) \le N(\fU_{n}) \times N(\fU_{m})\]
so that we can define $h(f,\fU) = \lim_{n\to\infty} \frac1n \log N(\fU_n)= \liminf_{n\to\infty} \frac1n \log N(\fU_n) \in \N^*$.
Sometimes, we write $N(f,\fU,n)$ for $N(\fU_n)$ to emphasis $f.$
\begin{remark}
We shall work both with open covers (when $U_i$'s are all open) and with compact covers (when $U_i$'s are all compact). The latter situation appears naturally in 
Berkovich theory where proper analytic sets admits finite covers by affinoids.
\end{remark}

Let $\Gamma$ be any non-empty family of finite covers of $X$. Then we define
the $\Gamma$-entropy of $f$ by:
\[h(f,\Gamma):= \sup_{\fU \in \Gamma} h(f,\fU)~.\]

For any two non-empty families $\Gamma_1,\Gamma_2$ of finite covers of $X$, we say that $\Gamma_1$ is finer than $\Gamma_2$, if 
any finite cover in $\Gamma_2$ admits a refinement by a finite cover in $\Gamma_1$.

The next lemma follows immediately from the remarks in the previous section.
\begin{lemma}
Let $X$ be a set and $f: X\to X$ be an endomorphism. Let  $\Gamma_1,\Gamma_2$ be two non-empty  families of finite covers of $X$.
If $\Gamma_1$ is finer than $\Gamma_2$, then $h(f,\Gamma_1)\geq h(f, \Gamma_2).$
\end{lemma}

\subsection{Topological entropy and the variational principle}\label{subsectionentropcover}
When $X$ is a topological space, denote by $\Gamma_{{\rm top}}$ the family of finite open covers.
The topological entropy of $f$ is then defined as
\[
\htop(f) :=\htop(f,\Gamma_{\rm top})= \sup_{\fU \in \Gamma_{\rm top}} h(f,\fU)~.
\]

Recall that a quasi-compact  space is a topological space for which any open cover admits a finite subcover, 
and that compact spaces are Hausdorff quasi-compact spaces.  The notion of topological entropy is mainly interesting when $X$ is quasi-compact and $f$ is continuous.
Although $h(f,\fU)$ is finite for all $\fU$, it may happen that $\htop(f)$ is infinite (e.g. when $f$ is the full shift on an infinite set of symbols), so that in general $\htop(f)\in [0,+\infty]$.

\medskip

Suppose now that $X$ is compact. A Radon measure is a positive measure $\mu$ 
on the $\sigma$-algebra of Borelian sets which is inner regular, i.e., for all Borel set $A\subset X$
we have $\mu(A) = \sup \mu(K)$ where the supremum is taken over all compact sets $K\subset A$.

By Riesz' representation theorem, any positive linear functional on the space of continuous functions on $X$
is represented by a unique Radon measure.

Let $\mu$ be any probability measure on $X$. 
The entropy of a finite partition $\xi =\{\xi_i\}_{i\in I}$ of $X$ by Borel sets is defined to be
$H(\xi) = -\sum_{i\in I} \mu(\xi_i) \log \mu(\xi_i)$. One can check that $H(\xi \vee \zeta) \le H(\xi) + H(\zeta)$, 
so that if $f\colon X \to X$ is a measure preserving map we have 
$H(\xi_{n+m})\le H(\xi_{n})+ H(\xi_{m})$ where $\xi_n = \xi \vee \cdots \vee f^{-(n-1)}(\xi)$
so that $h(f,\xi)= \lim \frac1n H(\xi_n)$ is well-defined. 

The metric entropy of $f$ with respect to $\mu$ is then the quantity
$h(f, \mu) = \sup_\xi h(f,\xi)$ where the supremum is taken over all finite partitions of $X$ by Borel subsets. 

The topological entropy and the metric entropy are related by the so-called variational principle which states: 
\begin{theorem}\label{thm:variational}
Let $X$ be a (Hausdorff) compact topological space, and let $f\colon X\to X$ be any continuous map.
Let $\cM(f)$ be the set of all Radon probability measures that are $f$-invariant. 
Then 
\[
\htop(f) = \sup_{\cM(f)} h_\mu(f)~.\]
\end{theorem}
The original proof by Goodwyn~\cite{goodwyn} and Goodman~\cite{goodman} is valid in this degree of generality. 
Note that the Hausdorff assumption plays a crucial role in their proofs. It is believed that the variational principle holds for normal quasi-compact spaces, see~\cite{hood} 
and the
references therein. 

\subsection{Uniform structures and entropy of continuous maps}\label{sec:uniform}
Our basic reference is~\cite[\S 6]{kelley}.
A uniform structure on a set $X$ is a non-empty collection $\Phi$ of sets $\cE\subset X\times X$ (called entourage)
containing the diagonal  $\Delta= \{ (x,x) , \,  x \in  X\}$, and satisfying the conditions:
\begin{itemize}
\item
$\cE,\cE'\in \Phi$  implies $\cE\cap\cE'\in \Phi$;
\item
$\cE\in \Phi$ and $\cE\subset \cE'$ implies $\cE'\in \Phi$;
\item
$\cE^{-1} := \{(y,x) , \,  (x,y)\in \cE\}\in \Phi$ whenever $\cE\in \Phi$; 
\item
for any $\cE\in \Phi$, there exists $\cE'\in \Phi$ such that $\cE \supset \cE' \circ \cE' :=  \{(x,y) , \,  (x,z), (z,y) \in \cE\text{ for some } z\}$.
\end{itemize}
An entourage $\cE$ satisfying $\cE^{-1} = \cE$ is called symmetric.

\smallskip

Any uniform structure $\Phi$ defines a topology $\Phi^{\top}$ for which a set $U$ is open iff for every point $x\in U$ one can find an entourage $\cE$ such that 
$B_{\cE}(x) = \{y \in X , \,  (x,y)\in \cE\}\subset U$. 
A topological space with collection of open sets $\cT$ 
is uniformazible if one can find a uniform structure $\Phi$  such that $\Phi^{\top} = \cT$. 
One can prove that a topological space is uniformizable iff it is completely regular, see~\cite[p.188]{kelley}.

Any compact space $X$ is uniformizable; more precisely there exists a unique uniform structure $\Phi$ on $X$ defining its topology, see~\cite[p.198]{kelley}.
The uniform structure is given by the set of all neighborhoods of the diagonal. Note that any entourage $\cE$ contains a symmetric one of the form 
$\cup_i U_i \times U_i$ where $\{U_i\}_{i\in I}$ is a finite open cover of $X$.

Any metric space $(X,d)$ is also uniformizable: take $\Phi$ to be the neighborhoods of $\Delta$ included in 
$\{ (x,y), d(x,y) <\epsilon\}$ for some $\epsilon>0$.

\medskip
Let $X$ be any compact space equipped with its canonical uniform structure. Pick 
any continuous map $f\colon X \to X$. Observe that $f$ is uniformly continuous, i.e.
the preimage of any entourage by $f$ remains an entourage. We aim at characterizing the topological entropy of $f$
in terms of entourages,  following~\cite[\S 2]{hood}.

Let $\cE$ be any symmetric entourage of $X$, and pick any integer $n\in\N$. 
  \begin{itemize}
\item
  Two points $x,y$ are said to be $(n,\cE)$-separated iff there exists an integer
 $k\le n$ such that $(f^k(x), f^k(y)) \notin \cE$.
\item  
  A family of points  $F \subset X$ is called an $(n,\cE)$-covering set,
  iff for any point $x\in X$, there exists a point $y\in
  F$ such that $(f^k(x),f^k(y)) \in \cE$ for all $k\le n$.
\end{itemize}
  For any symmetric  entourage $\cE$, and for integer $n$, we write
  $R(n,\cE)$ for the minimum of the cardinality of an 
  $(n,\cE)$-covering set,  and $S(n,\cE)$ for the maximum of the cardinality of 
 set of points that are pairwise $(n,\cE)$-separated. 

Because $X$ is assumed to be compact, these two functions have finite values.
They are also decreasing on $\cE$ with respect to the inclusion. 
Let us prove that 
\begin{equation}\label{eq:entropy-basic}
R(n,\cE) \le S(n,\cE) \le R(n,\cE')
\end{equation}
if $\cE' \circ \cE' \subset \cE$.  Any maximal 
$(n,\cE)$-separating set is also an $(n,\cE)$-covering set hence $S(n,\cE) \ge
R(n,\cE)$.  If $E$ is an $(n,\cE)$-separating set, and $F$ an $(n,\cE')$-covering set with
 $\cE' \circ \cE' \subset \cE$, we choose
for all $x \in E$ a point $\phi(x) \in F$ such that
$(x,\phi(x))$ lies in $\cE'$. It is then clear that $\phi$ is then
injective. 
\begin{proposition}
For any continuous map $f\colon X\to X$ on a compact space $X$, we have 
\[\htop(f) = \sup_\cE \limsup_n \frac1n \log R(n,\cE) = \sup_\cE
  \limsup_n \frac1n \log S(n,\cE)~.\]
\end{proposition}

When $(X,d)$ is a compact metric space, we recover the standard definition of the topological entropy
given by Bowen~\cite{bowen}.

\begin{proof}
Denote by $h(f)$ the quantity in the right hand side. 
Let us fix  $\cE$ any entourage of $X$. One can then find a finite open cover 
$\fU = \{ U_i\}$ such that $U_i \times U_i
  \subset \cE$ for all  $i$. If $E$ is an $(n-1,\cE)$-separating set of maximal  cardinality, 
then an element of $\bigvee_0^{n-1}
  f^{-i} \fU$  contains at most one element of $E$, hence
  \begin{equation}\label{eq:bddsepless}
  S(n-1,\cE) \le N\left(\bigvee_0^{n-1} f^{-i} \fU\right).
  \end{equation}
 It follows that
$\htop(f) \ge h(f)$.

  For the converse inequality, fix any finite open cover $\fU = \{ U_i \}$. 
\begin{lemma} \label{lem:small-nghd}
There exists an entourage $\cE$  such that for every $x \in X$, the set $B_\cE(x)$ is included in some $U_i$.
\end{lemma}
  Fix any such entourage,  
 and let  $B^{n}_\cE(x):= \{ y , \,  (f^i(x),f^i(y))\in \cE \text{ for all }0 \le i \le n\}$. 
  Observe that for all $x\in X$ and for all $0\le i \le n$ we have $f^i(B^n_\cE(x))\subseteq B_\cE(f^i(x))$, so that $f^i(B^n_\cE(x))$ is contained in an element of
  $\fU$. 
  Let $F$ be a $(n-1,\cE)$-covering set with minimal cardinality so that 
   $X\subseteq \cup_{x\in F}B^{n-1}_\cE(x)$.
Since, for every $x\in F$, there exist $j(x)_0,\dots, j(x)_{n-1}$ such that 
 \[B^{n-1}_\cE(x)\subseteq U_{j(x)_0}\cap\dots\cap f^{-(n-1)}(U_{j(x)_{n-1}}),\]
it follows that 
$\{U_{j(x)_0}\cap\dots\cap f^{-(n-1)}(U_{j(x)_{n-1}}), x\in F\}$ forms a cover of $X$,
and we have 
   \begin{equation}\label{eq:bddsepup}
  N\left(\bigvee_0^{n-1}f^{-i} \fU\right) \le R(n-1,\cE).
   \end{equation}
The inequality
  $h(f) \ge \htop(f)$ follows.
\end{proof}

\begin{proof}[Proof of Lemma~\ref{lem:small-nghd}]
Pick any $x\in X$. We claim that there exists an entourage $\cE(x)$ such that 
$B_{\cE(x)}(x) \subset U_{i(x)}$.

Pick any $U_i$ containing $x$, set $i(x):= i$, and choose two open neighborhoods $V, W$ of $x$ such that
$x \in V \subset \bar{V} \subset W \subset U_i$. Set $\cE(x) := W^2\cup (X\setminus \bar{V})^2$. 
By construction $B_{\cE(x)}(x) \subset W \subset U_i$ which proves the claim. 

For each $x$, choose an entourage $\cE'(x)$ such that $\cE'(x) \circ \cE'(x) \subset \cE(x)$, and pick a finite cover
 $X\subset B_{\cE'(x_1)}(x_1) \cup \cdots \cup B_{\cE'(x_n)}(x_n)$.
 Set $\cE := \cap_{j=1}^n \cE'(x_j)$.
 
 Choose $x\in X$. Then we may find $j\in \{1, \cdots, n\}$ such that $x \in B_{\cE'(x_j)}(x_j)$. 
 It follows that 
 \[
 B_\cE(x) \subset B_{\cE'(x_j)}(x)\subset B_{\cE'(x_j)\circ \cE'(x_j)}(x_j)\subset
 B_{\cE(x_j)}(x_j) \subset U_{i(x_j)}~.\]
This concludes the proof.
\end{proof}

\subsection{Partially continuous maps}\label{sec:part-continuous}
Let $X$ be any compact space equipped with its canonical uniform structure. 
A partially continuous self-map of $X$ is a continuous map $f\colon X\setminus I(f)\to X$ where 
$I(f)\subseteq X$ is a closed subset.

\begin{remark}
A typical example of partially continuous self-map is a rational self-map of a complex projective variety. 
Another source of examples arise from rational maps on algebraic varieties endowed with the constructible topology, see~Section \ref{sec:definition-topo-rat} below. 
Note that in the latter case $I(f)$ can have non-empty interior.
\end{remark}

We define an \emph{admissible sequence} of $(X,f)$ to be a  sequence $\xi:=\{\xi_{n}, n\geq 0\}$ of (possibly empty) subsets of $X$ such that for $n\geq 1$, $\xi_n\subseteq X\setminus I(f)$ and $f(\xi_n)\subseteq \xi_{n-1}$. By induction on $n$, one sees that if $\xi_n$ is a non-empty open set for all $n$, then the $n$th iterate $f^n$ defines a partially continuous map
with $I(f^n) := X\setminus \xi_n$.

There is a maximal choice $\xi^{\max}(f)$ of admissible sequence which is defined as follows: $\xi^{\max}(f)_0:=X$, for $n\geq 1$, $\xi^{\max}(f)_n:=f|_{X\setminus I(f)}^{-1}(\xi^{\max}(f)_{n-1}).$  We note that all sets $\xi^{\max}(f)_n$ are open.

Another natural choice of admissible sequence is $\xi^{\self}(f)$ given by $\xi^{\self}_n(f):=\cap_{j\geq 0} \xi^{\max}_j(f)$ (note however that   $\xi^{\self}_n(f)$ may not be open).
When $I(f)=\emptyset$, we have $\xi^{\max}(f)_n=\xi^{\self}(f)_n=X$ for $n\geq 0.$

\begin{remark}
Since $X$ is compact, it is a Baire space, hence $\overline{\xi^{\max}(f)_n} =X$ for all $n$ implies
$\overline{\xi^{\self}(f)_n}=X$.
\end{remark}

Let $\cE$ be any symmetric entourage of $X$. For $n\geq 0$ and every $x\in \xi^{\max}_{n},$
define $B^n_\cE(x) := \{ y \in \xi_n^{\max} , \,   (f^i(x),f^i(y)) \in \cE, i=0,\dots,n\}.$ It is a neighborhood of $x$ in $\xi_n^{\max}$, 
and the map $\cE \mapsto B^n_\cE(x)$ is increasing with respect to inclusion. 
\begin{itemize}
\item If $y\in B^n_\cE(x)$, then $x\in B^n_\cE(y).$
\item For any symmetric entourage $\cE'$ such that $\cE'\circ\cE'\subseteq \cE$, then $x,y\in B^n_{\cE'}(z)$ implies $y\in B^n_{\cE}(x).$ 
\end{itemize}

\begin{lemma}\label{lemclosureball}Let $\cE, \cE'$ be any symmetric entourages of $X$ satisfying $\cE'\circ\cE'\subseteq \cE$. 
Let $Y, Z$ be subsets of $\xi_n^{\max}$. If $Y\subseteq \cup_{z\in Z} B^n_{\cE'}(z),$ then $\overline{Y}\cap \xi_n^{\max}\subseteq \cup_{z\in Z} B^n_{\cE}(z).$
In particular, we have $\overline{Y}\cap \xi_n^{\max}\subseteq \cup_{y\in Y}B^n_{\cE}(y).$
\end{lemma}
\proof For every $x\in \overline{Y}\cap \xi_n^{\max}$, there is $y\in B^n_{\cE'}(x)\cap Y$. 
Pick  $z\in Z$ such that $y\in B^n_{\cE'}(z)$. Then $x\in B^n_{\cE}(z)$ as required.
\endproof

\subsection{Topological entropy of a partially continuous map}\label{sec:entropy-part-continuous}
We fix any partially continuous map $f$ on a compact space $X$ and an admissible sequence $\xi$. 
We shall define the entropy $\htop(f, \xi)$.

Pick any symmetric entourage $\cE$ of $X$, and any integer $n\in\N$. 
  \begin{itemize}
  \item
  Two points $x,y\in \xi_n$ are said to be $(n,\xi,\cE)$-separated iff $x\not\in B^n_{\cE}(y).$
\item  
  A family of points  $F \subset \xi_n$ is called an $(n,\xi,\cE)$-covering set,
  iff $\xi_n\subseteq \cup_{x\in F}B^n_{\cE}(x).$
\end{itemize}
  For any symmetric  entourage $\cE$, and for integer $n$, we write
  $R(n,\xi,\cE)$ for the minimum of the cardinality of an 
  $(n,\xi,\cE)$-covering set,  and $S(n,\xi,\cE)$ for the maximum of the cardinality of 
the set of points that are pairwise $(n,\xi,\cE)$-separated.  
 
 When $\xi=\xi^{\max}$, we often omit $\xi$ in $R(n,\xi,\cE)$ and $S(n,\xi,\cE)$.

\begin{lemma}\label{lemrnxiefinite}For every $n\geq 0$, $R(n,\xi,\cE)$ is finite.
Moreover, we have
\begin{equation}\label{eq:entropy-basic*}
R(n,\xi,\cE) \le S(n,\xi,\cE) \le R(n,\xi,\cE') < \infty
\end{equation}
if $\cE' \circ \cE' \subset \cE$.  
\end{lemma}
\proof 
We begin with proving that for any subset $U \subset X$ there exists a finite subset $F_U \subset U$ such that $U \subset \cup_{x\in F_U}B_{\cE'}(x).$
Indeed since $X$ is compact, there is a finite subset $F'\subseteq X$ such that $U\subseteq X=\cup_{x\in F'}B_{\cE'}(x).$
Set $F'':=\{x\in F', B_{\cE'}(x)\cap U\neq\emptyset\}.$
For every $x\in F''$, pick $y_x\in B_{\cE'}(x)\cap U.$ Set $F_U = \{y_x , \,  x \in F''\}$. Then for every $y\in U$, there is $x\in F''$ such that 
$y\in B_{\cE'}(x)$, hence $y\in B_{\cE}(y_x)$ which concludes the proof of the claim. 

\medskip

Let us now prove that $R(n,\xi,\cE)$ is finite by induction on $n.$ 
Pick any symmetric entourage $\cE'$ such that $\cE'\circ\cE'\subseteq \cE$.

The fact that $R(0,\xi,\cE)$ is finite follows from the claim applied to $U= \xi_0$. 
Next assume that $R(n,\xi,\cE)<+\infty$. We will show that $R(n+1,\xi,\cE)<+\infty.$
The induction hypothesis implies that there is a finite set $G\subseteq \xi_{n}$ such that $\xi_{n}\subseteq \cup_{x\in G}B^n_{\cE'}(x).$
By the previous claim for each $x\in G$, there exists a finite set $F_{x,n}\subset f^{-1}( B^n_{\cE'}(x))\cap \xi_{n+1}$ such that 
$f^{-1}( B^n_{\cE'}(x))\cap \xi_{n+1} \subset \cup_{x\in F_{x,n}}B_{\cE'}(x).$

Pick any point $y\in\xi_{n+1}$. Then there exists a point 
$x\in G$ such that $f(y)\in B^n_{\cE'}(x)$, and then a point 
$x'\in F_{x,n}$ such that $y\in B_{\cE'}(x')$. Since $F_{x,n}\subset f^{-1}( B^n_{\cE'}(x))\cap \xi_{n+1}$, we have
$f(x') \in  B^n_{\cE'}(x)$ hence $y \in B^{n+1}_{\cE}(x')$, and $R(n+1,\xi,\cE)$ is less than the cardinality of
$\cup_{x\in G} F_{x,n}$. 

The proof of~\eqref{eq:entropy-basic*} is identical to~\eqref{eq:entropy-basic}. 
\endproof

Note that both $S(n,\xi,\cE)$ and $R(n,\xi,\cE)$ are increasing on $n$ and decreasing on $\cE$ with respect to the inclusion, and by~\eqref{eq:entropy-basic}, 
we have \[\sup_\cE \limsup_n \frac1n \log R(n,\xi,\cE)= \sup_\cE
  \limsup_n \frac1n \log S(n,\xi,\cE)~.\]
  \begin{definition}
  Pick any partially continuous self-map $f$ on a compact space $X$. Suppose $\xi$ is  an admissible sequence for which $\xi_n \neq \emptyset$ for all $n$. 
Then  we define the topological entropy of $(X,f,\xi)$ by
  \[
  \htop(f, \xi):=
  \sup_\cE \limsup_n \frac1n \log R(n,\xi,\cE)= \sup_\cE
  \limsup_n \frac1n \log S(n,\xi,\cE) \in [0, +\infty]~.\]
We also write $\htop(f):=\htop(f,\xi^{\max}),$ and call it the entropy of $f$. 
\end{definition}

  \begin{remark}
Note that when $\xi_n=\emptyset$ for some $n$, then $R(n,\xi,\cE)=0$.
In that case, it is convenient to set $\htop(f,\xi) = - \infty$.
  \end{remark}
  
 We now aim at comparing the entropy when the admissible sequence is modified.

\begin{lemma}\label{lemchangeseq}Let $\xi,\xi'$ be two admissible sequences for $(X,f).$ 
Let $\cE',\cE''$ be symmetric  entourages satisfying $\cE'\circ\cE'\subseteq \cE$ and $\cE''\circ\cE''\subseteq \cE'$.
Pick any integer $n\geq 0$.
\begin{itemize}
\item[(i)] If $\xi_n\subseteq \xi_n'$, then $S(n,\xi,\cE)\leq S(n,\xi',\cE)$. 
\item[(ii)] If $\overline{\xi_{n}}\subseteq\overline{\xi'_{n}}$,
 then $R(n,\xi,\cE)\leq R(n,\xi',\cE'')$.
\end{itemize}
 In particular, if $\overline{\xi_{n}}\subseteq\overline{\xi'_{n}}$ for all sufficiently large $n$, then we have $\htop(f,\xi)\leq \htop(f,\xi').$
\end{lemma}
\proof
Let us prove (i). Any 
$(n,\xi,\cE)$-separating set is also an $(n,\xi',\cE)$-separating set, hence we have $S(n,\xi,\cE)\leq S(n,\xi',\cE)$.

Now we prove (ii). Let $F$ be a $(n,\xi',\cE'')$-covering set of minimal cardinality. By Lemma \ref{lemclosureball}, we have
\[\xi_n\subseteq \overline{\xi'_n}\subseteq \cup_{x\in F}B^n_{\cE'}(x).\]
Set $F':=\{x\in F , \,  B^n_{\cE'}(x)\cap \xi_n\neq \emptyset\}.$ For every $x\in F'$, there is a point $y_x\in B^n_{\cE'}(x)\cap \xi_n.$
Note that
$\xi_n\subseteq  \cup_{x\in F}B^n_{\cE'}(x)\subseteq \cup_{x\in F}B^n_{\cE}(y_x)$
hence $R(n,\xi,\cE)\leq R(n,\xi',\cE'')$. This completes the proof.
\endproof

We now list direct consequences of the previous lemma. 
Recall that any compact space is a Baire space.
\begin{corollary}\label{corselfmax}
 If for every $n\geq 0,$ $\xi^{\max}_n$ is dense in $X$, then $\htop(f,\xi^{\self})=\htop(f,\xi^{\max}).$
\end{corollary}

The next result shows that $\htop(f, \xi)$ is insensitive to the size of $I(f).$
  \begin{corollary}\label{lemchangeif}Let $f,g$ be two partially continuous self-maps of a compact space $X$
that coincide outside $I(f)\cup I(g)$.
If $\xi$ is a sequence of subsets of $X$ which is both admissible for $f$ and $g$, 
then $\htop(f,\xi)=\htop(g,\xi).$
\end{corollary}

This corollary implies the following properties. 
\begin{corollary}\label{corentropyenlargeif} For any closed subset $I'\supset I(f)$, consider the partially continuous map $f':=f|_{X\setminus I'}\colon X\setminus I'\to X.$ 
Then $\htop(f')\leq \htop(f)$. If moreover, we have $\overline{\xi_n^{\max}(f)}=\overline{\xi_n^{\max}(f')}$ for every $n\geq 0$, then $\htop(f')= \htop(f)$.
\end{corollary}

\begin{corollary}\label{corsubsysentropy}Let $Y$ be a closed subset of $X$ such that $f(Y\setminus I(f))\subseteq Y.$ 
Set $f|_Y:=f|_{Y\setminus I(f|_Y)}\colon  Y\setminus I(f|_Y)\to Y$ where $I(f|_Y):=I(f) \cap Y.$ Then $\htop(f|_Y)\leq\htop(f).$
\end{corollary}

\begin{corollary}\label{corsujsysentropy}Let $\pi\colon X\to Y$ be a continuous surjection. Let $f\colon X\setminus I(f)\to X$ and $g\colon  Y\setminus I(g)\to Y$ be partially continuous self-maps
such that $I(f)\subseteq \pi^{-1}(I(g))$ and $\pi\circ f|_{X\setminus \pi^{-1}(I(f))}=g\circ\pi|_{X\setminus \pi^{-1}(I(f))}.$ Then $\htop(f)\geq \htop(g).$
\end{corollary}
\proof
Observe that our assumptions imply $\xi^{\max}_n(f) \supset \pi^{-1}(\xi^{\max}_n(g)).$ Then $\xi:=\pi^{-1}(\xi^{\max}_n(g))$ is an admissible sequence of subsets of $X$.
Let $\cE$ be a symmetric entourage of $Y.$ Then $\cE':=\pi^{-1}(\cE)$ is a symmetric entourage of $X$.
Since $\pi$ is surjective, 
any $(n,\cE)$-separating set $E\subset\xi^{\max}_n(g)$ can be lifted to a subset $E' \subset \xi_n(f)$ which is $(n,\cE')$-separating. It follows that 
$S(n,\xi^{\max}(g),\cE) \le S(n,\xi,\cE')$. 
Taking limits $n\to\infty$ and supremum over all entourages yield 
$\htop(g)\leq \htop(f,\xi^{\max}(f))\leq \htop(f).$
\endproof

\section{Gromov's upper bound on topological entropy for rational maps}\label{sec:gromov}

In this section, we prove the upper bound on the topological entropy (Theorem~\ref{thm:main}).
We first explain how to define the topological entropy of a dominant rational self-map of any projective variety $X$
defined over a complete metrized field $(k,|\cdot|)$. Then we explain in the non-Archimedean case how to build special 
models over $k^\circ$ adapted to a fix open cover of $X$. Finally we deduce from this construction our main theorem.

\subsection{Dynamical degrees of rational maps}
\subsubsection*{N\'eron-Severi space of an algebraic variety}
Let us fix some notation. For any complete algebraic variety $X$ of pure dimension $d$, denote by $\NS(X)$ its real Neron-Severi space, that is
the space of real Cartier divisors on $X$ modulo numerical equivalence, see e.g., \cite{lazarsfeld}. It is a finite dimensional space equipped with $d$-multilinear
symmetric intersection form $\NS(X)^d \to \R$
which we denote by $(\om_1, \om_2,  \ldots, \om_d) \mapsto (\om_1 \cdot \om_2 \cdot \ldots \cdot \om_d)\in \R$.
There is a canonical morphism from the Picard group to the N\'eron-Severi space given by the first Chern class
so that we may attach to any line bundle $L\to X$ an element 
$c_1(L)\in\NS(X)$.

\subsubsection*{Nef and psef classes}
A class $\om\in \NS(X)$ is said to be nef if $\om\cdot [C] \ge 0$ for all irreducible curve $C$ on $X$. 
It is pseudo-effective (psef for short), if there exists a sequence of effective divisors $D_n$ such that $[D_n] \to \om$ in $\NS(X)$. 
We write $\om\ge0$ when the class is psef. 

Given any family $\om_1, \om_2,  \ldots, \om_j\in \NS(X)$ we interpret $\om_1\cdot \om_2 \cdot \ldots \cdot \om_j$
as a symmetric $(d-j)$-multilinear form on $\NS(X)$ and we write $\om_1 \cdot \ldots \cdot \om_j\ge0$ whenever
$(\om_1 \cdot \ldots \cdot \om_j \cdot \alpha_1 \cdot \ldots \cdot \alpha_{d-j}) \ge 0$ for all nef classes $\alpha_i \in \NS(X)$. 

\subsubsection*{Siu's inequalities}
These inequalities, following from~\cite[Theorem~2.2.15]{lazarsfeld}, state that for any pair of nef classes $\alpha , \beta \in \NS(X)$ such that $(\beta^d) >0$ then we have
\[
\alpha \le d \frac{(\alpha\cdot\beta^{d-1})}{(\beta^d)} \beta~.
\]
By arguing by induction on the dimension this inequality may be generalized to the following statement, see~\cite{bac} (or~\cite{jiang-li} for better constants). 
There exists a constant $S_d$ depending only on the dimension such that
for any family of nef classes $\alpha_1, \cdots, \alpha_j$, and for any nef $\beta$ class such that $(\beta^d)>0$, then we have
\begin{equation}\label{eq:Siu}
\alpha_1\cdot \ldots \cdot \alpha_j \le S_d \frac{(\alpha_1\cdot \ldots \cdot \alpha_j \cdot \beta^{d-j})}{(\beta^d)}\, \beta^j~.
\end{equation}

\subsubsection*{Rational maps} 
Let $f\colon X \dashrightarrow X$ be any dominant rational map on a 
projective variety $X$ of dimension $d$. 
By definition $f$ is given by its graph $\Gamma_f$ which is an irreducible subvariety of $X \times X$ such that 
the first projection $p_{1,f}\colon \Gamma_f \to X$ is birational. The indeterminacy locus $I(f)$ of $f$ is by definition
the set of points $q\in X$ that do not admit a Zariski open neighborhood $U$ over which $p_{1,f}$ is an isomorphism. 
Note that $f = p_{2,f} \circ p_{1,f}^{-1}$ on $X \setminus I(f)$ where $p_{2,f}\colon \Gamma_f \to X$ denotes the projection onto the second factor.

\subsubsection*{Dynamical degrees} 
We keep the same notations as in the previous paragraph. 
Fix any ample line bundle $L\to X$.
For each $k\in\{0, \cdots, d\}$, we set \[\deg_{k,L}(f^n) = p_{2,f^n}^*(c_1(L)^k) \cdot p_{1,f^n}^*(c_1(L)^{d-k})\]
where $p_{1,f^n}, p_{2,f^n} \colon \Gamma_{f^n} \to X$ are the natural projections
onto the first and second factor respectively. 
By~\cite{TTT,bac}, there exists a constant $C>1$ such that for each $k$ the sequence $\{C\deg_{k,L}(f^n)\}_n$
is sub-multiplicative. By Fekete's lemma, we may thus define $\la_k(f) := \lim_n\deg_{k,L}(f^n)^{1/n}$.

Note that the sub-multiplicativity implies the existence of a constant $C(f)>0$ such that for all $k\in\{0, \cdots, d\}$ and for 
$n\in \N$ we have 
\begin{equation}\label{eq:basic-sub-mult}
\deg_{k,L}(f^n)\le C(f) (\Lambda + \epsilon)^n~,
\end{equation}
with $\Lambda = \max_i\{\la_i(f)\}$.

Note that when $k=d$, the quantity $\deg_{k,L}(f)$ is equal to the product of $(c_1(L)^d)$ by the degree of the finite field extension $[k(X):f^*k(X)]$. 
Since $[k(X):f^{n*}k(X)]= [k(X):f^*k(X)]^n$, we obtain that 
$\la_d(f) = [k(X):f^*k(X)]$. When $f$ is separable and $k$ is uncountable, 
then $\la_d(f)$ can be interpreted as the number of preimages of a general point in $X$.

\subsection{Analytification of projective $k$-schemes}\label{sec:anal-kspace}

Our main reference is~\cite[\S 3.4]{berkovich}.

\subsubsection*{Metrized fields}
Let $k$ be any complete metrized non-Archimedean field. We let $k^\circ=\{|z|\le 1\}$ be its ring of integers
with maximal ideal $k^{\circ\circ} = \{|z|<1\}$ and denote by $\tilde{k} = k^\circ/k^{\circ\circ}$ its
residue field. We also let the value group be $|k^*|$, which is a subgroup of $(\R^*_+,\times)$. 

When $k$ is trivially valued (i.e.  $|k^*| =\{1\}$), we have $k^\circ = k$, $k^{\circ\circ} = (0)$, and $\tilde{k} =k$.
When $k$ is not trivially valued and algebraically closed, then $|k^*|$ is a divisible group (hence dense in $\R^*_+$), and $\tilde{k}$ is also algebraically closed.

\subsubsection*{Berkovich analytification}
Berkovich constructed an analytification functor over any complete non-Archimedean field
that generalizes the classical construction over $\C$. We present it briefly referring to~\cite[\S 3.4]{berkovich} for details. 

Suppose $X= \spec A$ is an affine variety over $k$. The Berkovich analytification $X^{\an}$ is an analytic space
whose underlying topological space $|X^{\an}|$
consists of the set of multiplicative semi-norms on $A$ whose restriction to $k$ is  the given field norm $|.|_k$, and
endowed with the topology of the pointwise convergence. Embedding $X$ into an affine space
$\mathbb{A}^N_k$ we may write $X^{\an} = \bigcup_{r \ge 1} X \cap  \bar{\mathbb{D}}^N(0,r)$ and thus cover
$X^{\an}$ by a family of affinoid domains. In this way, we may define a natural structure sheaf on $X^{\an}$ which turns
this space into an analytic space. For any $x\in X^{\an}$, we denote by $\cH(x)$ the completed residue field, 
which in this case is the completion of the quotient ring $A/\{ |f(x)| =0\}$ with respect to the norm induced by $x$. It is a complete metrized field extension
of $(k,|\cdot|)$.

When $X$ is a general $k$-scheme of finite type, then we cover it by affine charts $X = \cup U_i$ and define $X^{\an}$ as the union of 
the Berkovich spaces $U^{\an}_i$ patched together using the analytification of the patching maps between the $U_i$'s.
The space $X^{\an}$ can also be identified with the set of pairs $x=(\xi_x,v_x)$ where $\xi_x$ is a scheme-theoretic point in $X$, and
$v_x$ is a real valued valuation on the residue field $\kappa(x)$ of $X$ at $x$ such that $e^{-v_x}$ extends the field norm on $k$.
Note that $\cH(x)$ is a field extension of $\kappa(x)$.

By~\cite[Proposition~3.4.7]{berkovich}, $X^{\an}$ is compact whenever $X$ is a proper $k$-scheme (e.g., when $X$ is projective).

\subsubsection*{Models}
Suppose $X$ is a projective variety of dimension $d$ defined over $k$. 
A (projective) model of $X$ over $k^\circ$ is a flat projective scheme $\fX \to \spec k^\circ$ whose generic fiber is equipped with an isomorphism
with $X$. 

Note that $\fX$ is covered by affine charts of the form $\spec A$ where $A$ is finitely generated over $k^\circ$. 
Since $A$ is flat over $k^\circ$, it has torsion-free, and a theorem of Nagata, see~\cite[Theorem~3.2.1]{antieau} implies  
$A$ to be finitely presented (i.e., is the quotient of $k^\circ[x_1, \cdots, x_n]$
by some finitely generated ideal), so that any model is locally of finite presentation over $k^\circ$.

The projective space $\p^d = \proj k[z_0, \cdots, z_d]$ admits a canonical model \[\p^d_{k^\circ} =  \proj k^\circ[z_0, \cdots, z_d].\] 
For any embedding $X \subset \p^d$, 
 the Zariski closure of $X$ inside $\p^d_{k^\circ}$ defines a model of $X$\footnote{This follows from the fact that $X$ is reduced and 
 a module over a valuation ring is flat iff it is torsion-free.}. 
A model $\fX_1$ dominates another one $\fX_2$  when the identity map on $X$ extends to a regular morphism
$\fX_1 \to \fX_2$. When $\fX_1$ dominates $\fX_2$, we write $\fX_1\ge \fX_2$. Given any two models $\fX_1, \fX_2$ of $X$ there always exists
a third one satisfying $\fX_3\ge \fX_1$, and $\fX_3\ge \fX_2$ (take the fibered product $\fX_1\times_{\spec k^\circ} \fX_2$). 

We denote by $\fX_s$ the special fiber of a model $\fX$. By definition it is a (possibly reducible) projective scheme defined over $\tilde{k}$.

Note that we do not assume $k$ to be discretely valued so that $\fX$ may not be a Noetherian scheme in general. 
Also we shall sometimes use non-normal models.

\begin{proposition}\label{lem:generic-to-special}
Let $\fX$ be any model of $X$ over $k^\circ$, and let $\fL_1, \cdots, \fL_d$ be $d$ line bundles over $\fX$. 
For each $1 \le i \le d$, denote by $L_i$ the line bundle induced by $\fL_i$ on the generic fiber $X$. 
Then we have
\[
(\fL_1\cdot \ldots \cdot\fL_d\cdot \fX_s) = (c_1(L_1) \cdot \ldots \cdot c_1(L_d))~.
\]
\end{proposition}

\begin{proof}
By multi-linearity, we only need to justify this formula in the case $\fL:= \fL_1= \cdots = \fL_d$. We set $L:= \fL|_X$. 
We rely on the definition of the intersection product in terms of the Euler characteristic as done by Snapper-Kleiman, see~\cite[Example~18.3.6]{fulton}
so that
$\chi (\fX_s,\fL_s^{\otimes n})= (\fL_s^d\cdot \fX_s) \frac{n^d}{d!} + o(n^d)$ and $ \chi (X,L^{\otimes n})= (c_1(L)^d)\frac{n^d}{d!} + o(n^d)$.
The lemma thus follows from the next result.
\end{proof}
 
\begin{lemma}\label{proflatcharcon}
For every line bundle $\fL$ on $\fX$, $\chi (X, L)=\chi(\fX_s,\fL_s).$
\end{lemma}
Note that when $k$ is a discrete valuation ring, then $\spec k^\circ$ is Noetherian, and the result follows directly from~\cite[Theorem~III.9.9]{hartshorne}. 
The point here is to see how to get around the Noetherian assumption. 

\proof[Proof of Lemma~\ref{proflatcharcon}]
Since $\fX$ is locally of finite presentation over $k^\circ$, the structure morphism $\Phi\colon\fX \to \spec k^\circ$ is finitely presented, and the structure sheaf $\mathcal{O}_\fX$
is coherent (see~\cite[Corollary~3.2.2]{antieau}). Since $\fL$ is an invertible sheaf on $\fX$, it is also coherent over $\fX$.

Now let us consider the derived direct image $R\Phi_* \fL$. 
Recall that if $0 \to \fL \to I^0 \to I^1 \cdots$ is any injective resolution of $\fL$ (of $k^\circ$-modules), then 
 $R\Phi_* \fL$ is determined by the complex $0 \to \Phi_* I^0 \to \Phi_* I^1 \cdots$. 
Since $\fL$ is coherent and $\Phi$ is flat and proper, it follows from~\cite[Corollary~3.2.2]{antieau} again that $R\Phi_* \fL$ is perfect, which means that it is quasi-isomorphic to a bounded complex $N^\bullet$ 
of finitely generated free $k^\circ$-modules. 

Since $\Phi$ is flat, by~\cite[Theorem~2.1.2]{antieau}, the derived direct image and the derived pull-back commute.
More precisely,
if $\eta\colon \spec k\to \spec k^\circ$ is the generic point, and $\phi\colon X \to \spec k$ is the structure morphism, then 
$R\phi_*L$ is represented by the complex $N^\bullet \otimes_{k^\circ} k$. In a similar way, let $s\colon\spec \tilde{k}\to \spec k^\circ$ denote the special point, and let $\tilde{\phi} \colon \fX_s \to \spec \tilde{k}$ be the structure morphism.  Again, $R\tilde{\phi}_*L_s$ is represented by $N^\bullet \otimes_{k^\circ}\tilde{k}$.

This implies
$H^{i}(X,L) = H^i(N^\bullet \otimes_{k^\circ} k)$ and 
$H^i(\fX_s, \fL_s)=H^i(N^\bullet \otimes_{k^\circ}\tilde{k})$ for all $i$. it follows that $\chi (X, L)=\chi(\fX_s,\fL_s)$. 

The reader may also consult~\cite[Appendix~A.4]{boucksom-eriksson} as an alternative reference to~\cite{antieau}.
\endproof

\subsubsection*{Reduction map}
Any model $\fX$ gives rise to a reduction map $\red_\fX \colon X^{\an} \to \fX_s$ which is defined as follows. 
We can send any point $x = (\xi_x, v_x) \in X^{\an}$ to the closure of $\xi_x$ in $\fX$. 
This gives a map $\spec \cH(x) \to \fX$ making the following  diagram commutative:
\[\xymatrix{
\spec \cH(x)  \ar[d] \ar[r] & \fX \ar[d]\\
\spec \cH(x)^\circ  \ar[r] & \spec k^\circ}
\]
By the valuative criterion of properness there exists a unique lift $\spec \cH (x)^\circ \to \fX$, and 
we let $\red_\fX(x)$ to be the image of the special point of $\spec \cH (x)^\circ$ under this morphism.
By construction $\red_\fX$ is anti-continuous in the sense that the preimage of an open (resp. closed) set is closed (resp. open). 
~\cite[Corollary 2.4.2]{berkovich}.

In Sections~\ref{sec:entropyaffinoid} and~\ref{sec:endo-goodred}, we will generalize the construction of reduction to $\e$-reduction for $\e\in (0,1].$
The usual reduction can be viewed as the case $\e=1.$

\subsubsection*{Model functions}
Fix any model $\fX$ of $X$, and let $Z$ be any Cartier divisors supported on the special fiber $\fX_s$. 
Then we may define a function $\varphi_Z\colon X^{\an} \to \R$ as follows. 
Pick an affine chart $U_i = \spec A_i$, where $A_i$ is a finitely generated $k^\circ$-algebra, 
choose an equation of $Z$ on $U_i$, say $f_i \in A_i$. 
Given any point $x\in U_i^{\an}$, we set $\varphi_Z(x) := \log |f_i(x)|$.
Since $Z$ is vertical, $f_i$ is a unit in the $k$-algebra $A_i \otimes k$, and $\varphi_Z$
induces a continuous function on $X^{\an}$. 
When $qZ$ is Cartier for some $q\in \Z$, then we set $\varphi_Z := \frac1q \varphi_Z$.

A model function is a continuous function $\varphi \colon X^{\an} \to \R$ which is equal to 
$\varphi = \varphi_Z$ for some $\Q$-Cartier divisor $Z$ in $\fX$ supported on $\fX_s$. 

Model functions are stable by $\max$, by multiplication by any rational numbers and by sums. It contains all constant functions, and
it is proved in~\cite[Theorem~7.12]{gubler} that they separate points.
It follows from Stone-Weierstrass' theorem that the space of model functions is dense in the space of continuous functions on $X^{\an}$ for the supremum norm.

\subsubsection*{Special covers on Berkovich spaces}

Let $X$ be any 
projective variety over $k$.
For any model $\fX$, we let $\fU(\fX)$ be the finite open cover $\red_\fX^{-1}(Z)$ where $Z$ ranges over all irreducible components
of $\fX_s$. 

\begin{proposition}\label{prop:special covers}
For any finite open cover $\fU$ of $X^{\an}$, there exists a projective model $\fX$ of $X$ such that 
$\fU(\fX)$ is a refinement of $\fU$.
\end{proposition}
\begin{proof}
Pick any point $x\in X^{\an}$. Choose $U_x\in \fU$ containing $x$. By Urysohn's lemma, 
we may find  a continuous function $\psi_x$ such that $\psi_x \equiv 0$ on $X^{\an} \setminus U_x$ and $\psi_x (x) =1$.
By the density of model functions, for any rational $\epsilon>0$ there exists a model function $\psi'_x$ such that 
$\sup |\psi'_x - \psi_x| \le \epsilon$. Set $\varphi_x = \max \{ \psi'_x- \epsilon , 0 \}$.
This is a model function that is equal to $0$ outside $U_x$ and takes a positive value at $x$.
Let $V_x := \{ \varphi_x >0\}$. By compactness of $X^{\an}$ we may extract a finite subcover $V_{x_i}$.
Pick any model $\fX$ such that for all index $i$, we have $\varphi_{x_i} = \varphi_{Z_i}$ for some vertical divisor $Z_i \subset \fX_s$. 
Then $\fU(\fX)$ is a refinement of $\fU$.
\end{proof}

\subsection{Constructible and Zariski topologies on schemes} \label{sec:two-topos}
We discuss now the notion of constructible topology on an arbitrary scheme. This brief section is not logically necessary for the proof of Theorem~\ref{thm:main}, but the constructible topology appears at various stages in the proof of Theorem~\ref{thm:entropy good reduction}.

Let $R$ be any commutative ring with unit.  We may view $\spec R$ as a subset of $\{0,1\}^R$ by attaching to any prime ideal $\fA$ 
its characteristic function $\chi_\fA \colon R\to \{0,1\}$
defined by $\chi_\fA(f) =1$ iff $f\in \fA$. The product topology on $\{0,1\}^R$ restricts to the so-called constructible topology $\tau^{\con}$
on $\spec R$. Since the set of proper prime ideals is closed in $\{0,1\}^R$ for the product topology, 
$(\spec R, \tau^{\con})$ is compact and totally disconnected. 

The collection of sets of the form $f^*(\spec B)$ for some ring morphism $f\colon R \to B$ satisfies the axioms for closed sets, and this topology
coincides with the constructible topology, see~\cite[Appendix~2]{knight} or~\cite[exercices~27,28, Chapter~3]{atiyah-macdonald}.

\begin{example}
Let $k$ be any algebraically closed field, and $R= k[T]$. Denote by $\eta$ the generic point so that
$\spec R= k \cup \{\eta\}$. Then \emph{all points} (including $\eta$) are closed for $\tau^{\con}$; 
the restriction of the constructible topology to $k$ is the discrete topology;
and a basis of neighborhood of $\eta$ is given
by complements of finite sets in $k$. In other words, $(\spec R, \tau^{\con})$ is the one point compactification of 
the discrete set $k$.
\end{example}

Recall that the constructible sets is the smallest collection of subsets of $\spec R$ that contains
retrocompact\footnote{that is open sets for which the intersection with any quasi-compact open sets remains quasi-compact.} open sets, and is stable under Boolean operations (that is by taking finite unions, finite intersections and complementation). 
Retrocompact open sets in $\spec R$ are complements of sets of the form $\spec (R/I)$ where $I$ is a finitely generated ideal. 
It follows that constructible sets form
a basis of clopen sets for the constructible topology. 

In fact, a set is constructible iff it is clopen for the constructible topology. 
Indeed, any clopen for $\tau^{\con}$ is a finite union of constructible sets hence is constructible.

When $R$ is Noetherian, any (Zariski) open (or closed) set is constructible.

\smallskip

The definition of the constructible topology on an arbitrary scheme $X$ is as follows. 
As before constructible sets is the smallest collection of subsets of $X$ that contains
retrocompact open sets, and is stable under Boolean operations; and $\tau^{\con}$ is the topology
generated by constructible sets.

Any open affine subscheme $U$ of $X$ is constructible (and clopen for $\tau^{\con}$), and the restriction of the 
constructible topology to $U$ is the one defined above for spectra of rings.

Note that $(X,\tau^{\con})$ is always Hausdorff, and a quasi-compact scheme is compact for $\tau^{\con}$.

\smallskip

Finally suppose that $X$ is a $k$-variety for some field $k$. Endow $k$ with its trivial norm, and consider $X^{\an}$ 
its Berkovich analytification. Recall that we have a canonical injection $\imath\colon X\to X^{\an}$ such that
on any affine chart $U = \spec A$, a prime ideal $x$ is mapped to the norm on $A$ induced by the trivial norm on the quotient $A/x$. 
In this case, the constructible topology on $X$ can be identified
to the restriction of the Berkovich topology on $\imath(X) \subset X^{\an}$. Indeed a basis for the Berkovich topology on $U^{\an}$
is given by $V(f,I)=\{|f(x)| \in I\}$ for some $f\in A$ and $I$ a segment in $\R_+$. And
$\imath^{-1} (V(f,I)) = \{ x,\, f\in x\}$, $\{ x,\, f\notin x\}$, $\spec A$, and $\emptyset$
when $I\cap \{0,1\}= \{0\}, \{1\}, \{0,1\}$ and $\emptyset$ respectively, so that $\imath^{-1} (V(f,I))$
also forms a basis for the constructible topology on $\spec A$.

\subsection{Topological entropy of rational self-maps} \label{sec:definition-topo-rat}
Let $f\colon X \dashrightarrow X$ be any dominant rational selfmap of a projective variety $X$ defined over a complete metrized field $k$.
We let $\Gamma_f \subset X \times X$ be its graph with projection maps $p_{1,f}, p_{2,f}\colon \Gamma_f \to X$, so that
$p_{1,f}$ is a birational morphism. The set of points above which $p_{1,f}$ is not a biregular morphism is the indeterminacy locus $\Ind(f)$ of $f$ which is a
subvariety of codimension at least $2$.

Since $f$ is not a continuous map, but we may view is as a partially continuous map $f\colon X^{\an}\setminus I(f)\to X^{\an}$ where $I(f):=\Ind(f)^{\an}$.
Its topological entropy is the topological entropy as a partially continuous map.

For each $n\in \N$, let $U_n(f)\subset X$ be the Zariski open set of points $x\in X$ for which
$x, f(x), \cdots, f^{n-1}(x)$ do not belong to the indeterminacy locus of $f$. Then $\xi^{\max}_n=U_n(f)^{\an}$, and
$\xi^{\max}_n$ is dense for all $n$. 

\begin{remark}
The topological entropy that is considered in~\cite{DS05} and \cite{Guedj2005} is $\htop(f, \xi^{\self})$. 
By Corollary \ref{corselfmax}, we have $\htop(f)=\htop(f, \xi^{\self})$, so that our notion coincides
with the existing one in the literature.
\end{remark}

Let us record the following result concerning base change. We refer to Appendix~\ref{sec:appennd} for the proof and a more detailed discussion. 

\begin{proposition}\label{probasechangeentropy}
Let $K/k$ be any complete metrized field extension, and let $f_K\colon X_K \dashrightarrow X_K$ be the rational self-map induced by base change. 
Then $\htop(f_K) \ge \htop(f)$.
\end{proposition}

\subsection{Proof of  Theorem~\ref{thm:main}}

When $k=\C$, Theorem~\ref{thm:main} was proved by Dinh and Sibony \cite{DS05}. 

We may and shall thus assume that $k$ is non-Archimedean.
By Proposition~\ref{probasechangeentropy}, after a suitable base change, we may assume that $k$ is algebraically closed and is not trivial valued. In this case, $X(k)\subseteq X^{\an}$ is dense.  By Corollary~\ref{corsujsysentropy}, after replacing $X$ by its normalization, we may assume that $X$ is normal.

Given an entourage $\cE$, we shall give upper bounds for maximal $(n,\cE)$-separating sets. 
By Proposition~\ref{prop:special covers}, we may suppose $\cE = \cup_i |U_i| \times |U_i|$
where $ \{U_i\} = \fU(\fX)$ is a special cover induced by a projective model $\fX$.
Consider the admissible system  $\xi$ with $\xi_n=U_n(f)(k), n\geq 0.$ Since $\overline{\xi_n}=X^{\an}$ for all $n\geq 0$,
Lemma \ref{lemchangeseq} yields $\htop(f,\xi)=\htop(f).$
We now relate $S(n,\xi,\cE)$ to a geometric invariant that is amenable to geometric estimations.

Denote by $\fX^{\langle n \rangle}$ the $n$-fold fibered product of $\fX$ over $\spec k^\circ$. 
This is a model of $X^n$ over $k^\circ$. Look at the $n$th graph
$\Gamma_n \subset X^n$ of $f$, i.e. the Zariski closure of the image of the morphism 
of the set of points
$(x, f(x), \ldots, f^{n-1}(x))$ with $x \in U_n(f)(k)$.
We let  $\mathfrak{G}^{\langle n \rangle}$ be its closure in $\fX^{\langle n \rangle}$.
It is a projective model for $\Gamma_n$. For each $k=0, \cdots, n-1$, we let $\pi_k\colon \mathfrak{G}^{\langle n \rangle}\to \fX$
be the projection onto the $(k+1)$th factor. 

\begin{proposition}\label{prop:bound-lelong}
The cardinality of a maximal $(n,\cE)$-separating set is bounded by the number of irreducible components of 
the special fiber of $\mathfrak{G}^{\langle n \rangle}$.
\end{proposition}
\begin{proof}
Let $\{Z_i, i\in I\}$ denote the set of irreducible components of $\fX_s$. Pick $x_1, \ldots, x_r$ an
$(n,\xi,\cE)$-separating set of  maximal  cardinality, and define
$y_l= (x_l, \cdots, f^{n-1}(x_l)) \in \Gamma_n(k)$.
Let $W_l$ be an irreducible component of $\mathfrak{G}^{\langle n \rangle}_s$ containing the point $r_{\fX^{\langle n \rangle}}(y_l)$. 
Since the set is $(n,\xi,\cE)$-separating, 
for each $l \neq l'$ there exists at least one $0\le k\le n-1$ such that 
 $r_\fX (f^k(x_l)),  r_\fX (f^k(x_{l'}))$ do not belong to the same  $Z_{i}$ for any $i$.
Suppose by contradiction that $W:= W_l = W_{l'}$. Then $\pi_k(W)$ is an irreducible subvariety of $\fX_s$
that contains both points $r_\fX (f^k(x_l))$ and  $r_\fX (f^k(x_{l'}))$.
Since $\pi_k(W) = \cup_i (\pi_k(W)\cap Z_i)$ it follows that $\pi_k(W) \subset Z_i$ for at least one index $i$, 
which is impossible. 
\end{proof}

\begin{remark}
Proposition~\ref{prop:bound-lelong} is the non-Archimedean analog  of the fact that in the complex case, maximal $(n,\cE)$-separating sets
control the euclidean volume of $\Gamma_n$ thanks to a theorem by Lelong, see~\cite{gromov}.
\end{remark}

We now conclude by estimating the number say $Q_n$ of irreducible components of $\mathfrak{G}^{\langle n \rangle}_s$.
Pick any ample line bundle $\fL \to \fX$, and look at $\fL_n= \sum_{i=0}^{n-1}\pi_i^* \fL$ on $\fX^{\langle n \rangle}$ where $\pi_i:  \fX^{\langle n \rangle}\to \fX$ denotes the projection onto the $i$th factor. Then 
 $\fL_n$ is an ample line bundle. Since $X$ is a normal algebraic variety, it has pure dimension $d$, and since $\fX\to \spec k^\circ$ is flat,
 $\fX_s$ (hence $\mathfrak{G}^{\langle n \rangle}_s$) has also pure dimension $d$.
 
 The key observation is that the intersection product 
$c_1(\fL_n)^{\wedge d} \cdot \mathfrak{G}^{\langle n \rangle}_s$ is greater than the sum of $c_1(\fL_n)^{\wedge d}|_Z$ over all irreducible components $Z$ of $\mathfrak{G}^{\langle n \rangle}_s$ so that
$$
Q_n \le c_1(\fL_n)^{\wedge d} \cdot \mathfrak{G}^{\langle n \rangle}_s~.
$$
Now since $\mathfrak{G}^{\langle n \rangle} \to \spec k^\circ$ is flat, we may compute the right hand side 
over $k$ and by Proposition~\ref{lem:generic-to-special} we get 
$$
Q_n \le c_1(L_n)^{\wedge d}~,
$$
where $L_n\to \Gamma_n$ is the line bundle obtained by restricting $\fL_n$ to the generic fiber of  $\mathfrak{G}^{\langle n \rangle}$.
We may now follow the arguments of Gromov~\cite{gromov} and Dinh-Sibony~\cite{DS05} and prove
\begin{keylemma}
For any $\e>0$, there exists a constant $C = C(\e)>0$ such that 
\[ 
c_1(L_n)^{\wedge d} \le C\, \max_{0\le i \le d}  (\lambda_i(f) + \e)^n
\]
for all $n$.
\end{keylemma}
Taking the $n$th root of both sides, letting $n\to\infty$ and then $\e\to0$ yields the required bound 
$\limsup_n \frac1n \log c_1(L_n)^{\wedge d} \le \max_{0\le i \le d}  \lambda_i(f)$.

\begin{proof}[Proof of the key lemma]
The proof of this lemma in characteristic zero by Dinh and Sibony~\cite[Lemme~2]{DS05} uses the regularization of  positive closed currents.
It is likely that one can use the methods  based on Chow's moving lemma and devised by the second author in~\cite{TTT}
to adapt Dinh and Sibony's arguments to the case of positive characteristic. 
We propose below an algebraic argument relying on Siu's inequalities in the spirit of~\cite{bac} which is technically less demanding.
This proof works in arbitrary characteristic.

\medskip

Recall that $\Gamma_n$ is the Zariski closure in $X^n$ of the image of the morphism 
of the set of points
$(x, f(x), \ldots, f^{n-1}(x))$ with $x \in U_n(f)(k)$.
Denote by $\pi_k\colon \Gamma_n \to X$ the projection onto the $k$th factor.
Set $\om := c_1(L)$ and $\om_k := \pi_k^* \om\in \NS(\Gamma_n)$
Since $L_n := \sum_{k=0}^{n-1} \pi_k^* L$, we have 
\[c_1(L_n)^{\wedge d} = 
\sum_{I\in\{0, \cdots, n-1\}^d }
\om_I \text{ with }\om_I:=
\left(\om_{i_1}\cdot  \om_{i_2}\cdot \ldots \cdot  \om_{i_d}\right).
\]
Given a set $I = \{i_1, \cdots, i_d\} \in \{0, \cdots, n-1\}^d$, we may re-order its elements so that $i_1 \le i_2 \le \cdots \le i_d$, 
and set $n_1= \min\{i_j\}$, $r_1= \#\{j, i_j = n_1\}$, and then define recursively 
$n_k = \min\{i_j, j> r_{k-1}\}$ and $r_k =  \#\{j, i_j = n_k\}$. 
In this way, we get a natural bijection between sets $I\in\{1, \cdots, n\}^d$ and
the set $T$ of pairs
$(\underline{n}, \underline{r})$  where $\underline{n}=(n_1,\dots, n_s)$  with $0\leq n_1<\dots<n_s\le n-1$, $\underline{r}=(r_1,\dots, r_s)$  with $r_i\geq 1, i=1,\dots, s$ and $\sum_{i=1}^sr_i=d$.
For each $(\underline{n}, \underline{r})\in T$, we set 
\[
\Omega(\underline{n}, \underline{r}) := \left(\om_{n_1}^{r_1} \cdot \om_{n_2}^{r_2}\cdot \ldots \cdot\om_{n_s}^{r_s}\right)
\]
so that 
\[c_1(L_n)^{\wedge d} = \sum_{(\underline{n}, \underline{r})\in T} \Omega(\underline{n}, \underline{r}) ~.\]
Pick any $\epsilon>0$, and let $\Lambda= \max_i \lambda_i(f)$.
We claim the existence of a positive constant $C>0$ such that 
\begin{equation}\label{eq:estim0}
\Omega(\underline{n}, \underline{r}) \le C (\Lambda+\epsilon)^{n_s}
\end{equation}
where $n_s = \max n_i$.

Note that this claim implies $c_1(L_n)^{\wedge d}\leq n^dC (\Lambda + \epsilon)^{n-1}$
and concludes the proof.

\medskip 

It thus remains to prove the claim. 
Pick $0\le l \le n-1$, and consider the natural map $\phi \colon  \Gamma_n \to \Gamma_{f^l}$ which sends the point
$(x, \cdots, f^{n-1}(x))$ to $(x, f^l(x))$ when $x\in U_n(f)(k)$. 
Recall that $p_{1,f^l}, p_{2,f^l}\colon \Gamma_{f^l} \to X$ denotes the canonical projections of the graph of $f^l$ to $X$. 
Since $\phi$ is birational and
$\phi^* p_{1,f^l}^*(c_1(L))= \om_0$,  $\phi^* p_{2,f^l}^*(c_1(L))= \om_l$ we conclude that 
\[
\deg_{k,L}(f^l)
=
(\om_{0}^{d-k}\cdot \omega_{l}^k)
~.\]
Recall also that $\deg_{k,L}(f^n) \le C(f) (\Lambda + \epsilon)^n$ by~\eqref{eq:basic-sub-mult}.

The proof of~\eqref{eq:estim0} now goes as follows. We prove by induction on $s$ that there exists a constant $C(s)>0$ such that for each $n$, for any pair
$\underline{n}=(n_1,\dots, n_s)$  with $0\leq n_1<\dots<n_s\le n-1$, and $\underline{r}=(r_1,\dots, r_s)$ with $|r|:=\sum_{i=1}^sr_i\le d$
we have  
\begin{equation}\label{eq:estim00}
\Omega(\underline{n}, \underline{r}):=  \left(\om_{n_1}^{r_1} \cdot \om_{n_2}^{r_2}\cdot \ldots \cdot\om_{n_s}^{r_s}\right)
 \le C(s) (\Lambda+\epsilon)^{n_s} \om_0^{|r|}~.
\end{equation}
Note that~\eqref{eq:estim0} is the subcase of~\eqref{eq:estim00} when $|r|=d$.

When $s=1$, we have
 $\underline{n}=(n_1)$ with $0\leq n_1\le n-1$, and $r_1 \le d$, and~\eqref{eq:Siu} implies:
\begin{align*}
\Omega(\underline{n}, \underline{r}) 
&= 
\om_{n_1}^{r_1} 
\le 
S_d  \frac{(\om_{n_1}^{r_1}\cdot \om_0^{d-r_1})}{(\om_0^d)} \om_0^{r_1}
\\
&=
\frac{S_d}{(\om_0^d)}  \deg_{r_1,L}(f^{n_1}) \om_0^{r_1}
\le
\frac{S_dC(f)}{(\om^d)}(\Lambda+\epsilon)^{n_1} \om_0^{r_1}~.
\end{align*}
which proves~\eqref{eq:estim00} for $s=1$, and $C(1):= \frac{S_dC(f)}{(\om^d)}$.

Now assume~\eqref{eq:estim00} has been proved for $s$, and pick any pair $\underline{n}=(n_1,\dots, n_s, n_{s+1})$  with $0\leq n_1<\dots<n_s< n_{s+1}\le n-1$, 
and $\underline{r}=(r_1,\dots, r_s, r_{s+1})$ with $\sum_{i=1}^{s+1}r_i \le d$. 
 We consider the map $F_{n_1} \colon \Gamma_n \to \Gamma_{n-n_1}$
 which sends the point $(x, \cdots, f^{n-1}(x))$ to $(f^{n_1}(x), \cdots, f^{n-1}(x))$ when $x\in U_n(f)(k)$.
 We have
 \begin{align*}
 \Omega(\underline{n}, \underline{r}) 
 &=
  \left(\om_{n_1}^{r_1} \cdot \om_{n_2}^{r_2}\cdot \ldots \cdot \om_{n_{s+1}}^{r_{s+1}}\right)
 \\
 &=
 F_{n_1}^* \left(\varpi_{0}^{r_1} \cdot \varpi_{n_2-n_1}^{r_2}\cdot \ldots \cdot \varpi_{n_{s+1}-n_1}^{r_{s+1}}\right)
 \end{align*}
 where $\varpi_j$ is the pull-back by the $j$th projection of $c_1(L)$ from $X$ to $\Gamma_{n-n_1}$.
 
 We now apply the induction hypothesis to the pair
 $\underline{n}'=(n_2-n_1,\dots, n_s-n_1, n_{s+1}-n_1)$, 
and $\underline{r}'=(r_2,\dots, r_s, r_{s+1})$. This yields the bound
\[
(\varpi_{n_2-n_1}^{r_2}\cdot \ldots \cdot \varpi_{n_{s+1}-n_1}^{r_{s+1}}
)
 \le C(s) (\Lambda+\epsilon)^{n_{s+1}-n_1} \varpi_0^{r_2+ \cdots + r_{s+1}}
 ~,
 \]
so that 
\begin{align*}
\Omega(\underline{n}, \underline{r}) 
& \le 
C(s) (\Lambda+\epsilon)^{n_{s+1}-n_1} F_{n_1}^*\varpi_0^{r_1+r_2+ \cdots + r_{s+1}}
\\ 
&\le 
\frac{S_d}{(\om_0^d)} C(s) (\Lambda+\epsilon)^{n_{s+1}-n_1} (F_{n_1}^*\varpi_0^{|r|} \cdot \om_0^{d-|r|}) \om_0^{|r|}
\\ 
&= 
\frac{S_d}{(\om_0^d)} C(s) (\Lambda+\epsilon)^{n_{s+1}-n_1} \deg_{|r|,L}(f^{n_1}) \om_0^{|r|} 
\\ 
&\le 
C(s+1) (\Lambda+\epsilon)^{n_{s+1}} \om_0^{|r|}
\end{align*}
with $C(s+1)=\frac{S_d C(f)}{(\om_0^d)} C(s)$
by~\eqref{eq:Siu}. This concludes the proof.
\end{proof}

%%%%%%%%%%%%%%%%%%%%%%%%%%%%%%%%%%%%%%%%%%%%%%%%

\section{Zero entropy maps on Noetherian and Priestley topological spaces}\label{sec:noetherian-priestley}

In this section, we prove two results. First we show that the topological entropy of any continuous self-map $f\colon X \to X$ on a quasi-compact Noetherian topological space is $0$ (Theorem~\ref{thm:zar-entropy}). 
We actually estimate directly the growth of the complexity of covers of the form  $\fU \vee \cdots \vee f^{-n} \fU$. 

Then we discuss the notion of Priestley topological spaces which are compact totally disconnected spaces homeomorphic the
spectrum of a commutative ring endowed with its constructible topology. 
We introduce the notion of Noetherian Priestley space, and prove the vanishing of the entropy of any 
partially continuous self-map of such a space (Theorem~\ref{thm:priestley-entropypartial}). This in turn implies Theorem~\ref{thm:entropy zar et con}.

In Section~\ref{sec:ideal-space-commutative}, we discuss an example of Noetherian Priestley spaces that will play a prominent role in the next section: 
the set of all (not necessarily prime) ideals of a Noetherian ring.

\subsection{Entropy on Noetherian spaces}\label{sec:noether-entropy}

A topological space $X$ is said to be Noetherian if any decreasing sequence of closed subsets is eventually stationnary.
Equivalently $X$ is Noetherian if and only if every nonempty collection of closed subsets has an element which is minimal under inclusion. 
A closed set $F\subset X$ is said to be irreducible if it cannot be written as a union $F = F_1 \cup F_2$ of two strict closed subsets $F_1, F_2 \subsetneq F$. 
It is a fact that any closed subset $F\subset X$ can be written as a finite union $F= F_1\cup \cdots \cup F_r$ of irreducible closed subsets. 
A decomposition with $F_i\not\subset F_j$ for $i\neq j$ is unique up to permutation, and in this case the $F_i$'s are called the irreducible components of $F$, see~\cite[Proposition~1.5]{hartshorne}.

One also defines the dimension of a Noetherian topological space $X$ as the supremum over all integers $d$ such that there exists a chain $\emptyset\subsetneq F_0\subset F_1\subset \cdots \subset F_d$ of distinct irreducible closed subsets of length $d+1$.

\begin{theorem}\label{thm:zar-entropy}
Let $X$ be any quasi-compact Noetherian space, and let $f\colon X \to X$ be any continuous map. Then for any open cover $\fU$ of $X$, and for any $\epsilon>0$, one has
\begin{equation}\label{eq:sub-expo}
N(\fU_n) \le C (1+\epsilon)^n
\end{equation}
for some $C>0$. In particular, we have $\htop(f)=0$.
\end{theorem}

\begin{remark}
In particular for any quasi-compact Noetherian scheme $X$, and any regular map $f\colon X \to X$ we have $\htop(f,\tau^{\zar})=0$ which yields a special case of Theorem~\ref{thm:entropy zar et con}.
\end{remark}

\begin{remark}
It was proved by Gignac that any signed Borel measure of finite mass on a Noetherian space for which every non-empty irreducible closed subset has a generic point 
(a.k.a. a Zariski space) is necessarily atomic, see~\cite[Theorem~1.2]{gignac}.
It is however not clear whether the variational principle applies on arbitrary non Hausdorff quasi-compact spaces so that Gignac's theorem 
does not immediately imply $\htop(f)=0$.
\end{remark}

\begin{proof}We need the following lemma.
\begin{lemma}\label{lem:subexpbdd}
Pick any integer $\epsilon>0$. Then one can find positive numbers $M_1,\cdots, M_k$ and positive integers $N_1 < N_2 < \cdots < N_k$ such that
$M_i (1+\epsilon)^{(1-N_i)} \le \frac{\epsilon}{10^i}$, and
\[
N(n,\fU)\leq N(n-1,\fU) + M_1 \, N(n-N_1,\fU)+ \cdots + M_k \, N(n-N_k,\fU)~,\]
for all $n\geq N_k$.
\end{lemma}
This lemma implies that the growth of the sequence $N(n,\fU)$ is bounded from above by the growth of the sequence $\{U_n\}$ defined by induction
$U_n =  U_{n-1} + M_1 \, U_{n-N_1} + \cdots + M_k \, U_{n-N_k}$, $U_i = N (i,\fU)$ for all $0\le i\le N_k$. 
The characteristic polynomial of the linear recurrence equation is equal to 
$P(T)= T^{N_k} - T^{N_k-1} - M_1 T^{N_k-N_1} - \cdots - M_k$. 
Note that any complex root $\lambda$ of $P$ satisfies
\[
|\lambda| 
= \left|1 + \frac{M_1}{\lambda^{N_1-1}} + \cdots + \frac{M_k}{\lambda^{N_k-1}}\right|
\]
so that if $|\lambda| \ge 1+\epsilon$, then we obtain
$|\lambda| \le 1 + \sum_{i=1}^\infty \frac{\epsilon}{10^i} < 1+\epsilon$, which is a contradiction. 
It follows that all complex roots of $P$ have modulus $\le 1+\epsilon$, and the bound~\eqref{eq:sub-expo} follows.
\end{proof}

\begin{remark}
The argument does not give a polynomial bound for the growth of $N(n,\fU)$. It would be interesting to explore whether $N(n,\fU)\lesssim n^{\dim(X)}$. 
This bound holds when $\dim(X) \le 1$.
\end{remark}

 \begin{proof}[Proof of Lemma~\ref{lem:subexpbdd}]
 For any closed subset $Z\subset X$, denote by $N_Z(n,\fU)$
 the minimal number of elements in $\fU_n$ to cover $Z.$
 Note that $N_{Z\cup Z'}(n,\fU) \le N_Z(n,\fU)+N_{Z'}(n,\fU)$, hence we may (and shall) assume that $X$ is irreducible.
Write $\fU=\{U_1,\dots,U_m\}.$

For every irreducible closed subset $Z$ of $X$, there is $j(Z)\in \{1,\dots,m\}$ such that $Z\cap U_{j(Z)}\neq \emptyset.$
Fix an integer $N\geq 1.$ Set $H(Z,N):=Z\setminus(\cap_{i=0}^Nf^{-i}(U_{j(f^i(Z))})),$ which is a (possibly reducible) proper closed subset of $Z.$

For $n\geq N,$ if $\fV$ is a subset of $\fU_{n-N}$ that covers $f^N(Z)$ and $\fV'$ is a subset of $\fU_{n}$ which covers $H(Z,N)$, the union of 
\[\{U_{j(Z)}\cap \dots\cap f^{-N}U_{j(f^{N-1}(Z))}\}\vee f^{-N}(\fV) \text{ and }\fV'\]
is a subset of $\fU_n$ which covers $Z.$
We thus get
\begin{equation}\label{eq:partial}
N_Z(n,\fU)\leq N_{f^N(Z)}(n-N,\fU)+ N_{H(Z,N)}(n,\fU) \leq N(n-N,\fU) +N_{H(Z,N)}(n,\fU)~,
\end{equation}
Apply this argument first to $Z=X$ and $N=1$ so that 
$$N(n,\fU) \leq N(n-1,\fU) +N_{H_1}(n,\fU), n\geq 1$$ for some proper closed subset $H_1$ of $X$.
Denote by $M_1$ the number of irreducible components of $H_1$, and pick $N_1\geq 1$ large enough such that 
$M_1 (1+\epsilon)^{(1-N_1)} \le \frac{\epsilon}{10}$.
Now apply~\eqref{eq:partial} to $N=N_1$ and to each irreducible component of $H_1$. We obtain
$$N_{H_1}(n,\fU)\le M_1 N(n-N_1,\fU) +M_1 N_{H_2}(n,\fU), n\geq N_1$$ where $H_2$ is a proper closed subset of $H_1$.
We may thus proceed by induction, obtaining a sequence of integers $M_i$ and closed subsets $H_{i+1}\subsetneq H_i$
such that
\[
N(n,\fU) \leq N(n-1,\fU) + M_1 \, N(n-N_1,\fU)+ \cdots + M_i \, N(n-N_i,\fU) + M_{i}\, N_{H_{i+1}}(n,\fU), n\geq N_i
.\]
Since $X$ is Noetherian, we have $H_i=\emptyset$ for $i$ large enough, which completes the proof.
 \end{proof}

\subsection{Priestley spaces}

We refer to~\cite{dickmann} for generalities on this notion.

Recall that in a poset $(X,\le)$ a down-set (resp. an up-set) $U$ is a set such that $x\in U$ and $y\le x$ (resp. $x\le y$) implies $y\in U$.
Note that $U$ is a down-set if and only if $X\setminus U$ is an up-set.
A topological poset $(X,\tau,\le)$ is called a Priestley space if it is quasi-compact and for any $y\not\le x$
there exists a clopen down-set $U$ such that $x\in U$ and $y\notin U$. Any Priestley space is Hausdorff, zero-dimensional, and in fact 
its set of clopen down-sets and up-sets forms a basis for its topology.

Pick any Priestley space $(X,\tau,\le)$. The set $\tau^u$ of all open up-sets of $X$ defines a topology, and $(X,\tau^u)$ is
a spectral space\footnote{In other words, $X$ is quasi-compact, satisfies the $T_0$-separation axiom, the set of its quasi-compact open subsets
is closed under finite intersections and forms a basis for its topology, and any non-empty irreducible closed subset of $X$ has a generic point.}.
By a theorem of Hochster~\cite{hochster}, a topological space is spectral iff it is homeomorphic to the spectrum of a commutative ring (hence the name). 

Conversely, let $(X,\tau)$ be any spectral space. Define the order relation $x\le y$ iff $x$ belongs to the closure of $\{y\}$. 
The constructible (or patch) topology $\tau^\#$ on $X$ is determined by a basis of open sets given by quasi-compact open subsets of $X$ and their complements.
Then $(X,\tau^\#,\le)$ is a Priestley space.

A continuous map $f\colon X \to Y$ between spectral spaces is called spectral if the preimage of any quasi-compact open subset remains quasi-compact open.
The induced map $f^\#\colon X^\# \to Y^\#$ on the corresponding Priestley spaces is continuous and order-preserving.  Any continuous and order-preserving map $f\colon X \to Y$ between Priestley spaces is called a Priestley map.

The above correspondence induces an equivalence of categories between Priestley spaces and spectral spaces~\cite[1.5.15]{dickmann}. Unlike Noetherian spaces, there are Priestley selfmaps with positive entropy, see e.g. Remark \ref{RemarkPositiveEntropyPriestley}. In the next subsection, we present an important subclass of Priestley spaces (to be used later in the paper) for which selfmaps always have zero entropy.

\subsection{Noetherian Priestley spaces and entropy}
Let $\tau^u$ be the collection of all open up-sets.
\begin{definition}
We say that a Priestley space $(X,\tau,\le)$ is Noetherian when its associated spectral space $(X,\tau^u)$ is a Noetherian topological space.
\end{definition}

\begin{remark}\label{remrenoetherian}
Equivalently $(X,\tau,\le)$ is Noetherian if any decreasing sequence of closed down-sets is eventually stationnary. If $(X,\tau)$ is a Noetherian topological space, then 
$(X,\tau,\le)$ is a Noetherian Priestley space. However a Noetherian Priestley space is not necessarily a Noetherian topological space (see \S\ref{sec:commutative} below for examples). To complicate things further, there exist Noetherian spectral topological spaces that are not homeomorphic to the spectrum of a Noetherian ring, see~\cite[12.4.10]{dickmann}.
\end{remark}

The following theorem generalizes \cite[Theorem 1.12]{xie}; see \cite[Theorem~1.2]{gignac} for another related result. 
\begin{theorem}\label{thm:priestley-entropy}
Any Radon measure $\mu$  of finite mass on a Noetherian Priestley space $(X,\tau,\le)$ is atomic: there exists a countable set $F$ such that 
$\mu = \sum_{x\in F} \mu (x) \delta_x$ where  $\sum_F \mu(x)< \infty$.
\end{theorem}

\begin{proof}[Proof of Theorem~\ref{thm:priestley-entropy}]
Any Radon measure can be decomposed as a sum of an atomic measure and a measure without atom.
It is thus sufficient to prove that for any probability Radon measure $\mu$ on $X$, there exists $x\in X$ such that $\mu(x)>0$.

We proceed by contradiction and assume $\mu(x)=0$ for all $x\in X$.
We claim that for every clopen down-set $K$ of $X$ with $\mu(K)>0$, there is a clopen down-set $K'\subsetneq K$ with $\mu(K')>0.$
Note that $X$ is a clopen down-set, applying this claim, we get a sequence of clopen down-sets $K_n, n\geq 0$ satisfying $K_0=X$, for every $n\geq 0,$ $\mu(K_n)>0$ and
$K_{n+1}\subsetneq K_n.$ This contradicts the Noetherianity assumption.

Now we prove the claim.  First assume that there is a maximal element $o$ of $K.$ Then $U:=K\setminus \{o\}$ is open and $\mu(U)>0.$
Since $\mu$ is a Radon measure, there is a compact subset $E\subseteq U$ such that $\mu(E)>0.$
For every $x\in E$, since $x\not\geq o$, there is a clopen $K_x$ of $X$ such that $x\in K_x$ but $o\not\in K_x$. After replacing $K_x$ by $K_x\cap K$, we may assume that $K_x\subseteq K.$ Then $K_x\subsetneq K.$
Since $E$ is compact, there is a finite subset $F\subseteq E$ such that $E\subseteq \cup_{x\in F}K_x.$
So there is $x\in F$, such that $\mu(K_x)>0.$ 

Now we may assume that for every $x\in K$, there is $y_x\in K$ such that $y_x> x.$
Then there is a clopen $K_x$ of $X$ such that $x\in K_x$ but $y_x\not\in K_x$. After replacing $K_x$ by $K_x\cap K$, we may assume that $K_x\subseteq K.$ Then $K_x\subsetneq K.$ Since $K$ is compact, there is a finite subset $F\subseteq K$ such that $K= \cup_{x\in F}K_x.$
So there is $x\in K$, such that $\mu(K_x)>0.$
\end{proof}

This result yields the following estimates on entropy.

\begin{theorem}\label{thm:priestley-entropypartial}
Let $f\colon X\setminus I(f) \to X$ be any partially continuous self-map of a Noetherian Priestley space such that $I(f)$ is clopen. Then we have $\htop(f)=0$.

In particular, the topological entropy of any continuous self-map of a Noetherian Priestley space is $0$.
\end{theorem}

Let us insist on the fact that we did not assume $f$ to be a Priestley map in the previous theorem.
\proof 
We first prove the statement when $f$ is continuous. Since $X$ is compact, the variational principle applies (Theorem~\ref{thm:variational}). As any ergodic measure $\mu$ is supported on a periodic cycle by Theorem~\ref{thm:priestley-entropy}, we get $\htop(f)=0$ as required.

Suppose now that $f$ is a partially continuous self-map. To reduce to the previous case we proceed as follows. 
Consider the poset $X_+$ obtained by adding a singleton $\{o\}$ to $X$ so that the inclusion $X \subset X_+$ preserves the order and $o < x$ for all $x\in X$. 
The topology on $X_+$ is the coarsest one such that the inclusion $X\subset X_+$ is continuous and $\{o\}$ is clopen. 
Then $X_+$ is still a Noetherian Priestley space, and $o$ is a minimal element. 

Consider the map $f_+ \colon X_+ \to X_+$ defined by $f_+(x) = f(x)$ if $x\in X\setminus I(f)$ and $f_+(x) =o$ if $x\in I(f) \cup \{o\}$. 
Since  both $\{o\}$ and $I(f)$ are clopen, the map $f_+$ is continuous so that $\htop(f_+)=0$ by what precedes.

Pick any entourage $\cE$ of $X$. Then $\cE_+ = \cE \cup (\{o\} \times \{o\})$ is an entourage of $X_+$ and 
a simple observation shows $R(n,\xi^{\max}(f),\cE)\le R(n,\xi^{\max}(f_+),\cE_+)$. 
It follows that $\htop(f) \le \htop(f_+)=0$.
\endproof

\subsection{Examples of Priestley spaces}
\subsubsection{Totally disconnected compact sets}
Every totally disconnected compact space (i.e. any Stone space) endowed with the trivial order relation ($x\le y$ iff $x=y$) is a Priestley space. 
Any  totally disconnected compact subset $C$ of $[0,1]$ endowed with the usual total order $\le$ of the real line is also a Priestley space. 
One can build more general Priestley spaces by patching together  totally disconnected compact subset of the unit segment 
following a given tree pattern. 

\begin{remark}
There exists continuous maps on a Cantor set with positive entropy (e.g. the shift map on $\{0,1\}^\Z$).
Hence there exist Priestley maps of positive entropy.
\label{RemarkPositiveEntropyPriestley}\end{remark}

\subsubsection{Partially ordered spaces}\label{subsubsectionprdering}
Let $(X,\leq)$ be a partially ordered set.
For every $x\in X$, define $K_x:=\{y\in X, y\leq x\}.$
Let $\tau_{\leq}$ be the topology generated by the sets $K_x, X\setminus K_x$ for all $x\in X.$
For every down-set $K\subseteq X$, we have $K=\cup_{x\in K}K_x$.
So every down-set of $X$ is open with respect to $\tau_{\leq}.$

\begin{remark}\label{remordcomp}
For $x,y\in X$, if $y\not\leq x$, then $x\in K_x$ and $y\not\in K_x$. So $(X,\tau_{\leq},\leq)$ is a Priestley space if and only if $(X,\tau_{\leq})$ is compact.
\end{remark}

\medskip

We say that $(X,\leq)$ satisfies the descending chain condition if for every decreasing sequence 
$x_1\geq\ldots \ge x_n \ge x_{n+1}$ in $X$, 
 there is $N\geq 1$ such that $x_n=x_N$ for every $n\geq N.$

The next proposition is not essential to our discussion. We include it however for the convenience of the reader, as it clarifies the
problem on characterizing posets for which $(X, \tau_{\leq})$ is a Priestley space. 
\begin{proposition}\label{procompactorder}
The following two conditions are equivalent:
\begin{itemize}
\item[(i)] $(X,\tau_{\leq})$ is compact;
\item[(ii)] 
\begin{itemize}
\item
$(X,\leq)$ satisfies the descending chain condition;
\item
there is a finite subset $F\subseteq X$ such that $X=\cup_{x\in F}K_x$; 
\item
 for any $x,y\in X$ there exists a finite set $F(x,y) \subset X$ such that $K_x \cap K_y= \cup_{z\in F(x,y)} K_z$.
\end{itemize}
\end{itemize}
\end{proposition}
\proof 

Assume (i). Let  $x_n \ge x_{n+1}$ be any decreasing sequence of points in $X$. Then $K_\infty:= \cap K_{x_n}$ is a non-empty compact set. It is also a down-set hence it is open. 
The family $X\setminus K_{x_i}$ forms an open cover of the compact set $X\setminus K_{\infty}$. We conclude to the existence of an integer $N\geq 1$ such that $X\setminus K_{\infty}=\cup_{i=1}^{N}(X\setminus K_{x_i}),$
which implies $x_i = x_N$ for all $i\ge N$. This proves $X$ satisfies the descending chain condition.
Since $\{K_x\}_{x\in X}$ forms an open cover of $X$ and $X$ is compact, we have $X=\cup_{x\in F}K_x$ for some finite set $F$.
Finally $\{K_z\}_{z\in K_x \cap K_y}$ is an open cover of $K_x \cap K_y$, and again $K_x \cap K_y= \cup_{z\in F(x,y)} K_z$
for some finite set $F(x,y)$. This completes the proof $(i) \Rightarrow (ii)$.

Assume that (ii) holds. We only need to show that for every $x\in F$, $K_x$ is compact. 
Let $\{U_i\}_{i\in I}$ be any open cover of $K_x$. We proceed by contradiction and assume that it does not admit any finite subfamily covering $K_x$.

For a finite subset $S$ of $X$, set $K_S:=\cap_{s\in S}K_s$ and $K^S:=\cup_{s\in S}K_s$. 
A basis of the topology of $(X,\tau_{\leq})$ is given by clopen of the form
$K_S\setminus K^T$, where $S,T$ are finite subsets of $X$. 
We may thus assume that for each index $i$ there exist finite sets $S_i$ and $T_i$ such that $U_i=K_{S_i}\setminus K^{T_i}$.

Pick any index $i_0\in I$, such that $x\in U_{i_0}=K_{S_{i_0}}\setminus K^{T_{i_0}}$. Observe that we may take $S_{i_0}=\{x\}$ and that $x\not\in T_{i_0}.$ 
Choose any point $x_1\in T_{i_0}$ such that $K_{x_1}$ is not covered by any finite subset of $\{U_i\cap K_{x_1}, i\in I\}$.
By our standing assumption, $K_{x_1} \cap K_x$ is a finite union of sets of the form $K_y$ so that $x_1$ by one of the element $y$
we may suppose that $x_1<x$.

Proceeding by induction, we get an infinite sequence of points $x>x_1>\ldots> x_n$ such that for $n\geq 0$, $K_{x_n}$ is not covered by  any finite subset of $\{U_i\cap K_{x_n}, i\in I\}.$ This contradicts the descending chain condition.
\endproof

\begin{proposition}\label{proordnoetherian}
Assume that 
$(X,\tau_{\leq})$ is compact.
Then $(X,\tau_{\leq},\leq)$ is a Noetherian Priestley space. 
\end{proposition}
\proof
Because $(X,\tau_{\leq})$ is compact, Remark \ref{remordcomp} implies $(X,\tau_{\leq},\leq)$ to be a Priestley space.
Let \[K_1\supseteq K_2\supseteq\dots\] be any decreasing sequence of non-empty closed down-sets in $X$. Then $K_{\infty}:=\cap_{i=1}^{\infty}K_i$ is a closed down-set, hence is clopen.
Note that $X\setminus K_{\infty}=\cup_{i=1}^{\infty}(X\setminus K_i)$.
Since $X\setminus K_{\infty}$ is compact and $X\setminus K_i$ are open, there is $N\geq 1$ such that $X\setminus K_{\infty}=\cup_{i=1}^{N}(X\setminus K_i),$
and the proof is complete. 
\endproof

\subsubsection{Spectra of commutative rings with unity}\label{sec:commutative}
Let $R$ be any commutative ring with unity. Then $\spec R$ endowed with the Zariski topology is a spectral space. 
Recall the definition of the constructible topology from Section~\ref{sec:two-topos}: 
it is the topology $\tau^{\con}$ induced by the product topology when $\spec R$ is viewed as a subset of $\{0,1\}^R$ 
by attaching to any prime ideal $\fA$ its characteristic function $\chi_\fA \colon R\to \{0,1\}$.
The space $(\spec R,\tau^{\con})$ is compact and totally disconnected. 
The triple $(\spec R,\tau^{\con}, \le)$ is in fact a Priestley space where $\le$ is as follows: for $\fA,\fB\in \spec R$, $\fA\leq \fB$ if and only if $\fB\subseteq \fA.$ 

\medskip

Assume that $R$ is Noetherian, then $\tau^{\con}$ equals to $\tau_{\le}$ (as in Section \ref{subsubsectionprdering}). By Proposition \ref{proordnoetherian}, 
 the triple $(\spec R,\tau^{\con}, \le)$ forms a  Noetherian Priestley space. But in general (e.g., when $R=k[T]$ with $k$ an infinite field), $(\spec R,\tau^{\con})$ is not a Noetherian topological space.
 
\subsubsection{The ideal space of a commutative ring with unity}\label{sec:ideal-space-commutative}

The next construction  has been considered in a couple of papers including~\cite{cornulier,fontana}, see also~\cite[2.5.13]{dickmann}.

Let $R$ be any commutative ring with unity. By the discussion in the previous section, the set $\Id^+(R)$ of all ideals including $R$, and the set $\Id(R)$ of all proper ideals of $R$ are closed subsets of  $\{0,1\}^R$ with the product topology. We call this topology the constructible topology and denote it by $\tau^{\con}$.
Note that 
the constructible  topology 
makes $\Id^+(R), \Id(R)$ into compact totally disconnected sets. Define the partial ordering $\le$ on $\Id(R)$ (resp. $\Id^+(R)$) as follows: for $\fA,\fB\in \Id(R)$ (resp. $\Id^+(R)$), $\fA\leq \fB$ if and only if $\fB\subseteq \fA.$ 
Then  $(\Id(R),\tau^{\con},\le)$ and
$(\Id^+(R),\tau^{\con},\le)$ are Priestley spaces. 
\begin{remark}Note that $\spec R$ is a closed subset of $(\Id(R),\tau^{\con})$. Its ordering and its constructible topology $\tau^{\con}$ are induced from the ordering and the topology $\tau^{\con}$ of $\Id(R).$
\end{remark}
We can also endow $\Id(R)$ with the topology generated by the subsets of the form $\{\fA\in \Id(R),\, f\not\in \fA\}$ for some $f\in R.$
We call it the Zariski topology and denote it by $\tau^{\zar}$.
\begin{lemma}
The space $(\Id(R),\tau^{\zar})$ is quasi-compact and  $\tau^{\con}$ is finer than $\tau^{\zar}$.
\end{lemma}
It follows from~~\cite[Proposition~2.1]{fontana} that  $(\Id(R),\tau^{\zar})$ is in fact a spectral space. We provide a proof of the lemma for the reader's convenience.
\begin{proof}
The fact that $\tau^{\con}$ is finer than $\tau^{\zar}$ is clear. 
To prove quasi-compactness, we take a collection of closed sets
$F_i$ of the form $F_i =  \{ \fA\in \Id(R),  G_i \subset \fA\}$ for
some finite set $G_i \subset R$, and suppose that $\bigcap_I F_i = \emptyset$. 
The ideal generated by $\bigcup_I G_i$ is necessarily equal to $R$ since $\bigcap_I F_i = \emptyset$,
hence there exists a finite subset $J\subset I$
for which the ideal generated by $\bigcup_J G_i$ contains $1$. 
It follows that $\bigcap_J F_i = \emptyset$ as required. 
\end{proof}
\begin{remark}\label{remzaridr}
The restriction of the Zariski topology of $\Id(R)$ on $\spec R$ is the usual Zariski topology on $\spec R.$
\end{remark}

\begin{remark}
When $R$ does not contain a unit, then $\Id(R)$ may not be quasi-compact, e.g., when $R=k^{\circ\circ}$
and $k$ is a non-discrete non-Archimedean metrized field.
\end{remark}

Observe that $\Id^+(R)$ is the union of $\Id(R)$ and an isolated and minimal point corresponding to $R$.

\begin{lemma}\label{lemnoeidnoepri}
Suppose $R$ is Noetherian commutative ring with unit. Then both spaces $(\Id^+(R),\tau^{\con},\le)$ and $(\Id(R),\tau^{\con},\le)$ are Noetherian Priestley spaces. 
\end{lemma}

\proof
We only prove this lemma for $\Id(R)$. Let us first check that $\tau^{\con}=\tau_{\le}$.
On the one hand $\tau^{\con}$ is the topology generated by the subsets of the form: 
\begin{equation}\label{eq:basis-topo-id}
\{\fA\in \Id(R) , \,  f_1,\dots, f_m\in \fA, g_1,\dots, g_n\not\in \fA\}~.
\end{equation}
On the other hand, $\tau_{\le}$ is generated by subsets of the form: 
\[\{\fA\in \Id(R) , \,  \fA \supset \fA_1\} \text{ or } \{\fA\in \Id(R) , \,  \fA \not\supset \fA_1\}~.\]
Since $R$ is Noetherian, any ideal is finitely generated and we see that 
$\tau^{\con}=\tau_{\le}$.

Now since $R$ is Noetherian, $\Id(R)$ satisfies the descending chain condition. 
Moreover $(\Id(R), \tau^{\con})$ is compact hence  $(\Id(R), \tau_{\le})$ too and we conclude using
Proposition \ref{proordnoetherian}.
\endproof

\begin{example} Pick any discrete valued  non-Archimedean metrized field $(k,|\cdot|)$, and set $R=\spec k^\circ$. 
The map $\fA\mapsto 1-\max\{|f|, f\in \fA\}$ defines an order-increasing isomorphism from $\Id(R)$ onto its image $|R|$ in $[0,1]$
endowed with the canonical order relation, and $\Id(R)$ is a Noetherian Priestley space.

When the norm is non discrete, then $\Id(R)$ is not a Noetherian topological space. 
\end{example}

\subsection{Proof of Theorem~\ref{thm:entropy zar et con}}
Let $X$ be any  Noetherian scheme (in particular $X$ is quasi-compact). The constructible topology  $\tau^{\con}$  on $X$ can be characterized as follows. 
Introduce the ordering $x\leq y$ iff $x$ is contained in the Zariski closure of $\{y\}.$  Then  $\tau^{\con}$ coincides with the topology 
$\tau_{\leq}$ induced by the ordering $\le$ as in Section~\ref{subsubsectionprdering}. 
By Proposition \ref{proordnoetherian}, $(X,\tau^{\con},\leq)$ is a Noetherian Priestley space. 

Let now $f\colon X \dashrightarrow X$ be any dominant rational map. It induces a partially continuous self-map $f\colon X\setminus I(f)\to X$ with $I(f)=\Ind(f).$
Since $I(f)$ is a clopen in $(X,\tau^{\con})$, we get that $\htop(f,\tau^{\con})=0$ by Theorem~\ref{thm:priestley-entropypartial}.
\hfill
$\square$

%%%%%%%%%%%%%%%%%%%%%%%%%%%%%%%%%%%%%%%%%%%%%%%%

\section{Entropy of endomorphisms of affinoid spaces}\label{sec:entropyaffinoid}

Denote by $k$ a complete valued field endowed with a non-Archimedean norm $|\cdot |.$ 
Let $A$ be any $k$-affinoid algebra and denote by $\cM(A)$ its Berkovich spectrum (see below for details).
This is a compact space.  A morphism $f\colon \cM(A) \to \cM(A)$ is by definition one induced by a bounded
$k$-algebra morphism from $A$ to itself.

The aim of this section is to prove Theorem~\ref{thm:entropy-affinoid} that we recall for the convience of the reader.

\affinoid*

Note that this theorem is strictly speaking not a preliminary step 
of Theorem~\ref{thm:entropy good reduction}. However very similar ideas appear in both proofs and the affinoid case 
presents much less technicalities. 
Section~\ref{sec:entropyaffinoid} should thus be thought of as a warm-up for the proof of Theorem~\ref{thm:entropy good reduction} given in the next section.

\subsection{Affinoid spaces}\label{sec:def-affinoid}
Our basic references are~\cite{BGR,berkovich}.

As in Section~\ref{sec:anal-kspace},  we write $k^{\circ}:=\{|z|\leq 1\}$, $k^{\circ\circ}:=\{|z|< 1\}$, and $\tilde{k}:=k^{\circ}/k^{\circ\circ}$. 
For every $\e\in (0,1]$, we let $\fm_\e := \{|z|<\e\}$. It is an ideal of $k^{\circ}$ which is prime only if $\e=1$ when $k$ is non trivially valued.

For any $r= (r_1, \cdots, r_d) \in (\R^*_+)^d$, consider the Tate algebra \[k\{r^{-1}T\}= \left\{\sum_{|I|=d} a_I T^I ,\, |a_I| r^{I}\to 0\right\}.\] 
The multiplicative norm $|\sum a_I T^I| := \sup_I |a_I| r^{I}$ turns it into a Banach $k$-algebra. 
An affinoid algebra $A$ is a Banach $k$-algebra for which there exists an epimorphism
$\alpha\colon k\{r^{-1}T\} \to A$ such that the norm on $A$ is equivalent to the residue norm on $k\{r^{-1}T\} /\ker (\alpha)$. When such an epimorphism exists with $r=1$, then we say that $A$ is strictly affinoid. 

The Berkovich spectrum $\cM(A)$ of an affinoid algebra is the set of bounded multiplicative semi-norms on $A$ endowed with the product topology. 
It is a compact set. It can be equipped with a canonical structure sheaf

The spectral norm on an affinoid algebra is defined by $\rho(a) := \lim_n \| a^n \|^{1/n}$, for $a\in A$. It is a power-multiplicative norm
which is equivalent to the original norm when $A$ is reduced. When $A$ is strictly affinoid, then 
$\rho(a)= \sup_{x\in \cM(A)} |a(x)|$, and we write in this case $\rho(a)= \| a\|_{\sup}$. 

Define $A^\circ = \{ a\in A , \,  \rho(a) \le 1\}$ and 
$A^{\circ\circ} = \{ a\in A , \,  \rho(a) <1\}$. Then $\tilde A := A^\circ/A^{\circ\circ}$ is a finitely generated $\tilde k$-algebra, 
and we have a canonical reduction map $\red_A\colon \cM(A) \to \spec \tilde A$ sending a multiplicative semi-norm $x$
to the prime ideal $\red_A(x) :=\{a \in A^\circ, \,   |a(x)| <1\}$. By definition, this reduction map is anti-continuous. For example, if $\tilde{f}\in \tilde{A}$ for some $f\in A$, then $\red_A^{-1}(\tilde{f})=\{x\in \cM(A):~|f(x)|<1\}$ (see e.g.  \cite[Lemma 2.4.1]{berkovich}). 

Affinoid algebras form a category whose morphisms are given by bounded $k$-algebra morphisms. 
Given any two affinoid algebras $A$ and $B$,  a map $f \colon \cM(A) \to \cM(B)$ is said to be analytic when it is induced by a bounded morphism
 $f^* \colon B \to A$.  Any analytic map is continuous, and contracts the spectral norm $\rho(f^* b) \le \rho(b)$  so that we have a canonical commutative diagram 
\[\xymatrix{
\cM(A)  \ar[d]^{\red_A} \ar[r]^f & \cM(B) \ar[d]^{\red_B}\\
\spec \tilde A  \ar[r]^{\tilde f} & \spec \tilde B}
\]

Finally, let $K/k$ be any complete metrized field extension. Then for any affinoid $k$-algebra $A$, we may consider the complete tensor product
$A_K:= A \hat{\otimes}_k K$. This is an affinoid $K$-algebra and the following lemma holds (see e.g. \cite{Jonsson2016}):
\begin{lemma}\label{lemsurjectivebasechange}The morphism $\pi_K\colon\cM(A_K)\to \cM(A)$ is surjective.
\end{lemma}
\proof
For any point $x\in \cM(A)$,  the fiber  $\pi_K^{-1}(x)$ in $\cM(A_K)$ can be canonically identified to $\cM(\cH(x)\hotimes_{k} K).$
By \cite[Section 3, Theorem 1]{Gruson1966}, the natural morphism $\cH(x)\otimes_kK\to \cH(x)\hotimes K$ is injective, so that $\cH(x)\hotimes_{k} K\neq \{0\}.$
It follows from \cite[Theorem 1.2.1]{berkovich} that $\pi_K^{-1}(x)=\cM(\cH(x)\hotimes_{k} K)\neq\emptyset$, which concludes the proof.
\endproof

%%%%%%%%%%%%%%%%%%%%%%%%%%%

\subsection{Zero entropy beyond Noetherian spaces}

The aim of this section is to partially extend results from Section~\ref{sec:noetherian-priestley} to situations where the underlying space is not supposed to be Noetherian. 
This will be used in the proof of Theorem~\ref{thm:entropy-affinoid}.
More precisely we shall prove:
\begin{theorem}\label{thmentropyfinitetype}
Let $R$ be a commutative ring with unit, and $A$ be any finitely generated $R$-algebra. 
Pick any endomorphism of $R$-algebra  $F\colon A\to A$.
Then \[\htop(f)=0\]
where $f$ denotes  the map induced by $F$ either on $\Id(A)$ or on $\spec A$.
\end{theorem}

\begin{example}
Let $k$ be any field.
Let $A$ be the $k$-algebra of locally constant functions on $X:=\{0,1\}^\N$ taking values in $k.$ 
Any ideal of $A$ is of the form $\mathcal{I}(K)= \{ g|_K \equiv 0\}$ where $K$ is a compact set of $X$ endowed with the product topology\footnote{This follows for instance from Lemma~\ref{lemdirectlimring} below and the fact that for any finite set $Y$, any ideals of the ring $k^Y$ is given by the set of functions vanishing on some subset $Y'$ of $Y$.}. 
It follows that the map $\mathcal{I}\colon X \to \spec A$ sending $x$ to $\mathcal{I}(x)$ is a bijection. Since the topology on $\spec A$ is generated by
sets of the form $\{I \in \spec A, g \notin I\}$ for some $g\in A$, the map $\mathcal{I}$ is also continuous, and since $X$ is compact $\mathcal{I}$ induces a homeomorphism. 

Consider the shift map  $\sigma(a_0,a_1,\cdots):= (a_1,\cdots)$. It is a continuous self-map of $X$ such that 
$F(g) :=  g \circ \sigma$ is an endomorphism of $A$. The map $f$ induced by $F$ on $\spec A$ is conjugated to 
$\sigma$ by $\mathcal{I}$ since
$F^{-1}(\mathcal{I}(x)) = \mathcal{I}(\sigma(x))$ so that the topological entropy of $f$ is positive. 
This proves that the assumption on $A$ to be finitely generated is necessary in the previous theorem. 
\end{example}

Before we give a proof of this result, we need to recall some facts on inductive limits of sets and rings. 

Let $\pi_{ij}:X_i\to X_j$, $i\ge j$, be an inverse system of compact topological spaces indexed by an inductive set $I$. Then the projective limit $X:=\projlim_{i\in I} X_i$
endowed with the product topology is compact. The next lemma is standard in the literature. 

\begin{lemma}\label{leminverselimitentropy}
Suppose we are given for each $i\in I$ a continuous map $f_i\colon X_i\to X_i$ such that $\pi_{ij} \circ f_i = f_j \circ \pi_{ij}$ for all $i\geq j$.
Denote by $f\colon X \to X$ be the induced continuous map on the projective limit.
Then we have $\htop(f)\leq \limsup_{i\in I}\htop(f_i),$ and the equality holds if $\pi_{ij}$ is surjective for every $i\geq j$.
\end{lemma}
Two remarks are in order.
\begin{remark}
Suppose $X_i$ is a Noetherian Priestley space for each $i$. Then the previous lemma implies $\htop(f)=0$
whereas in general $X$ is not Noetherian. 
\end{remark}

\begin{remark}
The fact that $f$ is the limit of self-maps of $X_i$ is a very strong assumption.  
The topological entropy of a continuous self-map of $X$ might be positive. 
Take for instance $I=\N$, $X_i = \Z/(p^i)$ so that $X= \Z_p$. 
The map defined in the $p$-adic expansion by the formula $f(\sum_{i\geq 0}a_ip^i):=\sum_{i\geq 1}a_ip^{i-1}$ has entropy $\log p>0$. 
Indeed, this map is topologically conjugated to the one-sided shift on $p$ symbols, whose entropy is calculated in, e.g.,~\cite[p.178]{walters}.
\end{remark}

\proof
For every $i\in I$, denote by $\pi_i\colon X\to X_i$ the natural map.
Pick $\epsilon>0$ and an index $i_0$ large such that 
$\htop(f_i)\leq \limsup_{i\in I}\htop(f_i) + \epsilon$ for all $i \ge i_0$.

Let $\fU$ be a finite open cover of $X$. For every $x\in X$, there is an index $i(x)\in I$ and an open subset $V_{x}\subseteq X_{i(x)}$ such that $x\in \pi_{i(x)}^{-1}(V_x)$ and $\pi_{i(x)}^{-1}(V_x)$ is contained in some element of $\fU.$
Since $X$ is compact, one can find an index $i_1 \ge i_0$ and a finite open cover $\fV$ of $X_{i_1}$ such that $\pi_{i_0}^{-1}(\fV)$ refines $\fU.$
We obtain the series of inequalities:
\[\htop(f,\fU)\leq \htop(f,\pi_{i_1}^{-1}(\fV))\leq \htop(f_{i_1},\fV)\leq \htop(f_{i_1})\leq \limsup_{i\in I}\htop(f_i) + \epsilon,\]
Letting $\epsilon \to0$, we get $\htop(f)\leq \limsup_{i\in I}\htop(f_i).$

If $\pi_{ij}$ is surjective for every $i\geq j,$ then $\pi_i$ is surjective for every $i\in I.$ It follows that 
$\htop(f)\geq \htop(f_i)$ for every $i\in I$, and $\htop(f)= \limsup_{i\in I}\htop(f_i).$
\endproof

For the sake of the reader, we record the following simple observation. 
Recall  the constructible  and Zariski topologies on $\Id(R)$ were defined in Section~\ref{sec:ideal-space-commutative}.
\begin{lemma}\label{lemdirectlimring}
Let $\sigma_{ij}\colon A_i\to A_j$, $i\le j$, be a directed system of rings (with unit) indexed by a set $I$, and 
set $A:=\injlim_{i\in I}A_i.$ 

Then  we have canonical homeomorphisms  \[\Id(A)\simeq\projlim_{i\in I}\Id(A_i) \text{ and } \spec A\simeq\projlim_{i\in I}\spec A_i~,\]
where the ideal space is endowed either with the constructible or the Zariski topology, and the spectrum is endowed either with the constructible or the Zariski topology.
\end{lemma}
\proof 
We only do the proof for $\Id(A)$ with the constructible topologys. The proofs in all other cases are similar.

Denote by $\sigma_i \colon A_i\to A$ the natural morphism. 
Pick any proper ideal $\fA \in \Id(A)$. Then $\fA_i:= \sigma_i^{-1}(\fA)$ is a proper  ideal in $A_i$
and $\sigma_{ij}^{-1}(\fA_j) = \fA_i$ for all $i \le j$ since $\sigma_i= \sigma_j \circ \sigma_{ij}$.
We thus get a canonical map $\phi\colon \Id(A)\to \projlim_{i\in I}\Id(A_i)$ which is continuous.
Indeed for any finitely many elements $f_1, \ldots, f_n$, $g_1, \ldots, g_m \in A_i$, the
set of proper  ideals $\fA \in \Id(A)$ such that $f_k \in \sigma_i^{-1}(\fA)$ and $g_l \notin \sigma_i^{-1}(\fA)$
is equal to the set $\fA \in \Id(A)$ such that $\sigma_i(f_k) \in \fA$ and $\sigma_i(g_l) \notin\fA$. 

Conversely, pick any family of proper  ideals $\fA_i\in \Id(A_i)$ such that $\sigma_{ij}^{-1}(\fA_j) = \fA_i$ for all $i \le j$. 
Then define $\fA$ as the set of elements $\sigma_i(f), f\in \fA_i, i\in I.$ Easy to see that $\fA$ is a proper ideal of $A$.
We claim that: $\fA_i=\sigma_i^{-1}(\fA).$
We now prove the claim.
Let $f'\in \fA_i$ and $f:=\sigma_i(f')$.
If $f'\in \fA_i$, it is clear that $f\in \fA$. On the other hand, if $f\in \fA$, then there is $j\in I$ such that $f=\sigma_{j}(f'')$ and $f''\in \fA_j.$
There is $l\in I$ with $l\geq i,j$ such that $\sigma_{il}(f')=\sigma_{jl}(f'').$ Then $\sigma_{jl}(f'')\in \sigma_{jl}(\fA_j)\subseteq \fA_l$, hence 
$f'\in \sigma_{il}^{-1}(\fA_l)=\fA_i,$ which concludes the claim.
Define $\psi:\projlim_{i\in I}\Id(A_i)\to \Id(A)$ sending $(\fA_i)_{i\in I}$ to $\fA.$

The claim implies that $\phi\circ \psi=\id.$ For $\fA \in \Id(A)$, set $\phi(\fA):=(\fA_i)_{i\in I}.$
It is clear that $\cup_{i\in I}\sigma_i(\fA_i)=\fA,$ hence $\psi\circ \phi=\id.$
For any finitely many elements $f_1, \ldots, f_n$, $g_1, \ldots, g_m \in A$, there is $i\in I$ such that 
there are $f_1', \ldots, f_n'$, $g_1', \ldots, g_m' \in A_i$ such that $f_j=\sigma_i(f_j'), j=1,\dots,n$ and $g_j=\sigma_i(g_j'), j=1,\dots,m.$
The claim implies that 
\begin{align*}
\psi^{-1}\left\{ \fA\in \Id(A),\,\, \right.&\left.f_1, \ldots, f_n\in \fA, g_1, \ldots, g_m \not\in \fA \right\}
= \\
&\left\{(\fA_j)_{j\in I}\in \projlim_{i\in I}\Id(A_i) ,\,\,  f_1', \ldots, f_n'\in \fA_i, g_1', \ldots, g_m' \not\in \fA_i\right\},
\end{align*}
So $\psi$ is continuous. This concludes the proof.
\endproof

\begin{proof}[Proof of Theorem~\ref{thmentropyfinitetype}]
Since $(\spec A,\tau^{\zar})$ is coarser than  $(\spec A,\tau^{\con})$
and the latter  is a closed subset of $\Id(A)$, we are reduced to prove
$\htop(f)=0$ for the map $f\colon \Id(A) \to \Id(A)$ induced by the endomorphism $F$ of the $R$-algebra $A$. 

Assume that $A$ is generated by $t_1,\dots,t_r.$
Write $F(t_i)=\sum_{I\in \N^r}a_{i,I}t^I, i=1,\dots, r$ as polynomials in $t_1,\dots,t_r.$ (hence, most of the coefficients $a_{i,I}$'s are 0). 
The subring $R_\star$ of $R$ which is generated by the finite set $\{a_{i,J}\}_{i\in \{1, \ldots, r\}, J\in \N^r}$ is then Noetherian.

Consider the set $\cN$ of all finitely generated (hence Noetherian) subrings of $R$ containing $R_\star$ ordered by inclusion. 
For any $S \in \cN$, let $A(S)$ be the (Noetherian) subring of $A$ generated by $S$ and $t_1,\dots,t_r.$
Observe that the direct limit $\injlim_{S\in \cN} A(S)$ is canonically isomorphic to $A$. 
By Lemma~\ref{lemdirectlimring}, $\Id(A)$ can be identified with $\projlim_{S\in \cN} \Id(A(S))$. 

But $F(A(S)) \subset A(S)$ since $A(S) \supset R_\star[t_1, \ldots , t_r]$, hence for any $S\in \cN$
we get a continuous map $f_S\colon \Id(A(S)) \to \Id(A(S))$, 
and for any two subrings $S' \subset S$ the following diagram is commutative:
\[\xymatrix{
\Id(A(S)) \ar[d] \ar[r]^{f_S} &\Id(A(S))\ar[d]\\
\Id(A(S')) \ar[r]^{f_{S'}} &\Id(A(S'))}
\]
where the vertical arrows are defined by sending an ideal in $A(S)$ to its intersection with $A(S')$. 
We may thus apply Lemma~\ref{leminverselimitentropy}. By Theorem~\ref{thm:priestley-entropypartial}, 
we have $\htop(f_S) =0$ for all $S\in \cN$, hence the topological entropy of the map induced by  $F$
on $\projlim_{S\in \cN} \Id(A(S))$ is zero. This proves $\htop(f)=0$. 
\end{proof}

%%%%%%%%%%%%%%%%%%%%%%%%%%%

\subsection{Proof of Theorem \ref{thm:entropy-affinoid}}
Recall that $A$ is a  $k$-affinoid algebra, $F \colon A \to A$ is a bounded endomorphism of $k$-algebras, and
$f \colon \cM(A) \to \cM(A)$ is the induced map on the Berkovich spectrum of $A$.
We want to prove that $\htop(f)=0$.

We first observe that replacing $A$ by the quotient by its nilradical does not change its spectrum so that we may assume $A$ is reduced. 
Pick any algebraically closed complete field extension $K/k$, such that $|K^*|$ is dense in $(0,+\infty)$ and $A\hotimes K$ is strict (see~\cite[p.22]{berkovich} for the existence of such a field).
Since by Lemma~\ref{lemsurjectivebasechange} the canonical $\cM(A_K) \to \cM(A)$ is surjective, we have
$\htop(f)\leq \htop (f_K)$. It is thus sufficient to prove $\htop(f_K)=0$. 
We may (and shall) thus assume that $k$ is algebraically closed, $A$ is a strict reduced $k$-affinoid algebra and $|k^*|$ is dense. 

By~\cite[Theorem 6.4.3/1]{BGR}, we may find
a distinguished epimorphism $k\{T\}\to A$, with $T= (t_1, \cdots , t_d)$. Denote by $\|\cdot\|$ the Gauss norm on $k\{T\}$, i.e. $||\sum _ja_jT^j||=\max _j|a_j|$.  
Since  $F\colon A \to A$ is a bounded endomorphism, we have $\| F(h)\|_{\sup} \le \| h \|_{\sup}$, hence $F$
 can be lifted to a bounded
endomorphism $G \colon k\{T\} \to k\{T\}$ such that $\| G(f) \| \le \| f \|$. Since 
we have a closed embedding $\cM(A)\hookrightarrow \bar{\D}^d (0,1) := \cM(k\{T\})$, 
we are reduced to proving $\htop(g) =0$ where $g \colon  \bar{\D}^d (0,1) \to  \bar{\D}^d (0,1)$ is the map induced by $G$.

For any $\e \in (0,1]$, consider the ideal $\fm_\e := \{ |z| < \e \} \subset k^\circ$, and the ring
$k_\e:=k^{\circ}/\fm_\e$. Observe that  (see Lemma \ref{lem:fg}) $A_\e:=k^{\circ}\{T\}/\fm_\e\{T\}=k_\e[T]$ is a finitely generated $k_\e$-algebra. 
We have a canonical reduction map $\red_\e \colon \bar{\D}^d (0,1) \to \Id(A_\e)$ which sends a multiplicative 
semi-norm $x \in \cM(k\{T\})$ to the image of the ideal $\{f\in k^\circ\{T\} , \,  |f(x)| < \e\}$ in $A_\e$ (note that $\{f\in k^\circ\{T\} , \,  |f(x)| < \e\}$ contains $\fm_\e\{T\}$).

Note that given any $f \in k^\circ\{T\}$, we have (for the case $\e =1$ see \cite[Lemma 2.4.1]{berkovich})
\begin{align*}
\red_\e^{-1}\{ \fA \in \Id(A_\e) , \,  \tilde{f} \in \fA\} & = \{ |f| < \e \}
\\
\red_\e^{-1}\{ \fA \in \Id(A_\e) , \,  \tilde{f} \notin \fA\} & = \{ |f| \ge \e \}
\end{align*}where $\tilde{f}$ denotes the image of $f$ in $A_\e$.
It follows that $\red_\e$ is not continuous. 
However, as the usual reduction map $\red_1$, $\red_\e$ is anti-continuous with respect to  the Zariski topology on $\Id(A_r)$ defined in Section~\ref{sec:ideal-space-commutative}.

The map $G$ is a bounded $k$-algebra homomorphism which maps $k^\circ \{T\}$ to itself so that
$G  (\fm_r\{T\}) \subset \fm_r\{T\}$, hence $G$ descends to a continuous map $g_\e\colon \Id(A_\e)\to \Id(A_\e)$
satisfying $\red_\e \circ g = g_\e \circ \red_\e$.

\begin{lemma}\label{lem:subcoverAr}
Let $\fU$ be a finite open cover of $\bar{\D}^d (0,1)$. Then for every sufficiently small $\e\in |k^*|$,
there is a finite constructible (hence clopen) cover $\fV$ of $\Id(A_\e)$ such that 
$\{\red_\e^{-1}(U) , \,  U\in \fV\}$ refines $\fU$.
\end{lemma}
Assume Lemma~\ref{lem:subcoverAr}. Pick any finite open cover $\fU$ and choose an open cover $\fV$ as in the previous lemma. Then we have
$h(g,\fU) \le h(g_\e, \fV)$. By Corollary \ref{thmentropyfinitetype} we have $\htop(g_\e)=0$, which concludes the proof of Theorem \ref{thm:entropy-affinoid}.
\endproof

\begin{proof}[Proof of Lemma~\ref{lem:subcoverAr}]
To any choice of elements $f_i, g_j \in k^\circ\{T\}$, and $a_i,b_j \in k^*$
we attach the set:
\begin{equation}\label{eq:superform}
U(f<a , \,  g\ge b):=
\bigcap_{i=1}^m \left\{ |f_i |< |a_i|\right\}
\bigcap_{j=1}^n \left\{ |g_j | \ge |b_j|\right\}
~.\end{equation}
Choose any $\e = |c| \in |k^*|$ smaller than
$\min \{|a_i|, |b_j|\}$, and observe that $f'_i := \frac{c}{a_i}f_i$, $g'_j := \frac{c}{b_j}g_j$ all belong to $k^\circ \{T\}$.
It follows that we have $U(f<a , \,  g\ge b) = \red_\e^{-1}(V)$
where
$V$ is the open set of $\Id(A_\e)$ given by 
\[V:=
 \bigcap_{i=1}^m  \left\{\fA \in \Id(A_\e) , \,  \widetilde{f'_i} \in \fA\right\}
\bigcap_{j=1}^n \left\{\fA \in \Id(A_\e) , \,  \widetilde{g'_j} \notin \fA\right\}
~.\]
Now pick any finite open cover $\fU$ of $\bar{\D}^d (0,1)$. For each $x$, choose a neighborhood $U_x$
of the form~\eqref{eq:superform} which is contained in an element of $\fU$ (the existence of such a $U_x$ follows from \cite[Remarks 2.2.2]{berkovich}). 
By compactness of $\bar{\D}^d (0,1)$ we may extract a finite subcover $U_{x_1}, \ldots, U_{x_n}$ which refines $\fU$. 

By our previous argument, we may find $\e>0$ so small so that for each $i$ we have $U_{x_i}=  \red_\e^{-1}(V_i)$ for some open subset $V_i$ of $\Id(A_\e)$. 
This completes the proof.
\end{proof}

%%%%%%%%%%%%%%%%%%%%%%%%%%%%%%%%%%%%

\section{Endomorphisms having good reduction}\label{sec:endo-goodred}
Our aim is to globalize the arguments of the previous section when a self-map of an affinoid domain is replaced by an endomorphism of 
a $k^\circ$-model of a projective variety. This will prove Theorem~\ref{thm:entropy good reduction}. 

To that end, we develop in \S\ref{sec:eps-red-aff} the theory of $\e$-reduction of any strictly affinoid domain defined
over an algebraically closed complete metrized field $(k,|\cdot|)$. We then extend this $\e$-reduction functor
to projective models $\fX$ over $k^\circ$ in \S\ref{sec:formal-aff}--\ref{sec:func}. 
The proof of Theorem~\ref{thm:entropy good reduction} is given in \S\ref{sec:proofB}. It is based on an argument of absolute Noetherian approximation, that we develop with some details.  

\subsection{$\e$-reduction of strictly affinoid domains} \label{sec:eps-red-aff}
Let $A$ be any strictly affinoid $k$-algebra. Recall the definition of the spectral norm $\rho(a) := \lim_n \|a^n \|^{1/n}$, and
that we set $A^\circ := \{\rho \le 1\}$. 
\begin{lemma}\label{lem:fg}
For any strictly affinoid $k$-algebra $A$ and for any $1\ge\e>0$, 
then \[A_\e := A^\circ/\{ \rho < \e \}\] is a finitely generated $k_\e$-algebra with unit. 
\end{lemma}
\begin{proof}
When $ A= k\{T\}$ is the Tate algebra, then $A_\e =k^{\circ}\{T\}/\fm_\e\{T\}= k_\e[T]$ so that the result is clear. 
When $A$ is reduced, since $k$ is algebraically closed, then it follows from~\cite[\S 6.4.3]{BGR} that there exists an epimorphism
$\alpha\colon k\{T\} \to A$ such that $\rho$ is equal to the residual norm. 
We get a surjective morphism $k_\e[T] \to A_\e$ and $A_\e$ is finitely generated over $k_\e$.
Finally, when $A$ is not reduced, we conclude by observing that $A_\e = B_\e$ where 
$B$ is the quotient of $A$ by its nilradical. 
\end{proof}
\begin{remark}
The previous proof works for any stable fields in the sense of~\cite[\S 3.6]{BGR}. 
When $\e=1$, then it is known that $\tilde A := A_1$ is finitely generated over $\tilde k$
without any assumption on $k$. It remains open whether $A_\e$  is finitely generated over $k_\e$
when $k$ is not necessarily stable. 
\end{remark}
Pick any multiplicative 
semi-norm $x \in \cM(A)$, and set $\fA_\e(x):= \{a\in A^\circ  , \,  |a(x)| < \e\}$ so that $\fA_\e(x) \cap k^\circ = \fm_\e$. 
Since $\fA_\e(x) \supset \{\rho < \e\}$, the ideal $\fA_\e(x)$ projects onto an ideal  $\tilde \fA_\e(x)$ of $A_\e$ 
which satisfies  $\tilde \fA_\e(x) \cap k_\e = (0)$. 
We get a canonical reduction map $\red_\e \colon \cM(A) \to \Id(A_\e)$ sending $x$
to  $\tilde \fA_\e(x)$.

Recall from \S\ref{sec:ideal-space-commutative} that the Zariski topology on  $\Id(A_\e)$ is generated by open sets of the form:
\[U:= \{\fA\in \Id(A_\e) , \,  a\not\in \fA\}~,\]
so that  $\red_\e^{-1}(U) = \{ x\in \cM(A) , \,  |a(x)| \ge \e\}$ is closed. 
It follows that $\red_\e$ is anti-continuous from the affinoid space $\cM(A)$ to  $(\Id(A_\e), \tau^{\zar})$.

\begin{remark}
When $\e=1$, the image of the reduction map $\red_1(A)$ is precisely the spectrum of $\tilde A$, see~\cite[7.1.5/4]{BGR},~\cite[\S 2.4]{berkovich}.
The description of $\red_\e (A)$ is unclear when $\e <1$. In general $\red_\e(A)$ is not closed in $\Id(A_\e)$
endowed either with the Zariski or the constructible topology. 
\end{remark}

\begin{remark}
It is instructive to compute the $\e$-reduction of $\bar\D^1(0,1)=\cM(k\{T\})$. Given any $a\in k^\circ$ and any $r\in[0,1]$, 
denote by $\zeta(a,r)$ the point in  $\bar\D^1(0,1)$ defined by $|P(\zeta(a,r))|= \sup_{\bar B(a,r)} |P|$. 

Direct computations show that $\red_\e(\zeta(a,r))$  is the prime ideal  in $k_\e[T]$ generated by $(T-a)$ if $r<\e$; whereas
for any $r\ge \e$, $\red_\e(\zeta(a,r))$ is not a prime ideal. Indeed, one has 
 $\red_\e(\zeta(a,r))= \{ \sum_i a_i (T-a)^i, \, a_0=0, |a_1| r < \e\}$.
Also, we have $\red_\e(\zeta(a,r))= \red_\e(\zeta(\alpha,\rho))$ iff  either $\zeta(a,r)=\zeta(\alpha,\rho)$, or $r, \rho < \e$ and $|a-\alpha| < \e$.
 
Observe that these computations imply that $\red_\e(\cM(k\{T\}))$ is \emph{not} a closed subset of $\Id(k_\e[T])$ (at least when $k$ is spherically complete).
\end{remark}

The $\e$-reduction map is functorial. For any bounded homomorphism of $k$-algebras $f^*\colon A \to B$ inducing an analytic map $f\colon \cM(B) \to \cM(A)$, 
we have $\rho(f(a))\le \rho(a)$ hence we get a canonical map $f_\e \colon \Id(B_\e) \to \Id(A_\e)$ making the following diagram commutative:
\begin{equation}\label{eq:comm-red-eps}
\xymatrix{
\cM(B)  \ar[r]^{f} \ar[d]_{\red_\e} & \cM(A)  \ar[d]^{\red_\e}
  \\
 \Id(B_\e)  \ar[r]^{f_\e} & \Id(A_\e) 
}
\end{equation}

The map $f_\e$ is continuous for both the Zariski and constructible topologies.
\smallskip

Let us finish this section by the following generalization of Lemma \ref{lem:subcoverAr}.
\begin{lemma}\label{lem:subcoverAra}
Let $\fU$ be a finite open cover of $\cM(A)$. Then for every sufficiently small $\e\in |k^*|$, there is
 a finite constructible cover $\fV$ of $(\Id(A_\e),\tau^{\con})$ such that 
$\{\red_\e^{-1}(U) , \,  U\in \fV\}$ refines $\fU$.
\end{lemma}

\begin{proof}
As in the proof of Lemma~\ref{lem:fg}, we may suppose that $A$ is reduced and choose an epimorphism $\alpha^*\colon k\{T\} \to A$ 
such that the spectral norm on $A$ is equal to the residual norm. This induces a closed embedding $\alpha\colon \cM(A)\to \cM(k\{T\})$,
and an epimorphism $k_\e[T] \to A_\e\to0$ of $k_\e$-algebras. 
For each $U\in \fU$ choose an open set $U'$ of $\cM(k\{T\})$ such that 
$U'\cap \cM(A) = U$, and consider the open cover $\fU':=\{U'\}\cup \{\cM(k\{T\})\setminus \cM(A)\}$
of $\cM(k\{T\})$.
By Lemma \ref{lem:subcoverAr}, for  sufficiently small $\e\in |k^*|$,  there exists a finite constructible cover $\fV'$ of 
$(\Id(k_\e[T]),\tau^{\con})$ such that 
$\{\red_\e^{-1}(U) , \,  U\in \fV'\}$ refines $\fU'$.
Since $\alpha^{-1}(\fU') = \fU$ by construction, it follows from~\eqref{eq:comm-red-eps} that 
the cover $\{\red_\e^{-1}(\alpha_\e^{-1}(U)) , \,  U\in \fV'\}$ refines $\fU$. We conclude the proof
by setting $\fV:= \alpha_\e^{-1}(\fV')$.
\end{proof}

\subsection{Primary ideals} \label{sec:red-to-primary}
There is a technical difficulty in globalizing the $\e$-reduction using the sets $\Id(A_\e)$ as above.
We refer to Remark~\ref{rem:technical} below for a discussion on this problem.
In order to properly globalize the $\e$-reduction and to develop smoothly our descent argument,
we will work instead with the set of primary ideals in $A_\e$.
In this section, we discuss briefly some general properties of this set. 

Recall that an ideal $I$ is primary if $fg\in I$ and $f\notin I$ implies $g^n\in I$ for some $n\ge1$. 
Any prime ideal is primary, and the radical of a primary ideal is prime. 

\begin{definition}
Let $R$ be any commutative ring with unit. Then we let ${\rm P}(R)$ be the set of all primary ideals
in $\Id(R)$ and  $\prim(R)$ be its closure (for the constructible topology) 
in $\Id(R)$. 
\end{definition}

The set of primary ideals is a countable intersection of clopen sets. As the examples below show, it is neither closed, nor open, nor dense in general. 
\begin{example}
Let $k$ be an algebraically closed field, and set $R= k[x,y]$.

Then $I =(xy)$ is not a primary ideal, but
the sequence of primary ideals $(x^n,y^n,xy)$ converges to $I$
in $\Id(R)$. This shows that the set of primary ideals is not closed in general. 

Take $I' = (x)$. Then a basis for open neighbourhoods of $I'$ are those of the form $U=\{ J\in \Id(R), f_i \in J, g_j \notin J\}$ (where $f_i$ divides $x$). Each such $U$
contains two maximal ideals $J_1$ and $J_2$ for which the product $J_1\cdot J_2$ is not primary
and belongs to $U$. This proves the set of primary ideals is not open in general.

Consider now $I'' =(xy(x+y-1))$. Then the set $\{J \in \Id(R), xy \notin J, y(x+y-1)\notin J, x(x+y-1)\notin J, xy(x+y-1)\in J\}$
is an open neighborhood of $I''$ but contains no primary ideals. Indeed if $J$ would be a primary ideal lying in this set then we could find
an integer $n>0$ such that $x^n, y^n$ and $(x+y-1)^n$ belong to $J$, a contradiction. 
This shows that the set of primary ideals is not dense in general. 
\end{example}

Let $f$ be any non-nilpotent element in $R$.
Recall that $R[f^{-1}] = R[T]/(1-Tf)$ is an overring of $R$ so that we have two natural maps
$\res \colon \Id(R[f^{-1}]) \to  \Id(R)$ and $\ext \colon \Id(R) \to  \Id^+(R[f^{-1}])$
defined by $\res(I) = I \cap R$ and $\ext(I) = I\cdot R[f^{-1}]$ respectively.  
Beware that if $f^n\in I$ for some $I\in \Id(R)$ then  $\ext(I)=R$.

\begin{lemma} \label{lem:res-prim}
The map $\res\colon \Id(R[f^{-1}]) \to  \Id(R)$
is injective and its image consists of those
ideals $I\in \Id(R)$
such that if $gf^n\in I$ for some $g\in R$ and some $n\ge0$, then $g\in I$. 
\end{lemma}
\begin{proof}
We first prove the second statement. 
Suppose $I = J\cap R$ for some ideal $J$ of $R[f^{-1}]$.
If $gf^n\in I$ for some $g\in R$ and some $n\ge0$, then $g\in J$
since $f$ is invertible in $R[f^{-1}]$. It follows that $g\in I$. 

Conversely suppose $I$ is an ideal in $R$ such that 
$gf^n\in I$ for some $g\in R$ and some $n\ge0$ implies $g\in R$.
We shall prove that $\ext(I) \cap R = I$.
It is clear that $\ext(I) \cap R \supset I$. So pick  $h\in \ext(I) \cap R$ and write $h=g/f^n$ with $g\in I$. 
Then $hf^n \in I$ hence $h\in I$. 

The former arguments imply $\ext \circ \res(J) = J$ for any $J\in R[f^{-1}]$. This implies $\res$ to be injective.
\end{proof}

Observe that $\res(I)$ is a prime (resp. primary) ideal whenever $I$ is prime (resp. primary).
\begin{lemma}\label{lem:car-prim}
A primary ideal $I\subset R$ belongs to 
$\res(\Id(R[f^{-1}]))$ if and only if
$f^n\notin I$ 
for all $n\ge0$.
\end{lemma}
\begin{proof}
Suppose $I = J \cap R$ for some $J\in \Id(R[f^{-1}])$, and $f^n$ belongs to $I$ for some $n>0$.
Then the previous lemma implies $1\in I$, a contradiction. 

Conversely, suppose $I$ is primary
and $f^n\notin I$ 
for all $n\ge0$. Suppose that 
 $gf^n\in I$ for some $g\in R$ and some $n\ge0$.
 If $g\notin I$, then $f^{nm}\in I$ for some $m$
 since $I$ is primary, a contradiction. Hence $g\in I$, and
 the previous lemma implies 
$I = J \cap R$ for some $J\in \Id(R[f^{-1}])$.
\end{proof}

\begin{proposition}\label{prop:cover-ideal}
For any commutative ring $R$ with unit, and any non-nilpotent element $f\in R$, 
$\res\colon\prim(R[f^{-1}]) \to \prim(R)$ induces a homeomorphism onto its image
which is compact in $\prim(R)$.

Moreover, for any finite family of non-nilpotent elements 
$f_1, \cdots, f_m \in R$ generating the unit ideal, 
the set $\{ \res (\prim(R[f_i^{-1}])\}$ forms a finite compact cover of $\prim(R)$.
\end{proposition}

\begin{proof}
By Lemma~\ref{lem:res-prim}, the map $\res$ is injective. Since it is also continuous (for the constructible topology), 
the restriction of $\res$ to any compact subset of $\Id(R[f^{-1}])$ induces a homeomorphism onto its image. 
By construction, $\prim(R[f^{-1}])$ is compact in  $\Id(R[f^{-1}])$, and the first statement follows.

We now prove the second statement. Suppose
$f_1, \cdots, f_m$ generate the unit ideal, and pick any primary ideal $I$. 
Suppose that $I$ does not belong to the union of $\{ \res (\prim(R[f_i^{-1}])\}$. Then, by the previous lemma,  for each index $i$, we have $f_i^{n_i} \in I$ for some integer $n_i$.
Set $n = \max n_i$. Then $f^n_1, \cdots, f^n_m\in I$, but it is  also easy to show that they generate the unit ideal.  
This is a contradiction. 
\end{proof}

\begin{remark}\label{rem:cover-prim}
The previous result is not true when $\prim(R)$ is replaced by $\Id(R)$. Indeed, in general 
$\{ \Id(R[f_i^{-1}])\}$ does not form a cover of $\Id(R)$ if $(f_1, \cdots, f_n)$ generates the unit ideal
(take $R=\Z$, $f_1= 2$ and $f_2=3$ for instance).
\end{remark}

\begin{remark}
To simplify notation, in the sequel we shall always identify $\prim (R[f^{-1}])$ with its image inside $\prim(R)$.
\end{remark}

We conclude this section with the following observations. 

\begin{lemma}\label{lem:primary}
For any strictly affinoid algebra $A$ and any $\e>0$, 
we have $\red_\e(\cM(A)) \subset {\rm P}(A_\e)$.
\end{lemma}
\begin{proof}
Pick any $x\in \cM(A)$. Then $\red_\e(x) = \{ f\in A_\e, |f(x)|<\e \}$ is a primary ideal. 
Indeed suppose $fg \in \red_\e(x)$ and $f \notin \red_\e(x)$. Then 
$ |fg(x)|<\e$ and $ |f(x)|\ge \e$, hence $ |g(x)|<1$. Hence we can find $n>0$ such that 
$ |g^n(x)|<\e$, and $g^n\in  \red_\e(x)$.
\end{proof}

To simplify notation, write 
$\bar{R}_\e(A)$ for the closure (in $\prim(A_\e)$) of
the set $\red_\e(\cM(A))$. 

\begin{proposition}\label{prop:clos-primary}
Let $A$ be any strictly affinoid algebra $A$. 
For any $\e\in |k^*|$, and any $f\in A^\circ$, 
we have 
$\bar{R}_\e(A\langle f^{-1}\rangle)=
\bar{R}_\e(A)\cap \prim(A_\e[f^{-1}])$, and this set is a clopen subset of 
$\bar{R}_\e(A)$.
\end{proposition}
\begin{proof}
Write $\e= |a|$ for some $a \in k^*$. 
Observe first that \[
\red_\e(\cM(A\langle f^{-1}\rangle))=
\red_\e(\cM(A)) \cap \prim(A_\e[f^{-1}])
=
\red_\e(\cM(A)) \cap \Id(A_\e[f^{-1}])
~.\]
We now prove 
\[ 
\red_\e(\cM(A)) \cap \Id(A_\e[f^{-1}])
=
\{ I\in  \red_\e(\cM(A)), af \notin I
\}
~.\]
Suppose $I= \red_\e(x)$ for some $x \in \cM(A)$. 
Observe that 
$f^n \notin I$ for all $n$, if and only if $|f(x)| =1$, and
the latter is equivalent to $|(af)(x)|= \e$. 
By Lemma~\ref{lem:car-prim}, this proves
$I\in \Id(\cM(A\langle f^{-1}\rangle))$ if and only if
$af \notin I$ as required.

Since the set of ideals $I\in \Id(\cM(A)))$ such that $af \notin I$
is clopen, the closure of $\{ I\in  \red_\e(\cM(A)), af \notin I
\}$ equals $\{ I\in \bar{R}_\e(A), af \notin I
\}$
which implies 
\[
\bar{R}_\e(A\langle f^{-1}\rangle)
=
\{ I\in \bar{R}_\e(A), af \notin I
\}
\subset \bar{R}_\e(A)\cap \prim(A_\e[f^{-1}])
.\]
Pick $I\in \bar{R}_\e(A)\cap \prim(A_\e[f^{-1}])$.  Observe that $a\notin I$. (Indeed, 
for any $x\in \cM(A)$ we have $\red_\e(x)\cap k_\e = (0)$. Hence,
for any ideal $J\in\bar{R}_\e(A)$
we have $J \cap k_\e =(0)$, so that $a \notin J$ since $a\in k_\e$.)

Now the set $\{J\in \Id(A), a \notin J\}$ is both open and closed hence
the closure of the set of primary ideals of $A_\e[f^{-1}]$ not containing $a$
is equal to $\prim(A_\e[f^{-1}]) \cap \{J\in \Id(A), a \notin J\}$.
It follows that
\begin{align*}
\prim(A_\e[f^{-1}]) \cap \{J\in \Id(A), a \notin J\} 
&= 
\overline{\{ J\in {\rm P}(A_\e) , f^n \notin J~\text{for all}~n, a \notin J\}}
\\
&\subseteq
\overline{\{ J\in {\rm P}(A_\e) , af \notin J\}}
\end{align*}
and we conclude that $af \notin I$. This proves
$\bar{R}_\e(A)\cap \prim(A_\e[f^{-1}])=\bar{R}_\e(A\langle f^{-1}\rangle)$
and concludes the proof.
\end{proof}

\subsection{Formal affinoid subdomains} \label{sec:formal-aff}
Let $A$ be any strictly affinoid $k$-algebra. An affinoid subdomain $V$ of $\cM(A)$ is a compact subset
$V \subset \cM(A)$ such that there exists a bounded morphism of affinoid $k$-algebras $\alpha \colon A \to A_V$ satisfying
the following universal property. Given any bounded morphism of affinoid $k$-algebras $\beta \colon A \to B$
such that $\beta^* (\cM(B)) \subset V$, then we may write $\beta =  \theta \circ \alpha$ for some unique bounded morphism 
$\theta\colon A_V \to B$.

\smallskip

For any  elements $a, b \in A$ generating the unit ideal, and for any  positive $r, s \in |k^*|$, the 
Weierstrass domain 
$\{x \in \cM(A), \,  |a(x)| \le r\}$, 
the Laurent domain
$\{x \in \cM(A) , \,  |a(x)| \le r, |b(x)| \ge s\}$, 
and
the rational domain
$\{x \in \cM(A) , \,  |a(x)| \le r |b(x)|\}$, 
 are all affinoid subdomains of $\cM(A)$. In each case, it is possible to describe the affinoid algebra of the affinoid subdomain
explicitely in terms of $A$ and the data $a, b, r, s$. For instance, we have
\[
\{x \in \cM(A) , \,  |b(x)| \ge 1\}
= 
\cM\left( A\langle b^{-1}\rangle\right)
\]
where $A\langle b^{-1}\rangle= A\{T\}/(Tb-1)$ is endowed with the residual norm, and the ring
$ A\{T\}:=\{\sum_n a_n T^n , \,  a_n\in A,  |a_n| \to 0\}$  with the Gauss norm.

\smallskip

Following~\cite[\S 4.3]{berkovich}, we say that an affinoid subdomain $V$ of $\cM(A)$ is formal
when the morphism induced at the level of their (standard) reductions $\spec \tilde A_V \to\spec \tilde A$
is an open immersion. 
It follows from~\cite[7.2.6]{BGR} that for any $f\in A^\circ$ such that $\|f\|_{\sup} = 1$, then 
the Laurent domain $V:=\{|f|\ge  1\} =\{|f|=  1\}$ is a formal affinoid subdomain
so that $\tilde A_V = \tilde A[\tilde f^{-1}]$.
Conversely, by~\cite[Satz~1.2]{Bosch77}, any formal affinoid subdomain of $\cM(A)$ is a finite union of domains
of the form $\{|f_i|\ge  1\}$ with $f_i \in A^\circ$.

\begin{lemma}\label{lem:sub-ideal}
For any formal affinoid subdomain $V \subset \cM(A)$ and for any $\e>0$, the canonical 
map $\Id((A_V)_\e) \to \Id(A_\e)$ induces a homeomorphism from $(\prim((A_V)_\e), \tau^{\con})$ 
onto a closed subset of $(\prim(A_\e), \tau^{\con})$.
\end{lemma}

\begin{proof}
Let us first consider the case $V = \{ |f| \ge 1\}$ with $f \in A^\circ$.
By~\cite[7.2.6/2]{BGR} the morphism $\sigma \colon A \to A_V$ satisfies 
$|\sigma(a)|_{\sup} = \lim_{n\to\infty} |f^n a|_{\sup} \le |a|_{\sup}$ so that 
we have a canonical morphism $\sigma_\e \colon A_\e \to (A_V)_\e$. 
Since $f$ is a unit in $A_V^\circ$, we get the following  commutative diagram 
\[\xymatrix{
A_\e  \ar[r]^{\sigma_\e} \ar[dr]_\phi & (A_V)_\e \\
 & A_\e[f_\e ^{-1}] \ar[u]_\imath}
\]
where $f_\e$ denotes the class of $f$ in $A_\e$.
Note that the lemma follows from Proposition~\ref{prop:cover-ideal}
as soon as we can prove that $\imath$ is an isomorphism.

To see that $\imath$ is injective, suppose that 
$\imath  (a f_\e^{-n}) = 0$ with $a\in A_\e$ and $n\ge0$. Then $\imath(a) = 0$ so that $\sigma_\e(a) = 0$. Choose
$b \in A^\circ$ representing $a$. Since $|\sigma(b)|_{\sup} < \e$, we have $|f^m b|_{\sup} < \e$ for some integer $m$, hence
$f_\e^m a= 0$ in $A_\e$, which implies $a= 0$ in $A_\e[f_\e ^{-1}]$, and
$\imath$ is injective. 

Pick any $b\in A^\circ_V$.
Observe that $A[f^{-1}]$ is dense (for the sup-norm) in $A_V = A\langle f^{-1}\rangle$, hence
we can find $a\in A$ and $n\in \N^*$ such that $|f^n b - \sigma(a)|_{\sup} < \e$. But this implies
$\sigma(a) \in A_V^\circ$, and $b = f^{-n} \sigma(a)$ modulo $\{\rho < \e\}$. 
It follows that $\imath$ is surjective, and the claim is proven. 

Suppose now that $V$ is a general formal affinoid subdomain of $\cM(A)$. Then $V$ is a finite union of affinoid subdomains 
$V_i = \{ |f_i| \ge 1\}$ with $f_i\in A^\circ$. 
By what precedes we have $\prim( (A_{V_i})_\e)= \prim( (A_{V})_\e[f_{i,\e}^{-1}])$.

We claim that $\prim( (A_{V_i})_\e)$ covers $\prim((A_V)_\e)$.
To see this, consider the usual reduction map $\red \colon \cM(A_V) \to \spec (\tilde{A}_V)$, and denote by $\tilde{f_i}$ the image of $f_i$ in $\tilde{A}_V$.
Since $V = \cup \{|f_i| \ge1\}$, we obtain 
\[\red^{-1} (\cap \{\tilde{f_i} =0\}) = \cap \{ |f_i | <1 \} = \emptyset\]
so that $\tilde{f_i}$ generates the unit ideal of $\tilde{A}_V$, and therefore $f_i$ generates the unit ideal of $A_V^\circ$, and we conclude using Proposition~\ref{prop:cover-ideal}.

Now the preceding argument shows that the restriction of the natural map \[\sigma_\e^*\colon \prim((A_V)_\e) \to \prim(A_\e)\]
to $\prim( (A_{V_i})_\e)$ is continuous and injective. To conclude the proof it remains to show that $\phi$ is injective on $\prim((A_V)_\e)$.
Pick $x, y \in \prim((A_V)_\e)$ such that $\sigma_\e^*(x) =\sigma_\e^*(y)\in \Id(A_\e)$. Then $x$ and $y$ are two ideals of $(A_V)_\e$ and we may find an index $i$
such that  $f_{i,\e} \notin x$. If $f_{i,\e} \notin x$ and $f_{i,\e} \in y$, then we have $f_{i,\e} \notin \sigma_\e^*(x)$ and $f_{i,\e} \in \sigma_\e^*(y)$ hence $\sigma_\e^*(x) \neq \sigma_\e^*(y)$. 
Hence $f_{i,\e} \notin y$ and $x=y$ since $\phi$ is injective on $\Id( (A_{V_i})_\e)$.
\end{proof}

\subsection{The functor of $\e$-reduction on $k^\circ$-models} \label{sec:func}
Let  $X$ be any projective variety over $k$, and $\fX$ a model of $X$ over $k^\circ$ (in the sense of \S \ref{sec:anal-kspace}).
Recall that there is a reduction map $\red_\fX \colon X^{\an} \to \fX_s$.
In this section, we globalize the construction of  the previous section, and build an $\e$-reduction map associated to $\fX$ for any $0<\e\le 1$. 
 Let us first explain how we construct the space $\fX[\e]$.

 \subsubsection{}
 The first step of the construction is to endow $X^{\an}$ with a canonical formal structure
 in the sense of~\cite{Bosch77} (see also the discussion in~\cite[\S 1]{gubler} or~\cite[\S 6.4]{berkovich}), that is, with a covering by affinoid domains $V_i$ such that
 for each $i,j$, $V_{ij} := V_i \cap V_j$ is a formal affinoid subdomain of $V_i$. 

 To do so we proceed as follows. Since $\fX$ is projective, we may fix once and for all an embedding of $k^\circ$-scheme
 $\imath \colon \fX \to \p^N_{k^\circ}$. In particular, we have an embedding of the special fiber $\imath_s \colon \fX_s \to 
 \p^N_{\tilde{k}}$, and an embedding of the analytifications $\imath_{\an} \colon X^{\an} \to \p^{N,\an}_{k}$. All these morphisms fit into the following commutative diagram:
  \[\xymatrix{
  X^{\an}   \ar[r]^{\imath_{\an}}
 \ar[d]_{\red_\fX}
 & \p^{N,\an}_{k}
  \ar[d]^{\red_{\p^{N}_{k^\circ}}}
\\
 \fX_s
 \ar[r]_{\imath_s}
 &\p^N_{\tilde{k}}
 }
\]
Take homogeneous coordinates $[z_0: \cdots : z_N]$ on $\p^N$, and 
set $U_i :=\{z_i \neq 0\}$. This defines a finite cover $\fU= \{U_i\}$ of $\p^N_{k^\circ}$ by affine subscheme over $k^\circ$, and
we get an induced finite cover $\fV_{s}= \{V_{i,s}\}$ of $\fX_s$ by affine subschemes
given by $V_{i,s}:= \imath_s^{-1} (U_i\cap \p^N_{\tilde{k}})$. 
\begin{lemma}\label{lem:key-formal}
For any $i$, $V_i:=\red_\fX^{-1}(V_{i,s})$ is an affinoid domain in $X^{\an}$.

Moreover, for any principal open affine subscheme $W_s$ of $V_{i,s}$, the 
set $W :=\red_\fX^{-1}(W_s)$ is a formal affinoid subdomain of $V_i$.
\end{lemma}
\begin{proof}
To simplify notation, let $i=0$. In the chart $U_0 =\{z_0\neq0\}$, 
we may write $\fX =\spec \mathcal{A}$ 
where $\mathcal{A}:= k^\circ [z_1, \cdots, z_N]/\mathfrak{I}$
where  $\mathfrak{I}$ is a finitely generated ideal.
Then $\red_\fX^{-1}(V_{0,s})$ is the set of multiplicative semi-norms $x$ on
$\mathcal{A} \otimes_{k^\circ} k$ such that $|f(x)| \le 1$ for all $f \in \mathcal{A}$.

Let $\hat{\mathcal{A}}:= \projlim_{\pi\in k^{\circ\circ}} \mathcal{A}/(\pi)$ 
be the formal completion of $\mathcal{A}$, and
set $\hat{\mathcal{A}}_k := \hat{\mathcal{A}}\otimes_{k^\circ} k$.
One checks that $\hat{\mathcal{A}} = k^\circ\langle z_1, \cdots, z_N \rangle/\mathfrak{I}$, hence 
$\hat{\mathcal{A}}_k = k\langle z_1, \cdots, z_N \rangle/\mathfrak{I}$ is a $k$-algebra which becomes an affinoid algebra
when endowed with the residue norm $|\cdot|_{\mathrm{res}}$ coming from the morphism $ k\langle z_1, \cdots, z_N \rangle \to \hat{\mathcal{A}}_k$. 

Observe that $\mathcal{A} \otimes_{k^\circ} k$ is a dense subring of $\hat{\mathcal{A}}_k$.
Any norm on $\mathcal{A} \otimes_{k^\circ} k$ bounded by $1$ on $\mathcal{A}$
is bounded by  $|\cdot|_{\mathrm{res}}$ hence defines a point  in the Berkovich spectrum of the affinoid algebra $\hat{\mathcal{A}}_k$. Conversely any norm
in this spectrum restricts to a norm on  $\mathcal{A} \otimes_{k^\circ} k$ bounded by $1$ on $\mathcal{A}$.
This proves $\red_\fX^{-1}(V_{0,s})$ is equal to the Berkovich spectrum of $\hat{\mathcal{A}}_k$, and is an affinoid domain. 

Suppose now that $\tilde{f} \in \tilde{k}[V_{i,s}]$ and $W_s = \{\tilde{f} \neq 0\} \subset V_{0,s}$. Let $f$ be a lift of $\tilde{f}$
to $\mathcal{A}$. Then $\red_\fX^{-1}(W_s)= \{ |f|\ge 1\}$ is a formal affinoid subdomain of $V_0$.
\end{proof}

The previous lemma implies that $\fV:= \{V_i\}$ is a finite cover of $X^{\an}$ by affinoid domains such that 
$V_{ij}$ is a formal affinoid subdomain of $V_i$ for all $i$ and $j$, and therefore defines a formal structure on $X^{\an}$. 
Declare that a formal structure $\fV'=\{V'_j\}$ on $X^{\an}$ is equivalent to $\fV$ if
$V'_j \cap V_i$ is a formal affinoid subdomain of $V_i$ for all $i$ and $j$.

We claim that the formal structure defined by $\fU$ on $X^{\an}$
is independent on the choice of an embedding into $\p^N_{k^\circ}$. 
Indeed, choose another embedding  $X^{\an} \subset \p^{N',\an}_{k}$, and consider the corresponding finite covers
$\{V'_{j,s}\}$ of $\fX_s$, and $\{V'_j\}$ of $X^{\an}$ respectively. Note that 
$V'_{j,s} \cap V_{i,s}$ is a affine open subscheme of $V_{i,s}$ for all $i$ and $j$ so that 
again by the previous lemma, $V'_j\cap V_i$ is a formal affinoid subdomain for all $i,j$.

 \subsubsection{}
Let us first recall the construction  of the reduction associated to the formal structure defined by $\fX$ following~\cite{Bosch77,gubler}. 
Recall we are given a finite cover $\fV = \{V_i\}$ by affinoid domains such that for all $i$ and $j$, we have $V_{ij}$ is a formal affinoid
subdomain of $V_i$. 

For each $i$, consider the quotient algebra $A_i:= \mathcal{O}(V_i)^\circ / \mathcal{O}(V_i)^{\circ\circ}$. 
By~\cite[6.3.4/3]{BGR}, this is a reduced $\tilde{k}$-algebra of finite type, hence $\tilde{V}_i := \spec (A_i)$ is an affine variety over $\tilde{k}$. 
Since $V_{ij}$ is a formal affinoid subdomain, its reduction $\tilde{V}_{ij}$ is an open affine subscheme of $\tilde{V}_i$. 
We may thus patch the $\tilde{V}_i$ along their intersections $\tilde{V}_{ij}$ to define a $\tilde{k}$-scheme of finite type $\fX\langle1\rangle$.
The canonical reduction maps on affinoid algebras also patch together, and defines an anti-continuous surjective
map $\red_1\colon X^{\an} \to \fX\langle1\rangle$.

The construction of $\fX\langle1\rangle$ and $\red_1$ do not depend on the choice of covering defining the formal structure induced by $\fX$. 

\begin{remark}
There is a canonical finite surjective morphism $\fX\langle1\rangle \to \fX_s$, which is an isomorphism
when $\fX_s$ is reduced, see~\cite[\S 2.11]{GRW} and~\cite[\S 4.1]{mehmeti} for a more detailed discussion of these facts. 
\end{remark}

 \subsubsection{}
We generalize the previous discussion to the case $\e$ is any positive real number less than $1$. For each $i$, endow $\prim(\cO(V_i)_\e)$ with its constructible topology. By Lemma~\ref{lem:sub-ideal}, 
the restriction map $\cO(V_i) \to \cO(V_{ij})$ induces a continuous and injective map
\[\prim(\cO(V_{ij})_\e)\to\prim(\cO(V_i)_\e).\]  We set $\fX[\e] $ to be the disjoint union of $\prim(\cO(V_i)_\e)$ modulo the relation that 
 two points $x_i \in \prim(\cO(V_i)_\e)$ and $x_j \in \prim(\cO(V_j)_\e)$ are identified whenever there exists
$x_{ij}\in\prim(\cO(V_{ij})_\e)$ whose images in $\prim(\cO(V_i)_\e)$ and  $\prim(\cO(V_j)_\e)$ are $x_i$ and $x_j$ respectively. 

We endow $\fX[\e]$ with the quotient topology which is the finest one making all maps $\prim(\cO(V_j)_\e)\to \fX[\e]$ continuous.
Note that patching finitely many compact sets always yields a compact space (the Hausdorff property can be easily seen using the Tietze-Urysohn's theorem).
It follows that $\fX[\e]$ is a compact topological space.

By Lemma~\ref{lem:primary} the $\e$-reduction maps $\red_\e \colon \cM(\cO(V_i)) \to \Id(\cO(V_j)_\e)$
constructed in \S\ref{sec:eps-red-aff} take its values in $\prim(\cO(V_j)_\e)$
and
patch together, and yield an anti-continuous map \[\red_\e \colon X^{\an} \to \fX[\e].\]

 \subsubsection{}
The $\e$-reduction is functorial in the following sense. If $\fX$ and $\fY$ are two models of $X$ and $Y$ respectively, 
and $\ff \colon \fX \to \fY$ is a morphism of schemes over $k^\circ$, then there is a unique
map $f_\e \colon \fX[\e]  \to \fY[\e]$ so that for any open affine subschemes $\tilde{V}\subset \fX_s$ and 
$\tilde{W}\subset \fY_s$ such that $\tilde{V} \subset \ff^{-1}(\tilde{W}) $, the map 
$f_\e |_{\prim(\cO(V)_\e)}$ is the canonical map $\prim(\cO(V)_\e) \to \prim(\cO(W)_\e)$.
Note that the following diagram is commutative:
\begin{equation}\label{eq:com-red-f}
\xymatrix{
X^{\an} \ar[r]^{f} \ar[d]_{\red_\e} & Y^{\an}  \ar[d]^{\red_\e}
  \\
 \fX[\e]   \ar[r]^{f_\e} & \fY[\e]
}\end{equation}
The map $f_\e$ is continuous for both the Zariski and the constructible topology.

It follows from the functoriality that $\fX[\e]$ does not depend on the choice of affine cover of $\fX_s$.
Note that for each affine subscheme $\tilde{U}$ of $\fX_s$, the natural map $\prim(\cO(U)_\e) \to \fX[\e]$ is a continuous injective map.

\begin{remark}\label{rem:i-schemes}
Pick any  ring $R$.  Then we may introduce the following natural topology $\tau_{z}$ on $\prim(R)$,
generated by sets of the form
\[\prim(R[f^{-1}])=\{I\in \prim(R), \text{ if } gf^n\in I \text{ for some } n\geq 0,  \text{ then } g\in I \},\]
where $f$ ranges over all elements in $R$. 
Note that the restriction of $\tau_z$ on $\spec R\subseteq \prim(R)$ is the Zariski topology. Endowing each set $\prim(\cO(V_j)_\e)$ with this topology induces
a topology $\tau_z$ on $\fX[\e]$.
The topological space $(\fX[\e],\tau_z)$ carries a canonical structure sheaf and naturally fits into the category
of ringed spaces that are locally isomorphic to $(\prim(R), \tau_z, R)$ for a ring $R$. 
We propose to call these objects $\imath$-schemes.
There is a natural equivalence of categories between schemes 
and  $\imath$-schemes. It would be interesting to investigate in further details the properties of this new category. 
\end{remark}

\begin{remark}\label{rem:technical}
We choose to patch the closure of the set of primary ideals $\prim(A_\e)$ in $A_\e$ instead of patching
directly $\Id(A_\e)$ because for the latter set Proposition~\ref{prop:cover-ideal} (hence Lemma~\ref{lem:sub-ideal}) is not valid.
\end{remark}

 \subsubsection{}\label{sec:barRe}
For any $\e>0$, denote by $\bar{R}_\e(\fX)$ the closure of the 
image by $\red_\e$ of $X^{\an}$. This is a compact subspace of $\fX[\e]$.

Suppose that $\e\in |k^*|$, and let $V_i$ be a finite cover by affinoid domains of $X^{\an}$
as before. 
Then \[\bar{R}_\e(\fX) = \cup_i \bar{R}_\e(\mathcal{O}(V_i))\subset  \cup_i \prim(\mathcal{O}(V_i)_\e)=\fX[\e].\] 
It follows from Proposition~\ref{prop:clos-primary}
that each set $\bar{R}_\e(\mathcal{O}(V_i))\cap \bar{R}_\e(\mathcal{O}(V_j))$ is a clopen set in $\bar{R}_\e(\mathcal{O}(V_i))$
so that $\bar{R}_\e(\mathcal{O}(V_i))$ forms a clopen cover of $\bar{R}_\e(\fX)$.

\begin{remark}\label{rem:technical2}
The space $\bar{R}_\e(\mathcal{O}(V_i))$ does not depend only on $A_\e$ which makes the descent argument in the next section more involved.
This explains why we prefered to work with $\fX[\e]$ instead of $\bar{R}_\e(\fX)$.
\end{remark}

\subsection{Proof of Theorem~\ref{thm:entropy good reduction}} \label{sec:proofB}

Recall the setting. Let $X$ be any projective variety over $k$, and $\fX$ a model of $X$ over $k^\circ$. 
Let $f\colon X \to X$ be a (regular) algebraic self-map whose extension $\ff\colon \fX\to \fX$ remains regular. 
We want to prove that the induced map  $f\colon X^{\an}\to X^{\an}$ has topological entropy $0$.

\begin{proposition}\label{prop:fine-cover}
For any finite open cover $\fU$ of $X^{\an}$, there exists $\e\in |k^*|$ and a finite clopen cover
$\fV$ of $\bar{R}_\e(\fX)$ such that $\red_\e^{-1}(\fV)$ refines $\fU$.
\end{proposition}

\begin{proof}
Pick any finite affine cover $V_j$ of $\fX_s$. For each $j$, 
$\{U_i \cap \red_\fX^{-1}(V_j)\}$ is a finite open cover of the affinoid domain
 $ \red_\fX^{-1}(V_j)$. By Lemma~\ref{lem:subcoverAra}, there exist $\e\in |k^*|$
 and a finite constructible cover
 $\fV_i$ of $(\Id(\cO(V_j)_\e),\tau^{\con})$ such that $\red_\e^{-1}(\fV_j)$ is a refinement of $\{U_i \cap \red_\fX^{-1}(V_j)\}$.

Since $\bar{R}_\e(\mathcal{O}(V_i))$ is a clopen subset of $\bar{R}_\e(\fX)$ (see \S\ref{sec:barRe}),  
$\fV:= \bigcup_i (\fV_i\cap \bar{R}_\e(\mathcal{O}(V_i)))$ forms a finite clopen cover 
of $\bar{R}_\e(\fX)$ such that  $\red_\fX^{-1}(\fV)$ refines $\fU$ as required.
 This concludes the proof.
\end{proof}
Let $\fU$ be a finite open cover of $X^{\an}$. By Proposition \ref{prop:fine-cover}, there is $\e\in |k^*|$ and a finite clopen cover
$\fV$ of $\bar{R}_\e(\fX)$ such that $\red_\e^{-1}(\fV)$ refines $\fU$.
As in Section \ref{sec:topo-def}, for each $n\in \N$, we write $\fU_n= \fU \vee \cdots \vee f^{-(n-1)}(\fU)$ and $\fV_n= \fV \vee \cdots \vee f_{\e}^{-(n-1)}(\fV)$.
Since every element in $\fV_n$ is open and closed,  a subset of $\fV_n$ covers $\red_{\e}(X^{\an})$ if and only if it covers $\bar{R}_\e(\fX)$.
Hence $$h(f,\fU)\leq h(f_{\e}|_{\bar{R}_\e(\fX)},\fV)\leq \htop(f_{\e}|_{\bar{R}_\e(\fX)}).$$
We conclude the proof by the following proposition.

\begin{proposition}\label{prop:eps-zero}
For any $1\ge\e>0$, the topological entropy of the map $f_\e \colon \fX[\e] \to \fX[\e] $
is $0$. In particular,  the topological entropy of the induced map $f_\e \colon \bar{R}_\e(\fX) \to \bar{R}_\e(\fX)$ is also $0$. 
\end{proposition}

The idea is to reduce the problem to a Noetherian situation where Theorem~\ref{thmentropyfinitetype} can be applied. 
To do so, we develop an argument of absolute Noetherian approximation analog to the one used in the proof of Theorem~\ref{thmentropyfinitetype}.

\subsubsection{Gluing data}\label{secgd}
Let $R$ be a commutative ring with unit.
A Gluing Data over $R$ (or (GD) for short) 
is the following collection of data (1) -- (7) subject to the conditions (a) -- (f) below: 
\begin{itemize}
\item[(1)] a family of $R$-algebras $A_i$ indexed by a finite set $I$;
\item[(2)] for every pair $i,j\in I$, $h_{i,j}\in A_i$ with $h_{i,i}=1$ for $i\in I;$
\item[(3)]for every pair $i,j\in I$,  isomorphisms $\alpha_{i,j}^*\colon A_{j,i}\to A_{i,j}$, with 
$A_{i,j}:=A_i[h_{i,j}^{-1}]$;
\item[(4)] for every $i,j,k\in I$, isomorphisms $\alpha_{i,j,k}^*\colon A_{j,i,k}\to A_{i,j,k}$, with 
 $A_{i,j,k}:=A_i[(h_{i,j}h_{i,k})^{-1}]$;
\item[(5)] for every $i\in I$, a finite subset $S_i$ of $A_i$, and a map $S_i \to A_i$, $s \mapsto h_s$;
\item[(6)] for every $i\in I$, a map $\sigma_i\colon S_i\to I$, and for every $s\in S_i$, a morphism
$\theta_{i,s}^*\colon A_{\sigma_i(s)}\to A_i[s^{-1}],$ $i\in I, s\in S_i$;
\item[(7)] for every $i\in I$, and $s_1,s_2\in S_i,$ a morphism $\theta^*_{i,s_1,s_2}\colon A_{\sigma_i(s_1),\sigma_i(s_2)}\to A_i[(s_1s_2)^{-1}]$.
\end{itemize}
We require these data to satisfy the following conditions:
\begin{itemize}
\item[(a)] $\alpha_{j,i}^*\circ \alpha_{i,j}^*=\id$ for every $i,j\in I;$
\item[(b)] for every $i,j,k\in I$, the following diagram is commutative
\[\xymatrix{
A_{j,i} \ar[r]^{\alpha_{i,j}^*} \ar[d] & A_{i,j}  \ar[d]
  \\
 A_{j,i,k}   \ar[r]^{\alpha_{i,j,k}^*} & A_{i,j,k}
}\]
where the vertical morphisms are the natural ones induced by localizations;
\item[(c)]  for every $i,j,k\in I$, $\alpha_{i,j,k}^*\circ \alpha_{j,k,i}^*=\alpha_{i,k,j}^*$;
\item[(d)]  for every $i\in I$, $\sum_{s\in S_i}sh_s=1\in A_{i}$;
\item[(e)] for every $i\in I$, and every $s_1,s_2\in S_i$, the  following diagram is commutative
\[\xymatrix{
A_{\sigma_i(s_1)} \ar[r]^{\theta_{i,s_1}^*} \ar[d] & A_{i}[s_1^{-1}]  \ar[d]
  \\
 A_{\sigma_i(s_1),\sigma_i(s_2)}   \ar[r]^{\theta_{i,s_1,s_2}^*} & A_{i}[(s_1s_2)^{-1}]
}\]
where the vertical morphisms are the natural ones induced by localizations;
\item[(f)]  for every $i\in I$, and every $s_1,s_2\in S_i$, $\theta^*_{i,s_2,s_1}\circ\alpha^*_{\sigma_i(s_2),\sigma_i(s_1)}=\theta^*_{i,s_1,s_2}$.
\end{itemize}

Assume that $\cG$ is a (GD)  over $R$. We explain how to define a compact topological space $X(\cG)$ , and a continuous self-map $f_\cG \colon X(\cG) \to X(\cG)$.

\medskip

First, we construct $X(\cG)$ as the disjoint union of the sets $\prim(A_i)$ modulo the identification
$x\in \prim(A_i) \sim y \in \prim(A_j)$ iff $h_{i,j} \notin x$, $h_{j,i} \notin y$ and $\alpha_{i,j}^*(x) =y$. Conditions (a), (b), (c) imply
this to be an equivalence relation. 
Note that we have natural maps $\phi_i\colon \prim(A_i) \to X(\cG)$ and that these maps are injective. Write $U_i:= \phi_i(\prim(A_i)) \subset X$. 
The topology on $X(\cG)$ is then defined by declaring that $U$ is open iff for all $i$, the set $\phi_i^{-1}(U\cap U_i)$ is open for the constructible
topology on $\prim(A_i)$. It follows that $U_i$ (which is compact and totally disconnected) forms a finite open cover of $X(\cG)$,  hence $X(\cG)$ is compact and totally disconnected. 

Note that
$\phi_i(\prim(A_{i,j}))=U_i\cap U_j$ and 
$\phi_j^{-1}\circ\phi_i|_{\prim(A_{i,j})}\colon \prim(A_{i,j})\to \prim(A_{j,i})$ is induced by $\alpha_{i,j}^*.$
Also $\phi_j^{-1}\circ\phi_i|_{\prim(A_{i,j,k})}:=\prim(A_{i,j,k})\to \prim(A_{j,i,k})$ is induced by $\alpha_{i,j,k}^*.$

\begin{remark}
The space $X(\cG)$ has a natural structure of $\imath$-scheme in the sense of 
Remark~\ref{rem:i-schemes}. It is also possible to show that $X(\cG)$ is defined by 
a suitable universal property, as in the case of schemes, see~\cite[Lemma~26.14.1]{stackproject}.
\end{remark}

Observe that for each $i\in I$, condition (d) implies that $\{\prim(A_i[s^{-1}])\}_{s\in S_i}$ forms an open cover of $\Id(A_i).$ 
For any $i\in I$, and any $s\in S_i$, set $U_{i,s}:=\phi_{i}(\prim(A_{i}[s^{-1}]))\subset X(\cG)$.
By conditions (e) and (f), there is a unique continuous map $f_{\cG}\colon X(\cG)\to X(\cG)$ such that:
\begin{itemize}
\item
for any $i\in I$ and any $s\in S_i$ we have
$f_{\cG}(U_{i,s})\subseteq U_{\sigma_i(s)}$;
\item
 and the morphism 
\[\phi_{\sigma_i(s)}^{-1}\circ f_{\cG}\circ \phi_i|_{\prim(A_i[s^{-1}])}\colon \prim(A_i[s^{-1}])\to \prim(A_{\sigma_i(s)})\]
is induced by $\theta_{i,s}^*.$
\end{itemize}

\begin{remark}
As above, the map $f_\cG$ is a morphism of $\imath$-schemes, and it can be characterized by a suitable universal property.
\end{remark}

\begin{definition}
We say that a (GD) over $R$  is finitely generated if the $R$-algebras $A_i$ are finitely generated over $R$ for all $i\in I$.
\end{definition}

\begin{remark}If $R'$ is a subring of $R$, then any (GD) over $R$ is a (GD) over $R'$.
\end{remark}

\subsubsection{Gluing data for $(\fX[\e], f_{\e})$}\label{subsubsecgddynpair}
Fix any $0<\e\le1$. Our objective in this section is to construct  a (GD)  $\cG_{\e}$ 
that is finitely generated over $k_{\e}$, and such that 
$X(\cG_\e)$ is homeomorphic to $\fX[\e]$ and $f_{\cG_\e}$ corresponds to $f_{\e}$ under that homeomorphism.

\medskip

We first collect data from the construction of $\fX[\e]$ done in \S \ref{sec:func} . 
Since $\fX_s$ is projective, we may choose any very ample line bundle $L_s\to \fX_s$, pick 
a (finite) basis $\{\sigma_i\}$ of $H^0(\fX_s, L_s)$, and set  $\tilde{V}_i:=\{\sigma_i\neq 0\}$.
We obtain a finite affine open cover $\{\tilde{V}_i\}_{i\in I}$ of $\fX_s$.
Observe that $\tilde{h}_{i,j} := \sigma_j/\sigma_i \in \cO_{\fX_{s}}(\tilde{V}_i)$, and
 $\tilde{V}_i\cap \tilde{V}_j= \{\tilde{h}_{i,j}\neq 0\}$ in $\tilde{V}_i$. Recall that  $V_i:=\red_\fX^{-1}(\tilde{V}_i)$ forms a cover of $X^{\an}$ by affinoid domains
(see Lemma~\ref{lem:key-formal}) so that $\cO(V_i)$ is an affinoid algebra, and we have a canonical finite morphism
$\cO_{\fX_{s}}(\tilde{V}_i) \to \widetilde{\cO(V_i)} $.

\smallskip

For each $i$, we set $A_i := \cO(V_i)_\e$. By construction, $U_i:= \prim(A_i)$ forms an open cover of $\fX[\e]$. 
Note that $A_i$ is a finitely generated $k_{\e}$-algebra by Lemma \ref{lem:fg}.

\smallskip

 Since the reduction map $\cO(V_i)^{\circ}\to \widetilde{\cO(V_i)}$ is surjective,
 we may choose a family of lifts $\hat{h}_{i,j}\in \cO(V_i)^\circ$ such that $\hat{h}_{i,i} =1$ for all $i$, 
 and the image of $\hat{h}_{i,j}$ equals $\tilde{h}_{i,j}$ in  $\cO_{\fX_{s}}(\tilde{V}_i)$.
 Let $h_{i,j}$ be the $\e$-reduction of $\hat{h}_{i,j}$ in $A_i$. Write $A_{i,j} = A_i[h_{i,j}^{-1}]$. 

\smallskip 

 The equality $V_i \cap \{h_{i,j} \neq 0\}= V_j \cap \{h_{j,i} \neq 0\}$ in $\fX_s$ gives
an isomorphism at the level of affinoid algebras $\cO(V_i)\langle \hat{h}_{i,j}^{-1}\rangle \to \cO(V_j)\langle \hat{h}_{j,i}^{-1}\rangle$, and
 induces a ring isomorphism $\alpha_{i,j}^*\colon A_{j,i}\to A_{i,j}$.
 In the same way,  the equality $V_i \cap \{h_{i,j}h_{i,k} \neq 0\}= V_j \cap \{h_{j,i}h_{j,k}  \neq 0\}$ in $\fX_s$ gives
 isomorphisms $\alpha_{i,j,k}^*\colon A_{j,i,k}\to A_{i,j,k}$
 with  $A_{i,j,k}:=A_{i}[(h_{i,j}h_{i,k})^{-1}]$. The first set of conditions (a), (b) and (c) are easily satisfied.

\smallskip

Now we collect the data defining the map $f_\e \colon \fX[\e] \to \fX[\e]$.
Recall that we are given morphisms  $f\colon X^{\an} \to X^{\an}$, and $\ff_s\colon \fX_s \to \fX_s$ such that the diagram~\eqref{eq:com-red-f} commutes.

It follows that for each $i\in I$, there is a finite subset $\tilde{S}_i$ of $\cO_{\fX_s}(\tilde{V}_i)$, and 
a map $\tilde{\sigma}_i\colon \tilde{S}_i\to I$ such that:
\begin{itemize}
\item
the family $\{\tilde{s}\neq0\}_{\tilde{s}\in \tilde{S}_i}$ forms an open cover of $\tilde{V}_i$;
\item
$\ff_s(\tilde{V}_i\setminus\{\tilde{s}=0\})\subseteq \tilde{V}_{\tilde{\sigma}(\tilde{s})}$ for every $\tilde{s}\in \tilde{S}_i.$
\end{itemize}
Now lift each $\tilde{s}\in \tilde{S}_i$ to some $\hat{s}$ in $\cO(V_i)^{\circ},$ and write $\hat{S}_i=\{ \hat{s}\}$. 
Since the ideal generated by $\tilde{S}_i$ in $\cO_{\fX_s}(\tilde{V}_i)$ is the unit ideal, $\hat{S}_i$ also generates the unit ideal in $\cO(V_i)^{\circ}$.
Fix a set $\hat{h}_{\hat{s}}\in\cO(V_i)^{\circ}$ such that $\sum \hat{s}\hat{h}_{\hat{s}}=1$.
Let $s$ (resp. $h_s$) be the image of $\hat{s}$ (resp. of $\hat{h}_{\hat{s}}$) in $A_i$. 
Write $S_i = \{ s\} \subset A_i$, and observe that $\sum_{s\in S_i} s h_s = 1$ so that condition (d) holds. 
The map $\tilde{\sigma}_i$ descends to a map $\sigma_i\colon S_i \to I$. 

Since $f(\red_\fX^{-1}(\tilde{V}_i\setminus\{\tilde{s}=0\}))\subseteq \red_\fX^{-1}(\tilde{V}_{\tilde{\sigma}(\tilde{s})})$
 for every $\tilde{s} \in \tilde{S}_i$, the $\e$-reduction of this morphism induces a morphism $\theta_{i,s}^*\colon A_{\sigma_i(s)}\to A_i[s^{-1}]$, and
 conditions (e) and (f) are easily satisfied.
  
It follows from the very definitions that $\fX[\e]$ is homeomorphic to $X(\cG_\e)$ and that $f_\e$ is conjugate to 
$f_{\cG_\e}$ under that homeomorphism.

\subsubsection{Descent argument}\label{subsubsecda}
Let us fix a (GD) $\cG$ over some ring $R$ as in Section \ref{secgd} (so that it is determined by a finite set $I$, 
a collection of rings $\{A_i\}_{i\in I}$, elements $h_{i,j}\in A_i$, etc.).
We say that $\cG'$ is a sub-(GD) and we write $\cG' \subseteq \cG$ if the following conditions hold:
\begin{itemize}
\item
$\cG'$ is determined by a family of sub-rings $B_i \subset A_i$ indexed by the set $I$;
\item
$B_i$ contains $h_{i,j}$ for all $j$, the set $S_i$, and $h_s$ for all $s\in S_i$;
\item 
the set $S_i(\cG')$ is equal to $S_i$, the map $\sigma_i(\cG')$ is equal to $\sigma_i$, and the maps
$S_i \to B_i$ are induced by the maps $S_i \to A_i$ (i.e. the image of $S_i \to A_i$ belongs to $B_i$);
\item
all morphisms $\alpha_{i,j}^*(\cG')$,  $\alpha_{i,j,k}^*(\cG')$, $\alpha^*_{i,s}(\cG')$, $\alpha^*_{i,s_1,s_2}(\cG')$ are induced
by the corresponding ones in $\cG$, that is, we impose:
\begin{align*}
\alpha_{i,j}^*(B_{j,i})\subseteq B_{i,j}; \,\,
&\alpha_{i,j,k}^*(B_{j,i,k})\subseteq B_{i,j,k}
\\
\alpha_{i,s}^*(B_{\sigma_i(s)})\subseteq B_i[s^{-1}]; \,\,
&
\alpha^*_{i,s_1,s_2}(B_{\sigma_i(s_1),\sigma_i(s_2)})\subseteq B_i[(s_1s_2)^{-1}]~.
\end{align*}
\end{itemize}
When $\cG'$ is a sub-(GD) of $\cG$, the inclusion morphisms $B_i \subset A_i$ induce a canonical continuous map
$\pi^{\cG'/\cG} \colon X(\cG) \to X(\cG')$ such that 
$\pi^{\cG/\cG'}\circ f_{\cG}=f_{\cG'}\circ \pi^{\cG/\cG'}$. 

\begin{lemma}\label{lem:descent}
Suppose $\cG$ is a (GD) that is finitely generated over $R$. 
Then there exists an inductive set $\cN$, a family of Noetherian rings $\{R_t\}_{t\in \cN}$
and a family of sub-(GD) $\cG_t\subset \cG$, each defined over $R_t$ such that:
\begin{enumerate}
\item
$\cG_t \subset \cG_{t'}\subset \cG$ for any $t\le t'$;
\item
$\cG_t$ is finitely generated over $R_t$ for all $t\in \cN$;
\item  
the canonical map $X(\cG)\to \projlim_\cN X(\cG_t) $ is a homeomorphism.
\end{enumerate}
\end{lemma}

\begin{proof}
The set $\cN$ will be defined as the poset of all subrings $R'$ of $R$ ordered by inclusion subject to the conditions that $R'$ is finitely generated (over $\mathbf{Z}$) and contains a finite set $O(\cG)\subset R$ that we define now. 

For each $i$, let us fix $T_i=\{t_{i,\cdot}\}$ a finite set of generators of $A_i$ as an $R$-algebra that contains $S_i$, $h_{i,j}$ for all $j$, and
$h_s$ for all $s\in S_i$. We express all morphisms $\alpha_{i,j}, \alpha_{i,j,k}, \theta_{i,s}$, and $\theta_{i,s_1,s_2}$ in these basis.
For instance, we write
\begin{equation}\label{eq:def-alpha}
\alpha_{i,j}(t_{j,\beta}) = \sum_{n\in \Z} \sum_{\alpha} r_\alpha(n,j,\beta) t_i^\alpha h_{i,j}^n
\end{equation}
where $r(n,j,\beta) \in R$, and the sum is finite. In this way, we collect a finite set of elements 
of $R$, and denote by $O(\cG)$ its union. 

Now take any finitely generated subring $R'$ of $R$ containing $O(\cG)$. We let $\cG'$ be the (GD) over $R'$
determined by the collection of subrings $A'_i$ of $A_i$ generated by $T_i$ as an $R'$ algebra. 

 Define $\alpha_{i,j}'\colon A_j'\to A_i$ as the restriction of  the morphism $\alpha_{i,j}\colon A_j\to A_i$ to $A_j'$.
By~\eqref{eq:def-alpha}, $\alpha_{i,j}'(A_j')$ is contained in $A_i'$, so that we get a morphism 
$\alpha_{i,j}'\colon A_j'\to A_i'$. 
We proceed in the same way to construct $\alpha'_{i,j,k}, \theta'_{i,s}$, and $\theta'_{i,s_1,s_2}$. 

It is clear that $\cG'$ is a sub-(GD) of $\cG$, that $\cG' \subset \cG''$ whenever $R' \subset R''$
and $\cG'$ is finitely generated over $R'$ by construction. 

Finally the injective limit of all finitely subrings $R'$ of $R$ containing  $O(\cG)$ is equal to $R$, 
hence for any $i$,  $\Id(A_i)$ is homeomorphic to $\projlim \Id(A'_i)$ by Lemma~\ref{lemdirectlimring}
which implies (3).
\end{proof}
 
\proof[Proof of Proposition~\ref{prop:eps-zero}]
Let $\cG_{\e}$ be the (GD) associated to $(\fX[\e], f_{\e})$ as defined in Section \ref{subsubsecgddynpair}
so that $X(\cG_\e) = \fX[\e]$ and $f_\e = f_{\cG_\e}$.

Pick any inductive family of sub-(GD) $\cG_t\subseteq \cG_\e$ as in Lemma~\ref{lem:descent}. 
For each $t\in \cN$, $\cG_t$ is finitely generated over a Noetherian ring $R_t$. 
The space $X(\cG_t)$ is a finite union of compact subsets of Noetherian Priestley spaces of the form $\Id(A)$
for some Noetherian ring $A$, hence any Radon measure on  $X(\cG_t)$ is atomic by Theorem~\ref{thm:priestley-entropy}.
Since  $X(\cG_t)$ is compact, the variationnal principle (Theorem~\ref{thm:variational}) implies that $\htop(f_{\cG_t})=0$
for all $t$. 
As $X(\cG)$ is homeomorphic to the projective limit $\projlim_\cN X(\cG_t)$ and the following diagrams commute
\[\xymatrix{
X(\cG_\e) \ar[r]^{f_{\cG_\e}} \ar[d]_{\pi^{\cG_\e/\cG_t}} & X(\cG_\e) \ar[d]^{\pi^{\cG_\e/\cG_t}}
  \\
X(\cG_t)   \ar[r]^{f_{\cG_t}} & X(\cG_t)
}
\]
Lemma~\ref{leminverselimitentropy} shows $\htop(f_{\e})=\htop(f_{\cG_\e}) \le \limsup_t \htop(f_{\cG_t}) = 0$ as required.
\endproof

%%%%%%%%%%%%%%%%%%%%%%%%%%%%%

\appendix
\section{Topological entropy and field extensions} \label{sec:appennd}

Let $k$ be any complete metrized field, and $X$ a projective variety defined over $k$. 
Let $f \colon X \dashrightarrow X$ be any dominant rational map. 
Recall that $\htop(f)$ is defined in \S \ref{sec:definition-topo-rat}, and coincides with the topological entropy of
the action of $f$ on the Berkovich analytification of $X$ when $f$ is a regular map. 

We propose the following problem.

\begin{conjecture}
Let $K/k$ be any complete metrized field extension, and let $f_K\colon X_K \dashrightarrow X_K$ be the rational self-map induced by base change. 
Then we have $\htop(f_K) = \htop(f)$.
\end{conjecture}

Let us prove the following easy bound (thereby proving Proposition~\ref{probasechangeentropy}). 

\begin{proposition}\label{probasechangeentropy-appendix}
Let $K/k$ be any complete metrized field extension, and let $f_K\colon X_K \dashrightarrow X_K$ be the rational self-map induced by base change. 
Then $\htop(f_K) \ge \htop(f)$.
\end{proposition}

\begin{proof}
Base change induces a canonical continuous map $\pi_{K/k}\colon X^{\an}_K \to X^{\an}_k$. 
This map is surjective
since the scheme morphism $X_K \to X_k$ is surjective, and any norm on a field $L$ admits at least one extension 
to any of its field extension $L'/L$, see~\cite[\S VI.6]{zariski}.\footnote{This fact also follows from Lemma~\ref{lemsurjectivebasechange}}.
Observe now that $\Ind(f_K) = \Ind(f)_K$.
We conclude the proof using Corollary \ref{corsujsysentropy}.
\end{proof}

We now prove a partial converse to this proposition.

\begin{proposition}\label{prop:conj-entropy}
Let $\widehat{k^{\alg}}$ be the completion of an algebraic closure of $k$. 
For any regular map $f\colon X \to X$, we have
$\htop(f_{\widehat{k^{\alg}}}) = \htop(f)$.
\end{proposition} 

Let $p$ be the characteristic of $k$.
\begin{proof}

Let us consider the separable closure $k^{\sep}\subset k^{\alg}$ which consists of the union of all separable finite extensions of $k$. 
Note that the algebraic extension $k^{\alg}/k^{\sep}$ is purely inseparable (i.e., for any $a\in k^{\alg}$ we have $a^{p^n} \in k^{\sep}$ for some $n$), 
and  $k^{\sep}/k$ is an algebraic Galois extension. 
It follows that the group $G:= \gal(k^{\sep}/k)$ is the projective limit of $\gal(K/k)$ where $K$ ranges over the set $\mathfrak{G}$ of all finite Galois subextensions of $k$ inside $k^{\alg}$, see e.g.~\cite[VI.\S 14]{lang-algebra}.

Recall the norm on $k$ extends in a unique way to any finite extension as a multiplicative norm, 
so that any element of  $G:=\gal(k^{\sep}/k)$ induces an isometry on $k^{\sep}$. In particular
$G$ acts isometrically on its completion $\widehat{k^{\sep}}$.

By~\cite[Proposition~1.3.5]{berkovich} the analytic space  $X_k^{\an}$ is homeomorphic
to the quotient of $X_{\widehat{k^{\sep}}}^{\an}$ by the action of $G$, and 
the base change morphism 
$X_{\widehat{k^{\alg}}}^{\an}\to X_{\widehat{k^{\sep}}}^{\an}$ is a homeomorphism,
so that 
\[X_{\widehat{k^{\alg}}}^{\an} \simeq X_{\widehat{k^{\sep}}}^{\an}\simeq \projlim_{K\in\mathfrak{G}} X_K^{\an}~.\]
Since $f$ lifts to a continuous map to  $X_K^{\an}$ for any complete field extension $K/k$, we conclude that
$\htop(f_{\widehat{k^{\alg}}}) = \sup_{K \in \mathfrak{G}} \htop(f_K)$ by Lemma \ref{leminverselimitentropy}.

When $K$ is a finite extension of $k$, then the map $X^{\an}_K \to X^{\an}_k$ has finite fibers. 
It follows from the  Abramov-Rokhlin's formula (see e.g.,~\cite[Proposition 2.15]{adler-marcus}) that the measure entropy is preserved under finite extension. 
We conclude by the variational principle that $\htop(f) = \htop(f_K)$, and therefore $\htop(f_{\widehat{k^{\alg}}}) =\htop(f)$. 
\end{proof}

We conclude this Appendix with the following solution to the conjecture in the case 
$\dim(X)=1$. 
\begin{proposition}\label{prop:end1}
Let $f\colon X \to X$ be any rational self-map of a smooth projective curve $X$ defined over $k$. 
Then for any complete metrized field extension  $K/k$, we have 
 $\htop(f_K) = \htop(f)$.
\end{proposition}

We rely on the next lemma  whose proof is left to the reader (see~\cite[Th\'eor\`eme~C]{FR10}).
\begin{lemma}\label{lem:finally}
Suppose $f \colon \D(0,1) \to \D(0,1)$ is an analytic map of the open unit disk. 
Then any $f$-invariant probability measure has zero metric entropy.
\end{lemma}

\begin{proof}[Proof of Proposition~\ref{prop:end1}]
We may assume that the genus of $X$ is either $0$ or $1$ since otherwise $f$ has finite order. 

By Proposition~\ref{probasechangeentropy-appendix}, it is sufficient to prove  $\htop(f_K) \le \htop(f)$.
By Proposition~\ref{prop:conj-entropy}, we may assume that $K$ and $k$ are both algebraically closed. 
By~\cite[Corollaires 3.7 et 3.14]{poineau13}, there exists a canonical injective continuous map $\sigma_{K/k} \colon  X \to X_K$
such that  $\pi_{K/k}\circ   \sigma_{K/k} = \id_{\p^1_k}$. Since this lift is canonical, we have
$\sigma_{K/k} \circ f = f_K \circ \sigma_{K/k}$ so that $f_K|_{\sigma_{K/k} (X_k)} = \htop(f)$.

Suppose first that $X$ has genus $1$. By the semi-stable reduction theorem~\cite[Theorem~4.3.1]{berkovich}, 
either $X$ admits a smooth model $\fX$ over $k^\circ$, or 
$X$ retracts by deformation on a unique topological circle $\Sk(X) \subset X$.

In the former case, the unique point $x$ corresponding to the special fiber of $\fX$ is totally invariant by 
$f$ and $X\setminus \{x\}$ is a union of open balls. It follows from Lemma~\ref{lem:finally} that
$\htop(f) =0$ By base change, $X_K$ also admits a smooth model over $K^\circ$, hence 
$\htop(f_K)=0$.

In the latter case, the complement $X \setminus \Sk(X)$ is a union of open balls, hence $\htop(f) = \htop(f|_{\Sk(X)})$ by Lemma~\ref{lem:finally}.
Now note that $\sigma_{K/k}(\Sk(X))$ is then a topological circle, hence by uniqueness we must have $\Sk(X_K) = \sigma_{K/k}(\Sk(X))$.
We conclude that 
\[\htop(f_K) = \htop(f_K|_{\sigma_{K/k}(\Sk(X))})=  \htop(f|_{\Sk(X)}) =\htop(f)~.\]
Now suppose that $X= \p^1_k$.  In that case, $\p^1_K \setminus \sigma_{K/k} (\p^1_k)$ is a union of open balls, 
hence again by Lemma~\ref{lem:finally} we have 
 \[\htop(f_K) =  \htop(f_K|_{\sigma_{K/k} (\p^1_k)} )=\htop(f)\] 
as required.
\end{proof}

%%%%%%%%%%%%%%%%%%%%%%%%%%%%%

\end{document}